\documentclass[11pt]{amsart}
\usepackage{amsfonts}
\usepackage{amssymb}
\usepackage{mathrsfs}
\usepackage{srcltx}
\usepackage{subfigure}
\usepackage{amsmath}
\allowdisplaybreaks[4]

\usepackage{amsmath,amsthm,amstext,amsopn,graphicx,cite,xcolor,enumitem}
\newtheorem{theorem}{Theorem}[section]
\newtheorem{lemma}[theorem]{Lemma}
\newtheorem{proposition}[theorem]{Proposition}
\newtheorem{corollary}[theorem]{Corollary}
\newtheorem{remark}[theorem]{Remark}

\theoremstyle{definition}

\theoremstyle{remark}

\numberwithin{equation}{section}
\newcommand\bes{\begin{eqnarray}}
\newcommand\ees{\end{eqnarray}}
\newcommand\bess{\begin{eqnarray*}}
	\newcommand\eess{\end{eqnarray*}}
\newcommand{\lf}{\left}
\newcommand{\rr}{\right}
\newcommand{\R}{{\mathbb R}}
\newcommand{\dd}{\displaystyle}

\newcommand{\td}{\tilde}
\newcommand{\wtd}{\widetilde}
\newcommand\yy{\infty}

\newcommand{\ol}{\overline}
\newcommand{\rd}{{\rm d}}

\begin{document}

 \pagestyle{myheadings}

\date{\today}
\title[Nonlocal diffusion equation with free boundary and radial symmetry]{The high dimensional Fisher-KPP nonlocal diffusion equation with free boundary and radial symmetry$^{\S}$}

\author[Y. Du and W. Ni]{Yihong Du\textsuperscript{$\dag$} and \ \ Wenjie Ni\textsuperscript{$\dag$}}
	
\thanks{\textsuperscript{$\dag$}School of Science and Technology, University of New England, Armidale, NSW 2351, Australia.}
\thanks{$^\S$Research of both authors was  supported by the Australian Research Council.}

\maketitle

\begin{abstract} We study the radially symmetric  high dimensional Fisher-KPP nonlocal diffusion equation with free boundary, and reveal some fundamental differences from its  one dimensional version
 considered in \cite{cdjfa} recently.
Technically, this high dimensional problem is much more difficult to treat since it involves two kernel functions which arise from the original kernel function $J(|x|)$ in  rather implicit ways. 
By introducing new techniques, we are able to determine the long-time dynamics of the model, including
firstly finding the threshold condition on the kernel function that governs the onset of accelerated spreading, and the determination of the spreading speed when it is finite. Moreover, for two important classes of kernel functions, sharp estimates of the spreading profile are obtained. More precisely, for kernel functions with compact support, we show that logarithmic shifting occurs from the finite wave speed propagation, which is strikingly different from the one dimension case;  for kernel functions $J(|x|)$ behaving like $|x|^{-\beta}$ for  $x\in\R^N$ near infinity, we obtain the rate of accelerated spreading when $\beta\in (N, N+1]$, which is the exact range of $\beta$ where accelerated spreading is possible. These sharp estimates are obtained by constructing subtle upper and lower solutions, based on careful analysis of the involved kernel functions.

\bigskip

\noindent
{\bf Key words:} {\sf Nonlocal diffusion, free boundary, spreading speed, accelerated spreading.}
\smallskip

\noindent
{\bf AMS subject classification:} 35K20, 35R35, 35R09.
\end{abstract}

\tableofcontents

\section {Introduction}
 In \cite{cdjfa}, the authors studied the following one dimensional nonlocal diffusion problem with free boundaries 
\begin{equation}\label{f1}
\begin{cases}
\dd u_t=d \int_{g(t)}^{h(t)}J(x-y)u(t,y)\rd y-du(t,x)+f(t,x, u), & t>0,\; x\in  (g(t), h(t)),\\
u(t,g(t))= u(t,h(t))=0, & t>0,\\
\dd g'(t)= -\mu \int_{g(t)}^{h(t)}\int_{-\yy}^{g(t)}J(x-y)u(t,x) dyd x, & t>0,\\[3mm]
\dd h'(t)= \mu \int_{g(t)}^{h(t)}\int_{h(t)}^{\yy}J(x-y)u(t,x) dyd x, & t>0,\\
g (0)=-h_0,\; h(0)=h_0,\; u(0,x)=u_0(x), &x\in [-h_0,h_0],
\end{cases}
\end{equation}
where $d, \mu, h_0$ are given positive constants, $ f$ is a smooth function satisfying $f(t,x,0)\equiv 0$, and the initial function $u_0(x)$ is continuous
and $u_0(x)>0$ in $(-h_0, h_0)$, $u_0(\pm h_0)=0$. When $f\equiv 0$, \eqref{f1} was studied in \cite{CQW}.

In \eqref{f1} the basic assumptions on the kernel function $J$ are
\begin{equation}\label{J0}
\mbox{ $J\in C(\R)\cap L^\infty(\R)$, is nonnegative, even, $J(0)>0$ and } \int_{\R}J(x)dx=1.
\end{equation}

Under suitable additional assumptions on $f$, it was shown in \cite{cdjfa} that \eqref{f1} has a unique solution $(u,g,h)$ defined for all $t>0$. Moreover, if $f$ is of Fisher-KPP type (see {\bf (f)} below for details), then the long-time dynamical behaviour of \eqref{f1} is characterised by a ``spreading-vanishing dichotomy":

As $t\to\infty$, either $[g(t), h(t)]$ converges to a finite interval $[g_\infty, h_\infty]$ and $u(t,x)$ converges to 0 uniformly (the \underline{vanishing} case), or $[g(t), h(t)]$ converges to $\R$ and $u(t,x)\to u^*$ which is the unique positive zero of $f(u)$ (the \underline{spreading} case).
This resembles the behaviour of the corresponding local diffusion model of \cite{DL}.

When spreading happens, the spreading speed of \eqref{f1} was determined in \cite{dlz2019, dn-speed}, which reveals significant differences from the local diffusion model of \cite{DL}; namely, depending on the behaviour of the kernel function $J(x)$, accelerated spreading may happen to \eqref{f1}. More precisely, if $J(x)$ satisfies additionally
\begin{equation}\label{J1}
\int_0^\infty xJ(x)dx<\infty,
\end{equation}
then the spreading has a finite speed: $\lim_{t\to\infty} h(t)/t=-\lim_{t\to\infty} g(t)/t=c_0$ for some $c_0>0$ uniquely determined by the so called semi-wave solution of \eqref{f1}; if \eqref{J1} is not satisfied, then accelerated spreading happens: $\lim_{t\to\infty} h(t)/t=-\lim_{t\to\infty} g(t)/t=\infty$.

When $J(x)\approx |x|^{-\gamma}$ near $\infty$ for some $\gamma>0$, namely $ c_1 |x|^{-\gamma}\leq J(x)\leq c_2 |x|^{-\gamma}$ for some positive constants $c_1, c_2$ and all large $|x|$,  it is easily seen that \eqref{J0} implies $\gamma>1$, and \eqref{J1} is equivalent to $\gamma>2$. The results in \cite{dn-speed} applied to \eqref{f1} then give the following conclusions:
\[\left\{\begin{array}{rll}
c_0t+g(t), c_0t-h(t)&\approx \ 1 & \mbox{ if } \gamma>3,\\
c_0t+g(t), c_0t-h(t)&\approx \ \ln t & \mbox{ if } \gamma=3,\\
c_0t+g(t), c_0t-h(t)&\approx \  t^{3-\gamma} & \mbox{ if } \gamma\in (2, 3),\\
-g(t),\ h(t) &\approx \ t\ln t & \mbox{ if } \gamma=2,\\
-g(t),\ h(t) &\approx\  t^{1/(\gamma-1)} & \mbox{ if } \gamma\in (1,2).
\end{array}\right.
\]

\bigskip

In this paper, we consider the high dimensional version of \eqref{f1} with radial symmetry. Here, ``with radial symmetry" means that the kernel function $J(x)$,
the initial function $u_0(x)$ and the nonlinear term $f(t,x,u)$ are all radially symmetric in $x\in\R^N$, $N\geq 2$. The population range $\Omega(t)$ then is a ball of radius $h(t)$, namely $\Omega(t)=B_{h(t)}:=\{x\in \R^N: |x|<h(t) \}$, with $h(t)$ an unknown function to be determined with the population density function $u(t,x)$, which is radially symmetric in $x$ too. For convenience, we will write $J=J(|x|)$, $u=u(t, |x|)$, etc.
 Then the radially symmetric version of \eqref{f1} in $\R^N$ is given by
 \begin{equation}\label{1.3}
\begin{cases}
\dd u_t=d \int_{B_{h(t)}}J(|x-y|)u(t,|y|)\rd y-du(t,|x|)+f(t, |x|,u), & t>0,\; x\in  B_{h(t)},\\
u(t,|x|)=0, & t>0,\; x\in \partial B_{h(t)},\\
h'(t)= \dd  \frac{ \mu}{|\partial B_{h(t)}|} \int_{B_{h(t)}} \int_{\R^N\backslash B_{h(t)}}J(|x-y|)u(t,|x|)\rd y\rd x, & t>0,\\
h(0)=h_0,\ u(0,|x|)=u_0(|x|), &x\in \ol B_{h_0}.
\end{cases}
\end{equation}

It is easy to check that $ \int_{B_{h(t)}}J(|x-y|)u(t,|y|)\rd y$ depends only on $|x|$. To see how the equation for $h'(t)$ in \eqref{1.3} is obtained, let us recall
that, in \eqref{f1}, the free boundary equations are obtained from the assumption that the expansion of the population range $[g(t), h(t)]$ is at a rate proportional to the outward flux of the population at the boundary of $[g(t), h(t)]$. In the current setting, the range boundary is the sphere $\partial B_{h(t)}$, 
and from the nonlocal dispersal rule governed by the  kernel function $J(|x-y|)$, 
the total population mass, at time $t$, moved out of $B_{h(t)}$ through $\partial B_{h(t)}$  per unit time is 
\begin{align}\label{1.2}
M^*(t)=\int_{B_{h(t)}}\int_{\R^N\backslash B_{h(t)}}J(|x-y|) u(t,|x|)\rd y\rd x.
\end{align}
Therefore the expansion rule of $B_{h(t)}$ gives
\begin{align*}
\dd h'(t)=\mu \frac{M^*(t)}{|\partial B_{h(t)}|}=\frac{\mu}{|\partial B_{h(t)}|}  \int_{B_{h(t)}}\int_{\R^N\backslash B_{h(t)}}J(|x-y|)u(t,|x|)\rd y\rd x.
\end{align*} 

For \eqref{1.3}, our basic assumptions on   the kernel function $J(|x|)$ are
\begin{itemize}
	\item[\textbf{(J):}] \ \ \   $J\in C(\R_+)\cap L^\infty(\R_+)$ is nonnegative,  $J(0)>0$,   $\displaystyle\int_{\R^N} J(|x|) \rd x=1$.
\end{itemize}
Here and throughout the paper, $\R_+$ denotes $[0,\infty)$.

  For $r:=|x|$ with $x\in \R^N$ and $\rho>0$,   denote 
\begin{align*}
\td J(r,\rho)=\td J(|x|,\rho):=\int_{\partial B_\rho} J(|x-y|)\rd S_y.
\end{align*}
Then \eqref{1.3} can be rewritten into the equivalent form 
\begin{equation}\label{1.4}
\begin{cases}
\dd  u_t(t,r)=d \int_{0}^{h(t)} \td J(r,\rho)  u(t,\rho)\rd \rho-d u(t,\rho)+f(t,r,  u), & t>0,\; r\in  [0,h(t)),\\
 u(t,h(t))=0, & t>0,\\
\dd h'(t)=  \frac{\mu}{h^{N-1}(t)}\dd\int_{0}^{h(t)} \int_{h(t)}^{+\yy} \td J(r,\rho)r^{N-1} u(t,r)\rd \rho\rd r, & t>0,\\
h(0)=h_0,\ u(0,r)= u_0(r), & r\in [0,h_0].
\end{cases}
\end{equation}
(Here a universal constant is absorbed by $\mu$.)

We  require  the initial function $u_0$ to satisfy
\begin{equation}\label{u_0}
		u_0\in C(\ol B_{h_0})\ {\rm is\ radially\ symmetric},\ u_0=0\ {\rm on  }\ \partial B_{h_0}\ {\rm and}\	u_{0}>0\ {\rm in }\   B_{h_0}.
\end{equation}

The function $f(t,r, u)$ is assumed to satisfy
\begin{equation}\label{f}
\left\{
\begin{array}{l}
f \mbox{ is continuous}, \ f(t,r,0)\equiv 0,\  f(t,r,u) \mbox{ is locally Lipschitz  in $u\in \R_+$}\\
\mbox{ uniformly for $(t,r)\in \R_+\times\R_+$, and there exists $K_0>0$ such that}\\ \mbox{ $f(t,r,u)\leq 0$ for $u\geq K_0$ and all $t, r\geq 0$.}\end{array}
\right.
\end{equation}

\begin{theorem}[\underline{Existence and uniqueness}]\label{th1.1}
Suppose   {\rm \textbf{(J)}}, \eqref{u_0}  and \eqref{f} are satisfied. Then problem  \eqref{1.3}, or equivalently \eqref{1.4}, admits a unique positive solution $(u,h)$  defined for all $t>0$. 
\end{theorem}

To study the long-time dynamical behaviour of \eqref{1.3}, we only consider Fisher-KPP type of $f$, namely $f=f(u)$ that satisfies
\begin{itemize}
	\item[$\mathbf{(f):}$] $\left\{
	\begin{array}{ll} \mbox{ $f$ is $C^1$, $f(0)=0<f'(0)$, there exists $u^*>0$ such that}\\
	\mbox{$f(u^*)=0>f'(u^*)$  and $(u^*-u)f(u)>0$ for $u\in (0,\infty)\setminus\{u^*\}$,}  \\
		 \mbox{  $\sigma f(u)\leq f(\sigma u)$ for all $\sigma\in [0,1]$ and $u\geq 0$.}
	\end{array}\right.$
\end{itemize}

\begin{theorem}[\underline{Spreading-vanishing dichotomy}]\label{th1.2}
	Suppose  {\rm \textbf{(J)}}, {\rm \bf{(f)}} and \eqref{u_0}  are satisfied.  Let $(u,h)$ be the solution of \eqref{1.3}. 	Then one of the following alternatives must occur {\rm :}
	\begin{itemize}
		\item[{\rm (i)}] {\bf Spreading:} $\lim_{t\to\infty} h(t)=\infty$ and 
		\[\mbox{ $ \lim_{t\to\yy}u(t,|x|)=u^*$ locally uniformly in $\R^N$,}\]
		\item[{\rm (ii)}] {\bf Vanishing:} $\lim_{t\to\infty} h(t)=h_\infty<\infty$ and 
		\[\mbox{$ \lim_{t\to\yy}u(t,|x|)=0$  uniformly for $x\in B_{h(t)}$.}
		\]
	\end{itemize}
\end{theorem}

\begin{theorem}[\underline{Spreading-vanishing criteria}]\label{th1.3}
	Suppose   the conditions in Theorem \ref{th1.2}  are satisfied, and $(u,h)$ is the solution of \eqref{1.3}.
	\begin{itemize}
		\item[{\rm (1)}] If $f'(0)\geq d$, then  spreading always happens.  
		\item[{\rm (2)}] 	If $f'(0)\in (0,d)$ then there exists $L_*>0$ such that 
		\begin{itemize}
			\item[{\rm (i)}] for $h_0\geq L_*$, spreading always happens,
				\item[{\rm (ii)}] for $0<h_0< L_*$, there is $\mu_*>0$ such that spreading  happens if and only if $\mu>\mu_*$.
		\end{itemize}
	\end{itemize}
\end{theorem}

In Theorem \ref{th1.3}, $L_*$ is determined by an associated eigenvalue problem, which is independent of the initial function $u_0$. On the other hand, $\mu_*$ depends on $u_0$.

Next we examine the spreading speed of \eqref{1.3} when spreading occurs.  To this end, we need to introduce the following function, which will play a pivotal role. For any $l\in\R$, define
\begin{align}\label{J}
J_*(l):=\int_{\R^{N-1}}J(|(l,x')|)dx',
\end{align}
where $x'=(x_2,..., x_N)\in\R^{N-1}$.

It is easy to see that  \textbf{(J)} implies
\begin{equation*}
\begin{cases}
J_*\in C(\R)\cap L^\infty(\R)\ {\rm  is\ nonnegative,\ even},\  J_*(0)>0,\\
\displaystyle\int_{\R} J_*(l) \rd l=\int_{\R^N}J(|x|) \rd x=1.
\end{cases}
\end{equation*}
Moreover, 
a simple calculation yields
\begin{equation}\label{1.10}\begin{aligned}
J_*(l)&=\int_{\R^{N-1}}J(|(l,x')|)dx'=\int_0^\infty J(\sqrt{l^2+\rho^2})\omega_{N-1}\rho^{N-2}d\rho\\
&=\omega_{N-1}\int_{|l|}^\infty J(r)r(r^2-l^2)^{(N-3)/2}dr,
\end{aligned}
\end{equation}
where $\omega_k$ denotes the area of the unit sphere in $\R^k$. It then follows easily that
\begin{align*}
J_*(l_2)\leq J_*(l_1)\leq J_*(0)&=\omega_{N-1}\int_0^\infty J(r)r^{N-2}dr\\
&\leq \omega_{N-1}\left[\|J\|_\infty+\int_1^\infty J(r)r^{N-1}dr  \right]\\
&\leq \omega_{N-1} (\|J\|_\infty+\omega^{-1}_N)\ \ \  \mbox{ when  $N\geq 3$ and $l_2>l_1>0$},
\end{align*}
and for $N=2$, $l\geq 0$,
\begin{align*}
J_*(l)&=2\pi \int_l^\infty J(r)\frac r{\sqrt{r^2-l^2}}dr\leq 2\pi \left[\int_{l}^{\sqrt{l^2+1}}+\int_{\sqrt{l^2+1}}^\infty\right]J(r)\frac r{\sqrt{r^2-l^2}}dr\\
&\leq 2\pi\left[\|J\|_\infty\int_{l}^{\sqrt{l^2+1}}\frac r{\sqrt{r^2-l^2}}dr+\int_0^\infty J(r)rdr \right]\\
&=2\pi(\|J\|_\infty+\omega_2^{-1}).
\end{align*}
A direct calculation also gives
\begin{equation}\label{J-J_*}
\int_0^\infty J_*(l)ldl=\frac{\omega_{N-1}}{N-1}\int_0^\infty J(r)r^Ndr.
\end{equation}

It turns out that the threshold condition for \eqref{1.3} to have a finite spreading speed  is
\begin{itemize}
	\item[\textbf{(J1):}] \  $\dd\int_0^\infty J(r)r^Ndr <+\yy$.
\end{itemize}

By \cite[Theorem 1.2]{dlz2019} and \eqref{J-J_*},  we have the following conclusions about the associated one-dimensional semi-wave problem. 
\begin{proposition}\label{prop1.4}
Suppose {\rm \textbf{(J)}} and {\rm \textbf{(f)}}  hold. 
Then the following equations 
\begin{equation*}
\begin{cases}
\dd d\int_{-\yy}^{0}J_*(x-y) \phi(y) {\rm d}y-d\phi+c\phi'(x)+f(\phi) =0,&
x<0,\\[2mm]
\dd	\phi(-\yy)=u^*,\ \ \phi(0)={0},\\
\dd c=\mu\int_{-\yy}^{0} \int_{0}^{\yy}J_*(x-y)\phi(x)\rd y\rd x,
\end{cases}
\end{equation*}
admit a solution pair $(c,\phi)=(c_0,\phi_0)$  if and only if {\rm\textbf{(J1)} } is satisfied. Moreover, when  {\rm\textbf{(J1)} } holds, the solution pair is
unique, and $c_0>0$, $\phi_0(x)$ is strictly decreasing in $x$.
\end{proposition}

\begin{theorem}[\underline{Spreading speed}]\label{th1.6}
	Assume   the conditions in Theorem \ref{th1.2}  are satisfied, and spreading happens to  \eqref{1.3}. Then 
\begin{align*}
\lim_{t\to\yy} \frac{h(t)}{t}=\begin{cases}c_0 & {\rm if\ \mathbf{(J1)}\ is\ satisfied},\\
\yy &{\rm if\ {\rm \mathbf{(J1)}}\ is\ not \ satisfied},
\end{cases}
\end{align*}
where $c_0$ is given by Proposition \ref{prop1.4}.
\end{theorem}

To obtain sharper estimates of the spreading speed, we focus on two important classes of kernel functions. The first consists of those with compact support, and therefore {\bf (J1)} is automatically satisfied and the spreading has a finite speed $c_0$ determined by Proposition \ref{prop1.4}. We show that in such a case $c_0t-h(t)$ grows to infinity like $\ln t$.
  Recall that $\xi(t)\approx \eta(t)$ for all large $t$
means there exist positive constants $c_1, c_2$ and $T$ such that
\[
c_1\eta(t)\leq\xi(t)\leq c_2\eta(t) \mbox{ for } t\geq T.
\]

\begin{theorem}[\underline{Logarithmic shift}]\label{th1.7}
	Suppose   the conditions in Theorem \ref{th1.2}  hold, and moreover the kernel function $J$ has  compact support and $f$ is $C^2$. If spreading happens, then 	\begin{align*}
 c_0t-h(t)\approx \ln t \mbox{ for all large } t.
	\end{align*}
\end{theorem}

This result reveals a striking difference from  the one dimensional situation in \cite[Theorem 1.4]{dn-speed}, which gives $c_0t-h(t)\approx 1$ for all large $t$ when the kernel function $J$ has compact support. Note that for kernel functions which satisfy {\bf (J1)} but do not have compact support,  $c_0t-h(t)$ may go to infinity faster than  $\ln t$, as already observed in the one dimension case. 

\medskip

The second class  consists of kernel functions $J(r)$ which behave like $r^{-\beta}$ for large $r$, and we have the following result on the rate of accelerated spreading.

\begin{theorem}[\underline{Rate of accelerated spreading}]\label{th1.8}
		Suppose   the conditions in Theorem \ref{th1.2}  are satisfied, and there exists $\beta\in (N,N+1]$  such that  
	$ J(\eta)\approx \eta^{-\beta}$ for all large $\eta$. 
		If  spreading happens, then for all large $t$,
	\begin{equation*}
\begin{cases}
 h(t)\approx t^{1/(\beta-N)}&{\rm if}\ \beta\in (N,N+1),\\
h(t)\approx t\ln t&{\rm if}\ \beta=N+1.
\end{cases}
\end{equation*}
\end{theorem}

Let us note that when $J(\eta)\approx \eta^{-\beta}$ for all large $\eta$, condition {\bf (J)} implies $\beta>N$, and {\bf (J1)} holds if and only if $\beta>N+1$. Therefore Theorem \ref{th1.8} covers exactly the case that {\bf (J)} holds but {\bf (J1)} does not, which is the very case that accelerated spreading can happen.

For such $J$ with $\beta>N+1$, as {\bf (J1)} holds, by Theorem \ref{th1.6}, when spreading happens, $\lim_{t\to\infty} h(t)/t=c_0$ is finite. Then one natural question is to find the rate of $c_0t-h(t)$ as $t\to\infty$, similar to what was done in \cite{dn-speed} for the one dimension case. It turns out that this question is much more difficult to answer in high dimensions, and the techniques here are not enough to cover this case;
in view of the length of this paper, we have refrained to pursue it here.

\bigskip

One major difficulty in treating the high dimension radially symmetric problem \eqref{1.4} arises from the fact that the kernel function in \eqref{1.4} is given by
\[
\td J(r,\rho)=\td J(|x|,\rho):=\int_{\partial B_\rho} J(|x-y|)\rd S_y,
\]
which inherits the properties of the original kernel function $J(|x|)$ in a rather implicit way. Moreover, the kernel function which determines the spreading speed of \eqref{1.4} is given by
\[
J_*(l):=\int_{\R^{N-1}}J(|(l,x')|)dx',
\]
and therefore the spreading behaviour of \eqref{1.4} can be understood only if the relationship between $J$, $\td J$ and $J_*$ is reasonably clear. 

We note that such difficulties do not occur in the random diffusion case. Indeed,
the random diffusion counterpart of \eqref{1.4} has the form
\begin{equation}\label{1.4-local}
\begin{cases}
\dd  u_t(t,r)=d \left[u_{rr}(t,r)+\frac{N-1}r u_r(t,r)\right]+f(t,r,  u), & t>0,\; r\in  [0,h(t)),\\
u_r(t,0)= u(t,h(t))=0, & t>0,\\
\dd h'(t)=  \mu u_r(t, h(t)), & t>0,\\
h(0)=h_0,\ u(0,r)= u_0(r), & r\in [0,h_0],
\end{cases}
\end{equation}
and was examined in \cite{DG, DMZ}. The sole difference of \eqref{1.4-local} from the corresponding one dimension model is the additional term $\frac{N-1}r u_r(t,r)$ in the first equation of \eqref{1.4-local}.
When $f$ is of Fisher-KPP type, namely $f=f(u)$ satisfies {\bf (f)}, it follows from \cite{DG} that the long-time dynamics of \eqref{1.4-local} is roughly the same as that for the one dimension case considered in \cite{DL}, and when spreading happens, $\lim_{t\to\infty} h(t)/t=c_*$ for some $c_*>0$ determined by the semi-wave problem associated to the one dimensional model. Moreover, by \cite{DMZ}, there exists
another constant $\hat c>0$ independent of the dimension $N$ such that
\[
\lim_{t\to\infty}[h(t)-c_*t+(N-1)\hat c\ln t]=C
\]
for some constant $C$ depending on the initial function $u_0$. In contrast, when $J$ has compact support, our corresponding result Theorem \ref{th1.7} is not as precise yet.

It was shown in \cite{DG2} that when $\mu\to\infty$, the limiting problem of \eqref{1.4-local} is the corresponding Cauchy problem
\begin{equation}\label{Cauchy-local}\begin{cases}
u_t=d\Delta u+f(t, |x|, u) &\mbox{ for } (t,x)\in \R_+\times \R^N, \\
 u(0,x)=u_0(x) &\mbox{ for } x\in\R^N,
 \end{cases}
\end{equation}
which, since the pioneering works of Fisher \cite{F} and
Kolmogorov, Peterovski and Piskunov \cite{KPP}, has long been used to describe the propagation phenomena arising from invasion ecology and other problems. Similarly, it can be easily shown that when $\mu\to\infty$, the limiting problem of \eqref{1.3} is the nonlocal Cauchy problem
\begin{equation}\label{Cauchy}\begin{cases}
u_t=d\left[\dd\int_{\R^N}J(|x-y|)u(t,y)dy -u(t,x)\right]+f(t, |x|,u) &\mbox{ for } (t,x)\in\R_+\times \R^N,\\
 u(0,x)=u_0(x) &\mbox{ for } x\in\R^N.
 \end{cases}
\end{equation}
As a nonlocal extension of \eqref{Cauchy-local}, problem \eqref{Cauchy} and its various variations have been extensively studied in the last three decades (see, e.g., \cite{AC, A-V, bz2007, bcv2016, BGHP, CD, CDM, FG, FT, G, lcw2017, LZ, SZ, WZ, Y} and the references therein). When $f=f(u)$ is of Fisher-KPP type, the long-time behaviour of \eqref{Cauchy} 
with a compactly supported  initial function $u_0$ is roughly the same as \eqref{Cauchy-local}, namely
\begin{equation}\label{u-u^*}
\lim_{t\to\infty} u(t,x)=u^* \mbox{ locally uniformly for } x\in\R^N,
\end{equation}
where $u^*$ is the unique positive zero of $f(u)$ given in {\bf (f)}. However, differences arise when one looks at the spreading speed, where
 accelerated spreading can happen to \eqref{Cauchy} when the kernel function $J$ is fat-tailed, while \eqref{Cauchy-local} always spreads with a finite speed, determined by the minimal speed of its traveling wave solutions. The determination of the rate of accelerated spreading has been a difficult problem.
In space dimension one, the rate of accelerated spreading of \eqref{Cauchy} has been examined in several works (see, e.g., \cite{BGHP, G}), but no such result   appears available  for the case of higher space dimensions yet. On the other hand, in \cite{CR, ST}, for fractional Laplacian type nonlocal diffusion operators in any dimension $N\geq 1$, it was shown that the rate of accelerated spreading is given by $e^{[c+o(1)]t}$ for some $c>0$ depending on $N$ and the fractional Laplacian. It should be noted that  our basic condition {\bf (J)} here is not satisfied  by  the corresponding kernel function of the fractional Laplacian $(-\Delta)^s$, which is given by 
\[
J(|x|)=|x|^{-(N+2s)} \ (0<s<1).
\]
 It would be interesting to see what happens to \eqref{1.3} if the kernel function $J$ is allowed to behave like the kernel function of the fractional Laplacian. A related work with $f\equiv 0$ 
 can be found in \cite{dTEV}.
 
 The techniques developed here are useful to obtain sharp estimates of the spreading profile of \eqref{Cauchy} in high space dimensions, which will be considered in a separate work. Note that as population models, \eqref{1.3} and \eqref{1.4-local} have several advantages over  \eqref{Cauchy-local} and \eqref{Cauchy}. For example, they both give the precise spreading front of the species via the free boundaries, while \eqref{Cauchy-local} and \eqref{Cauchy} do not, since their solution $u(t,x)$ is positive for all $x\in\R^N$ once $t>0$; moreover, \eqref{Cauchy-local} and \eqref{Cauchy} predict consistant success of spreading (see \eqref{u-u^*}), but the long-time dynamics of  \eqref{1.3} and \eqref{1.4-local} is governed by a spreading-vanishing dichotomy, which seems more realistic.

\bigskip

The rest of the paper is organised as follows.   In Section 2 we prove some useful facts about the kernel function $J$ and the associated functions $\tilde J$ and $J_*$, which pave the way for further analysis of \eqref{1.4}.  In Section 3, we prove the well-posedness of \eqref{1.4} (Theorem \ref{th1.1}) and a comparison principle, which will be used in later sections. The arguments in this section are variations of those for the one dimension case in \cite{cdjfa, dn2020}, thanks to the preparations in Section 2.  The spreading-vanishing dichotomy and its governing
criteria (Theorems \ref{th1.2} and \ref{th1.3}) are proved in Section 4, following the approach of \cite{cdjfa}.  The spreading speed of  \eqref{1.4} is considered in Section 5, where Theorem \ref{th1.6} is proved. Compared to \cite{dlz2019} where similar results for the one dimensional case was proved, here the proof is much more difficult as the arguments rely on careful analysis of the relationship between $J$, $\tilde J$ and $J_*$. The most technical parts of the paper are Sections 6 and 7. Section 6 examines the asymptotic behaviour of $h(t)-c_0t$ for large $t$ when the kernel function has compact support, where Theorem \ref{th1.7} is proved by the construction of subtle upper and lower solutions, based on careful estimates of a variety of expressions involving $\td J$ and $J_*$. Section 7 is concerned with the rate of accelerated spreading  when the kernel function $J(r)$ is assumed to behave like $r^{-\beta}$ for large $r$ with $\beta\in (N, N+1)$, and Theorem \ref{th1.8} is proved; again this relies on the construction of suitable upper and lower solutions, based on  careful analysis of the behaviours of $\tilde J$ and $J_*$.

\section{Some basic facts on the kernel and associated functions}

In this section, we obtain some properties of the functions $J(r)$, $\tilde J(r,\rho)$ and $J_*(l)$, which will play important roles for our analysis in later sections. 

\begin{lemma}\label{lemma2.2} For $x\in\R^N\setminus \{0\}$ and $\rho>0$,
	let 
	\[
	\mbox{$\Gamma_+(x,\rho)=\{y\in \partial B_\rho:  y\cdot x\geq 0\}$ and $\Gamma_-(x,\rho)=\{y\in\partial B_\rho: y\cdot x\leq 0\}$.}
	\]
	 Then for $r=|x|$,
	\begin{align*}
	\td J_+(r,\rho):=&\int_{\Gamma_+(x,\rho)} J(|x-y|) \rd S_y\\
	=&\omega_{N-1}2^{3-N}\frac{\rho}{r^{N-2}}\int_{|\rho-r|}^{\sqrt{\rho^2+r^2}}  \Big(\lf[(\rho+r)^2-\eta^2\rr][\eta^2-(\rho-r)^2]\Big)^{\frac{N-3}2}\eta J(\eta) \rd \eta,
	\end{align*}
	\begin{align*}
	\td J_-(r,\rho):=&\int_{\Gamma_-(x,\rho)} J(|x-y|) \rd S_y\\
	=&\omega_{N-1}2^{3-N}\frac{\rho}{r^{N-2}}\int^{\rho+r}_{\sqrt{\rho^2+r^2}}  \Big(\lf[(\rho+r)^2-\eta^2\rr][\eta^2-(\rho-r)^2]\Big)^{\frac{N-3}2}\eta J(\eta) \rd \eta,
	\end{align*}
	and
	\begin{align*}
	\td J(r,\rho)= \omega_{N-1}2^{3-N}\frac{\rho}{r^{N-2}}\int^{\rho+r}_{|\rho-r|}  \Big(\lf[(\rho+r)^2-\eta^2\rr][\eta^2-(\rho-r)^2]\Big)^{\frac{N-3}2}\eta J(\eta) \rd \eta.
	\end{align*}
\end{lemma}
\begin{proof}
 For any  given $y\in\partial B_\rho$ and $x\in\partial B_r$ with $\rho, r>0$, let $\theta$ denote the angle between $y$ and $x$, namely $y\cdot x=\rho r\cos\theta$; then let $S_\theta$ denote the intersection of the hyperplane $H_\theta:=\{z\in\R^N: z\cdot x=\rho\cos\theta\}$ with the sphere $\partial B_\rho$, which clearly is an $N-2$ dimensional sphere of radius $\rho\sin\theta$. Then
\[
J(|x-y|)\equiv J\lf(\sqrt{\rho^2+r^2-2\rho r \cos \theta}\rr) \mbox{ for } y\in S_\theta,
\]
 and
	\begin{align*}
	\dd\td J(r,\rho)&=\int_{\partial B_\rho}J(|x-y|)dS_y=\int_0^\pi J\lf(\sqrt{\rho^2+r^2-2\rho r \cos \theta}\rr)|S_\theta|\rho d\theta\\
	=&\int_{0}^{\pi} J\lf(\sqrt{\rho^2+r^2-2\rho r \cos \theta}\rr) \omega_{N-1} (\rho\sin \theta)^{N-2}\rho\rd \theta\\
	=&\omega_{N-1}\int_{0}^{\pi}   \rho^{N-1}(\sin \theta)^{N-2} J\lf(\sqrt{(\rho-r)^2+2r\rho(1- \cos \theta )}\rr) \rd \theta\\
	=&\omega_{N-1}\int_{0}^{\pi}   \rho^{N-1}[2 \sin (\theta/2)\cos (\theta/2)]^{N-2} J\lf(\sqrt{(\rho-r)^2+4r\rho\sin^2 (\theta/2)}\rr) \rd \theta\\
	=&2^{N-1}\rho^{N-1}\omega_{N-1}\int_{0}^{1}   \xi^{N-2} (1-\xi^2)^{(N-3)/2}J\lf(\sqrt{(\rho-r)^2+4r\rho\, \xi^2}\rr)\rd \xi,
	\end{align*}
	where we have used $\xi=\sin (\theta/2)$. The change of variable $\eta=\sqrt{(\rho-r)^2+4r\rho\, \xi^2}$ then gives
	\begin{align}\label{2.3a}
	\dd\td J(r,\rho)=\omega_{N-1}2^{3-N}\frac{\rho}{r^{N-2}}\int^{\rho+r}_{|\rho-r|}  \Big(\lf[(\rho+r)^2-\eta^2\rr][\eta^2-(\rho-r)^2]\Big)^{\frac{N-3}2}\eta J(\eta) \rd \eta.
	\end{align}
				
Analogously, by the definition of  $\td J_+(r,\rho)$, we obtain
\begin{align*}
	\dd\td J_+(r,\rho)=&\omega_{N-1}\int_{0}^{\pi/2}  (\rho\sin \theta)^{N-2}J\lf(\sqrt{\rho^2+r^2-2\rho r \cos \theta}\rr) \rho\rd \theta\\
	=&2^{N-1}\rho^{N-1}\omega_{N-1}\int_{0}^{\sqrt{2}/2}   \xi^{N-2} (1-\xi^2)^{(N-3)/2}J\lf(\sqrt{(\rho-r)^2+4r\rho \,\xi^2}\rr)\rd \xi\\
=&\omega_{N-1}2^{3-N}\frac{\rho}{r^{N-2}}\int_{|\rho-r|}^{\sqrt{\rho^2+r^2}}  \Big(\lf[(\rho+r)^2-\eta^2\rr][\eta^2-(\rho-r)^2]\Big)^{\frac{N-3}2}\eta J(\eta) \rd \eta.
\end{align*}

Similarly,
\begin{align*}
\dd\td J_-(r,\rho)=&\omega_{N-1}\int_{\pi/2}^{\pi}  (\rho\sin \theta)^{N-2}J\lf(\sqrt{\rho^2+r^2-2\rho r \cos \theta}\rr) \rho\rd \theta\\
=&2^{N-1}\rho^{N-1}\omega_{N-1}\int_{\sqrt{2}/2}^1   \xi^{N-2} (1-\xi^2)^{(N-3)/2}J\lf(\sqrt{(\rho-r)^2+4r\rho \,\xi^2}\rr)\rd \xi\\
=&\omega_{N-1}2^{3-N}\frac{\rho}{r^{N-2}}\int^{\rho+r}_{\sqrt{\rho^2+r^2}}  \Big(\lf[(\rho+r)^2-\eta^2\rr][\eta^2-(\rho-r)^2]\Big)^{\frac{N-3}2}\eta J(\eta) \rd \eta.
\end{align*}
\end{proof}

Define
\begin{equation}\label{2.3}
\zeta_0(\eta)=\begin{cases}
0, & \eta<0,\\
 2\eta,  &\eta\in [0,1],\\
 2, & \eta>1,
\end{cases}
\end{equation}
and for $\epsilon>0$ and $l\in\R$, define
\begin{equation}\label{2.5d}
J_\epsilon(l):=\int_{\R^{N-1}} J(|(l, y')|)[1+\zeta_0(|(l, y')|-\epsilon^{-1})] \rd  y'.
\end{equation}

\begin{lemma}\label{lemma2.4} 
	For any given small numbers $\delta>0$ and $\epsilon>0$,  
	\begin{align*}
	\td J_+(r,\rho)\leq (1+\delta) J_\epsilon (r-\rho) \mbox{ for }  \rho\in \left [\frac r 2, (1+\delta^2)r\right ],\ r\geq (\delta\epsilon)^{-1},
	\end{align*} 
	where $\td J_+$ is given by Lemma \ref{lemma2.2}.
\end{lemma}
\begin{proof}
	
	\textbf{Step 1:} Split of  $J_\epsilon(r-\rho)$. 
	
	Denote 
	\begin{align*}
	G_\epsilon(|x|):=J(|x|)[1+\zeta_0(|x|-\epsilon^{-1})]\ \ \ \ {\rm for}\ x\in \R^N.
	\end{align*} 
Then
	\begin{align*}
	J_\epsilon(r-\rho)=&\int_{0}^{\yy} \omega_{N-1} \xi^{N-2}G_\epsilon(\sqrt{(r-\rho)^2+\xi^2}) \rd \xi\\
	=&\ \omega_{N-1}\int_{|\rho-r|}^{\yy}   [\eta^2-(\rho-r)^2]^{(N-3)/2} \eta G_\epsilon(\eta)\rd \eta\\
	=&\ \omega_{N-1}\int_{|\rho-r|}^{\yy}   [\eta^2-(\rho-r)^2]^{(N-3)/2} \eta J(\eta)\rd \eta\\
	&+\omega_{N-1}\int_{|\rho-r|}^{\yy}   [\eta^2-(\rho-r)^2]^{(N-3)/2} \eta J(\eta)\zeta_0(\eta-\epsilon^{-1})\rd \eta\\
	=:&\ W_1+W_2.
	\end{align*}

	\textbf{Step 2:} Split of  $\td J_+(r,\rho)$.
	
	By Lemma \ref{lemma2.2}, we have
	\begin{align*}
	&\td J_+(r,\rho)\\	=&\ \omega_{N-1}2^{3-N}\frac{\rho}{r^{N-2}}\int_{|\rho-r|}^{\sqrt{(r-\rho)^2+4\delta^2r\rho}}  \Big(\lf[(\rho+r)^2-\eta^2\rr][\eta^2-(\rho-r)^2]\Big)^{\frac{N-3}2}\eta J(\eta) \rd \eta
	\\
	&+\omega_{N-1}2^{3-N}\frac{\rho}{r^{N-2}}\int^{\sqrt{r^2+\rho^2}}_{\sqrt{(r-\rho)^2+4\delta^2r\rho}}  \Big(\lf[(\rho+r)^2-\eta^2\rr][\eta^2-(\rho-r)^2]\Big)^{\frac{N-3}2}\eta J(\eta) \rd \eta
	\\
	=: & \ Q_1+Q_2.
	\end{align*}

	\textbf{Step 3: } We prove $Q_1\leq (1+\delta)W_1$.

	For  $|\rho-r|\leq \eta\leq \sqrt{(r-\rho)^2+4\delta^2r\rho}$, we have
	\begin{align*}
	\lf[(\rho+r)^2-\eta^2\right]^{(N-3)/2}\leq \begin{cases} (4r\rho)^{(N-3)/2}, & N\geq 3,\\
	[4r\rho (1-\delta^2)]^{-1/2}, & N=2.
	\end{cases}
	\end{align*}
	Hence by the definitions of $Q_1$ and $W_1$ we have, for $\rho\in [r/2, (1+\delta^2)r]$,
	\begin{align*}
	Q_1\leq&\omega_{N-1}2^{3-N} \frac\rho{r^{N-2}}\frac{(4r\rho)^{(N-3)/2}}{\sqrt{1-\delta^2}}\int_{|\rho-r|}^{\sqrt{(r-\rho)^2+4\delta^2r\rho}}  
	 [\eta^2-(\rho-r)^2]^{(N-3)/2}\eta J(\eta) \rd \eta\\
	=&\frac{w_{N-1}}{\sqrt{1-\delta^2}}\left(\frac\rho r\right)^{(N-1)/2}\int_{|\rho-r|}^{\sqrt{(r-\rho)^2+4\delta^2r\rho}}  
	 [\eta^2-(\rho-l)^2]^{(N-3)/2}\eta J(\eta) \rd \eta\\
\leq &\frac {(1+\delta^2)^{(N-1)/2}}{\sqrt{1-\delta^2}}W_1\leq (1+\delta)W_1 \mbox{ sine $\delta>0$ is small}.
	\end{align*}

	\textbf{Step 4:}  We show $Q_2\leq W_2$. \smallskip

For  $ \sqrt{(r-\rho)^2+4\delta^2r\rho}\leq \eta\leq \sqrt{r^2+\rho^2}$, we have
	\begin{align*}
	\lf[(\rho+r)^2-\eta^2\right]^{(N-3)/2}\leq \begin{cases} (4r\rho)^{(N-3)/2}, & N\geq 3,\\
	(2r\rho)^{-1/2}, & N=2.
	\end{cases}
	\end{align*}
	Hence by the definitions of $Q_2$ and $W_2$ we have, for $\rho\in [r/2, (1+\delta^2)r]$,
	\begin{align*}
	Q_2\leq&\omega_{N-1}2^{3-N} \frac\rho{r^{N-2}}{(4r\rho)^{(N-3)/2}}{\sqrt{2}}\int^{\sqrt{r^2+\rho^2}}_{\sqrt{(r-\rho)^2+4\delta^2r\rho}}  
	 [\eta^2-(\rho-r)^2]^{(N-3)/2}\eta J(\eta) \rd \eta\\
	=&\ {w_{N-1}}{\sqrt{2}}\left(\frac\rho r\right)^{(N-1)/2}\int^{\sqrt{r^2+\rho^2}}_{\sqrt{(r-\rho)^2+4\delta^2r\rho}}  
	 [\eta^2-(\rho-r)^2]^{(N-3)/2}\eta J(\eta) \rd \eta\\
	 \leq &\ {w_{N-1}} 2\int^{\sqrt{r^2+\rho^2}}_{\max\{|\rho-r|,\sqrt{2}\delta r\}}  
	 [\eta^2-(\rho-r)^2]^{(N-3)/2}\eta J(\eta) \rd \eta.
	\end{align*}

On the other hand,
\[
W_2\geq \omega_{N-1} 2\int_{\max\{|\rho-r|, \epsilon^{-1}+1\}}^{\yy}   [\eta^2-(\rho-r)^2]^{(N-3)/2} \eta J(\eta)\rd \eta.
\]
	Therefore, $Q_2\leq W_2$ provided that $\sqrt 2\delta r\geq \epsilon^{-1}+1$, namely $r\geq R:=\frac{\epsilon^{-1}+1}{\sqrt 2 \delta}\in (0, (\delta\epsilon)^{-1}]$.
	
Thus	
	\[
	\td J_2(r,\rho)=Q_1+Q_2\leq (1+\delta)W_1+W_2\leq (1+\delta) J_\epsilon(r-\rho) \mbox{ for } \rho\in [r/2, (1+\delta^2)r],\; r\geq (\delta\epsilon)^{-1}.
	\]
	The proof is complete.
\end{proof}


\begin{proposition}\label{coro2.5}
	 Suppose that  $J$ has compact support, say  ${\rm supp} (J)\subset [0,K_*]$ for some $K_*>0$, and $J_*$ is given by \eqref{J}.
	Then  there exist constants $L_0>0$ and $C>0$  such that   for  $r\geq L_0$, 
	\begin{align*}
	&| \td J(r,\rho)- J_*(r-\rho)|\leq  Cr^{-1}&\mbox{ when }& \rho\in [r-K_*,r+K_*],\\
	& \td J(r,\rho)=J_*(r-\rho)=0 &\mbox{ when }& \rho\not\in [r-K_*,r+K_*].
	\end{align*}
\end{proposition} 
\begin{proof}  
	From Lemma \ref{lemma2.2} and Lemma \ref{lemma2.4}, we see
	\begin{align*}
	&\td J(r,\rho)= \omega_{N-1}2^{3-N}\frac{\rho}{r^{N-2}}\int^{\rho+r}_{|\rho-r|}  \Big(\lf[(\rho+r)^2-\eta^2\rr][\eta^2-(\rho-r)^2]\Big)^{\frac{N-3}2}\eta J(\eta) \rd \eta,\\
	&J_*(r-\rho)=\omega_{N-1}\int_{|\rho-r|}^{\yy}   [\eta^2-(\rho-r)^2]^{(N-3)/2} \eta J(\eta)\rd \eta.
	\end{align*}
	Hence,
	\begin{align*}
	\td J(r,\rho)=J_*(r-\rho)=0 \mbox{ when \ $ |r-\rho|> K_*$, namely when $\rho \not\in [r-K_*,r+K_*].$}
	\end{align*}
	
	On the other hand, for   $r> K_*$ and $|r-\rho|\leq K_*$,
	\begin{align*}
	&| \td J(r,\rho)- J_*(r-\rho)|\\
	\leq &\ \omega_{N-1}\int_{|\rho-r|}^{K_*}   \lf| \frac{2^{3-N}\rho}{r^{N-2}}(\rho+r)^2-\eta^2)^{\frac{N-3}{2}}-1
	\rr|[\eta^2-(\rho-r)^2]^{(N-3)/2}\eta 
	J(\eta) \rd \eta\\
	\leq&\  \omega_{N-1} \|J\|_\infty M(r)\int_{|\rho-r|}^{K_*}[\eta^2-(\rho-r)^2]^{(N-3)/2}\eta d\eta, 
	\end{align*} 
	where
	\begin{align*}
	M(r):=&\ {\rm max} _{\{\rho\in [r-K_*, r+K_*], \eta\in [0, K_*]\}}\left| \frac{2^{3-N}\rho}{r^{N-2}}[(\rho+r)^2-\eta^2)^{\frac{N-3}{2}}-1  \right|\\
	=&\ {\rm max}_{\{\rho\in [r-K_*, r+K_*], \eta\in [0, K_*]\}}\left|2^{3-N}\frac{\rho}r \left[\Big(1+\frac\rho r\Big)^2-\Big(\frac\eta r\Big)^2\right]^{\frac{N-3}2}-1\right|\\
	 =&\ {\rm max}_{\{\xi\in [-K_*, K_*], \eta\in [0, K_*]\}}\left|\left(1+\frac{\xi}r\right) \left[\Big(1+\frac\xi{2 r}\Big)^2-\Big(\frac\eta {2r}\Big)^2\right]^{\frac{N-3}2}-1\right|\\
	 =&\ O(r^{-1}) \mbox{ as } r\to\infty,
	\end{align*}
	and
	\begin{align*}
	\int_{|\rho-r|}^{K_*}[\eta^2-(\rho-r)^2]^{\frac{N-3}2}\eta d\eta\leq
	\begin{cases}\dd \int_0^{K_*} \eta^{N-2} d\eta & \mbox{ if } N\geq 3,\\
	\dd \int_{|\rho-r|}^{K_*}\frac{\eta}{\sqrt{\eta+|\rho-r|}}\frac{1}{\sqrt{\eta-|\rho-r|}}d\eta\leq \int_0^{K_*}\frac{\sqrt{K_*}}{\sqrt{\xi}} d\xi & \mbox{ if } N=2.
	\end{cases}
	\end{align*}
	Therefore there exists $C>0$ such that 
	\[
	| \td J(r,\rho)- J_*(r-\rho)|\leq C r^{-1} \mbox{ for $|r-\rho|\leq K_*$ and all large $r$}.
	\]
	The proof is complete.
\end{proof}

\begin{lemma}\label{lemma2.6d}
	If  {\rm \textbf{(J1)}} holds, then for any  $\epsilon\in (0,1)$,  
	\begin{align}
	\lim_{R\to\infty}\int_0^{(1-\epsilon)R} \int_{R}^{\yy} \td J_+(r,\rho) \rd \rho \rd r=0,\ \lim_{R\to\infty}\int_0^R \int_{R}^{\yy} \td J_-(r,\rho) \rd \rho \rd r=0.
	\end{align}
\end{lemma}
\begin{proof}
	For $R\geq r>0$, denote 
	\[
	x_r^1:=(r,0,..., 0)\in\R^N \mbox{ and	} \Omega_R^+:=\{y=(y_1,..., y_N): y_1>0, |y|>R\}.
	\]
	Then
	\[
	\int_R^\infty \td J_+(r,\rho)d\rho=\int_{\Omega_R^+}J(|x_r^1-y|dy.
	\]

	  For small $\delta\in (0,\epsilon)$ define 
	\begin{align*}
	Q_R^\delta:=\lf\{z=(z_1,z_2,\cdots,z_N): z_1\leq (1-\delta)R,\  |z_i|\leq \Lambda R,\ 2\leq i\leq N\rr\}
	\end{align*}
	with $\Lambda:=\sqrt{1-(1-\delta)^2}/\sqrt{N}$.
	Obviously,
		\begin{align*}
	\Omega_R^+\subset \R^N\setminus Q_R^\delta.
	\end{align*}
	Therefore
	\begin{align*}
	\int_R^\infty \td J_+(r,\rho)d\rho=\int_{\Omega_R^+}J(|x_r^1-y|dy\leq \int_{\R^N\setminus Q_R^\delta}J(|x_r^1-y|dy.
	\end{align*}
	The set $\R^N\setminus Q_R^\delta$ can be decomposed as follows:
	\begin{align*}
	\R^N\backslash Q_R^\delta= \cup_{i=1}^{N}S^{(i)} 
	\end{align*}
with overlapping sets
	\begin{align*}
	&S^{(1)}:=\lf\{z=(z_1,z_2,\cdots,z_N): z_1>(1-\delta)R\ {\rm and}\ z_i\in \R\ {\rm for}\ 2\leq i\leq N\rr\},\\
	& S^{(j)}:=\lf\{z=(z_1,z_2,\cdots,z_N): |z_j|>\Lambda R\ {\rm and}\ z_i\in \R\ {\rm for}\ i\neq j\rr\},\ 2\leq j\leq N.
	\end{align*}
	Thus, making use of the definition of $J_*$, we deduce
	\begin{align*}
	  \int_R^\infty \td J_+(r,\rho)d\rho\leq &\int_{\R^N\setminus Q_R^\delta}J(|x_r^1-y|dy 
	\leq \sum_{j=1}^N \int_{S^{(j)}} J(|x_r^1-y|dy \\
	=& \int_{(1-\delta)R}^{\yy} J_*(r-\rho) \rd \rho 
  +2(N-1)\int_{\Lambda R}^\infty J_*(\rho)d\rho\\
  =& \int_{(1-\delta)R-r}^{\yy} J_*(\xi) \rd \xi 
  +2(N-1)\int_{\Lambda R}^\infty J_*(\rho)d\rho.
	\end{align*}
	It follows that
	\begin{align*}
	\int_0^{(1-\epsilon)R} \int_R^\infty \td J_+(r,\rho)d\rho dr\leq 
	 \int_0^{(1-\epsilon)R} \int_{(1-\delta)R-r}^{\yy} J_*(\xi) \rd \xi dr  +2(N-1)(1-\epsilon)R\int_{\Lambda R}^\infty J_*(\rho)d\rho
  \end{align*}
  
  We have, due to {\bf (J1)},
	 \begin{align*}
   \int_0^{(1-\epsilon)R}\int_{(1-\delta)R-r}^{\yy} J_*(\xi) \rd \xi dr\leq & \int_{(\epsilon-\delta)R}^\infty [\xi-(\epsilon-\delta)R]J_*(\xi)d\xi\\
   \leq & \int_{(\epsilon-\delta)R}^\infty \xi J_*(\xi)d\xi \to 0 \mbox{ as } R\to\infty,
	\end{align*}
	and
	\[
	2(N-1)(1-\epsilon)R\int_{\Lambda R}^\infty J_*(\rho)d\rho\leq 2(N-1)(1-\epsilon)\frac{1}\Lambda \int_{\Lambda R}^\infty \rho J_*(\rho)d\rho \to 0 \mbox{ as } R\to\infty.
	\]
	Hence
	\[
	\int_0^{(1-\epsilon)R} \int_R^\infty \td J_+(r,\rho)d\rho dr \to 0 \mbox{ as } R\to\infty.
	\]
	
	Similarly, for $R\geq r>0$,  
	\[
	x_r^1:=(r,0,..., 0)\in\R^N \mbox{ and	} \Omega_R^-:=\{y=(y_1,..., y_N): y_1<0, |y|>R\},
	\]
	we have
	\[
	\int_R^\infty \td J_-(r,\rho)d\rho=\int_{\Omega_R^-}J(|x_r^1-y|dy.
	\]
	
	Let
\begin{align*}
	\tilde Q_R:=\lf\{z=(z_1,z_2,\cdots,z_N): z_1\geq -\frac R{2\sqrt N},  |z_i|\leq \frac R{2\sqrt N},\ 2\leq i\leq N\rr\}.
	\end{align*}
	Then 
	\[
	\Omega_R^-\subset\R^N\setminus \tilde Q_R=\cup_{j=1}^N\tilde S^{(j)},
	\]
	with
	\begin{align*}
	&\tilde S^{(1)}:=\lf\{z=(z_1,z_2,\cdots,z_N): z_1\leq -\frac R{2\sqrt N},  z_i\in\R \mbox{ for } i\not=1\rr\},\\
	&\tilde S^{(j)}:=\lf\{z=(z_1,z_2,\cdots,z_N):  |z_j|\geq \frac R{2\sqrt N} \mbox{ and } z_i\in\R \mbox{ for } i\not=j\rr\}, \ 2\leq j\leq N.
	\end{align*}
	Therefore
	\begin{align*}
	\int_R^\infty \td J_-(r,\rho)d\rho& =\int_{\Omega_R^-}J(|x_r^1-y|dy\leq \sum_{j=1}^N \int_{\tilde S^{(j)}}J(|x_r^1-y|dy\\
	&\leq \int_{-\infty}^{-\frac R{2\sqrt N}} J_*(r-\rho)d\rho+2(N-1)\int^{\infty}_{\frac R{2\sqrt N}}J_*(\rho)d\rho
	\\
	&=\int_{\frac R{2\sqrt N}+r}^\infty J_*(\xi)d\xi+2(N-1)\int^{\infty}_{\frac R{2\sqrt N}}J_*(\rho)d\rho.
	\end{align*}
	It follows that
	\begin{align*}
	\int_0^R\int_R^\infty \td J_-(r,\rho)d\rho dr	
		&\leq \int_0^R\int_{\frac R{2\sqrt N}+r}^\infty J_*(\xi)d\xi dr +2(N-1)R\int^{\infty}_{\frac R{2\sqrt N}}J_*(\rho)d\rho\\
		&\leq \int_{\frac R{2\sqrt N}}^\infty \xi J_*(\xi)d\xi+4(N-1)\sqrt N\int^{\infty}_{\frac R{2\sqrt N}}\rho J_*(\rho)d\rho\to 0 \mbox{ as } R\to\infty.
	\end{align*}
	The proof is completed. 
\end{proof}
\begin{lemma}\label{lemma2.7d}
	If 	  {\rm \textbf{(J1)}} holds, then
	\begin{align*}
	\lim_{R\to\yy}	\int_0^R \int_{R}^{\yy} \td J(r,\rho) \rd \rho \rd r=\int_{0}^{\yy} \eta J_*(\eta) \rd \eta.
	\end{align*}
\end{lemma}
\begin{proof} We complete the proof in two steps.

\textbf{Step 1}. We show that
\begin{equation}\label{step1}
	\limsup_{R\to\yy}	\int_0^R \int_{R}^{\yy} \td J(r,\rho) \rd \rho \rd r
	\leq \int_0^\infty \eta J_*(\eta)d\eta.
	\end{equation}

By Lemma \ref{lemma2.6d}, 	for any small $\epsilon_1>0$, 
	\begin{align*}
	\limsup_{R\to\yy}	\int_0^R \int_{R}^{\yy} \td J(r,\rho) \rd \rho \rd r =\limsup_{R\to\yy}	\int_{(1-\epsilon_1)R}^R \int_{R}^{\yy} \td J_+(r,\rho) \rd \rho \rd r,
	\end{align*}
	and
	\[
	\limsup_{R\to\infty}\int_{(1-\epsilon_1)R}^R \int_{(1+\epsilon_1)R}^{\yy} \td J_+(r,\rho) \rd \rho \rd r\leq \lim_{\tilde R\to\infty}\int_0^{(1+\epsilon_1)^{-1}\td R}\int_{\td R}^\infty \td J_+(r,\rho) \rd \rho \rd r=0.
	\]
	Therefore
	\begin{equation}\label{ep1}
	\limsup_{R\to\yy}	\int_0^R \int_{R}^{\yy} \td J(r,\rho) \rd \rho \rd r =\limsup_{R\to\yy}	\int_{(1-\epsilon_1)R}^R \int_{R}^{(1+\epsilon_1)R} \td J_+(r,\rho) \rd \rho \rd r.
	\end{equation}
	
	By Lemma \ref{lemma2.4},  for any small $\delta>0$ and $\epsilon>0$, we have
			\begin{align}\label{2.11}
	\td J_+(r,\rho)\leq (1+\delta)J_{\epsilon}(r-\rho) \mbox{ for $\rho\in [r/2, (1+\delta^2)r]$ and $r\geq (\delta\epsilon)^{-1}$.}
	\end{align}
	Therefore, if $\epsilon_1>0$ is sufficiently small, then \eqref{2.11} holds when 
	\[
	(1-\epsilon_1)R\leq r\leq R\leq \rho\leq (1+\epsilon_1)R \mbox{ and } R= 2(\delta\epsilon)^{-1}.
	\]
	We thus obtain
	\[
	\int_{(1-\epsilon_1)R}^R \int_{R}^{(1+\epsilon_1)R} \td J_+(r,\rho) \rd \rho \rd r
	\leq (1+\delta)\int_{(1-\epsilon_1)R}^R \int_{R}^{(1+\epsilon_1)R} J_\epsilon(r-\rho) \rd \rho \rd r.
	\]
	From \eqref{2.5d} we have
	\[
J_\epsilon(l)=J_*(l)+\xi^*_\epsilon(l) \mbox{ with } \xi^*_\epsilon(l):=\int_{\R^{N-1}} J(|(l, y')|)\zeta_0(|(l, y')|-\epsilon^{-1})] \rd  y'.
\]
	Clearly
	\begin{align*}
	\int_{(1-\epsilon_1)R}^R \int_{R}^{(1+\epsilon_1)R} J_*(r-\rho) \rd \rho \rd r&= \int_{-\epsilon_1 R}^0\int_0^{\epsilon_1 R}J_*(r-\rho)d\rho dr\\
	&\leq \int_{-\infty}^0\int_0^{\infty}J_*(r-\rho)d\rho dr=\int_0^\infty \eta J_*(\eta)d\eta.
	\end{align*}
	It follows that
	\[
	\int_{(1-\epsilon_1)R}^R \int_{R}^{(1+\epsilon_1)R} \td J_+(r,\rho) \rd \rho \rd r
	\leq (1+\delta)\left[ \int_0^\infty \eta J_*(\eta)d\eta+\int_{(1-\epsilon_1)R}^R \int_{R}^{(1+\epsilon_1)R} \xi^*_\epsilon(r-\rho) \rd \rho \rd r\right].
	\]
	Moreover, using $\epsilon^{-1}=\frac 12 \delta R$ and $\xi^*_\epsilon(s)\leq 2 J_*(s)$,\; $\xi^*_\epsilon(s)=0$ for $s\leq \epsilon^{-1}$, we obtain
	\begin{align*}
	&\int_{(1-\epsilon_1)R}^R \int_{R}^{(1+\epsilon_1)R} \xi^*_\epsilon(r-\rho) \rd \rho \rd r=\int_{-\epsilon_1R}^0 \int_{0}^{\epsilon_1R} \xi^*_\epsilon(r-\rho) \rd \rho \rd r\\
	&\leq 2 \int_{-\epsilon_1R}^0 \int_{ \epsilon^{-1}}^{\epsilon_1R-r} J_*(s) \rd s \rd r\leq 2 \epsilon_1 R\int_{\frac 12\delta R}^{\infty}J_*(s) \rd s\\
	&\leq 4\epsilon_1\delta^{-1} \int_{\frac 12 \delta R}^{\infty}s J_*(s) \rd s\to 0 \mbox{ as } R\to\infty.
	\end{align*}
	
	We thus obtain
	\[
	\limsup_{R\to\infty}\int_{(1-\epsilon_1)R}^R \int_{R}^{(1+\epsilon_1)R} \td J_+(r,\rho) \rd \rho \rd r
	\leq (1+\delta)\int_0^\infty \eta J_*(\eta)d\eta.
	\]
	Since $\delta>0$ can be arbitrarily small,  \eqref{step1}  now follows from \eqref{ep1}.
	
	\textbf{Step 2}. We show that 
	\begin{align*}
	\liminf_{R\to\yy}\int_0^R \int_{R}^{\yy} \td J(r,\rho) \rd \rho \rd r\geq  \int_0^\infty \eta J_*(\eta)d\eta,
	\end{align*}
	which then finishes the proof of the lemma. 
	
	From the definition of $\td J(r,\rho)$ and $J_*$, for $R\geq r>0$,
	\begin{align}\label{2.12}
	 \int_{R}^{\yy} \td J(r,\rho) \rd \rho \geq \int_{H_R} J(|x^1_r-y|) \rd y
	=\int_{R}^{\yy} J_*(r-\rho) \rd \rho,
	\end{align}
	where
	\begin{align*}
	x_r^1:=(r,0,..., 0)\in\R^N,\ H_R:=\{y=(y_1,y_2,\cdots,y_N): \ y_1>R,\ y_i\in \R\ {\rm for}\ 2\leq i\leq N \}.
	\end{align*}
	Therefore
	\begin{align*}
	\liminf_{R\to\yy}\int_0^R \int_{R}^{\yy} \td J(r,\rho) \rd \rho \rd r&\geq 
	\lim_{R\to\infty}\int_{0}^{R}\int_{R}^{\yy} J_*(r-\rho) \rd \rho\rd r\\
	&= \lim_{R\to\infty}\int_0^R\eta J_*(\eta)d\eta=\int_0^\infty\eta J_*(\eta)d\eta.
	\end{align*}
	The proof is now complete.
\end{proof}

\begin{theorem}\label{prop1}
	Assume {\rm \textbf{(J)}} holds. Then the following statements are equivalent:
	\begin{itemize}
		\item[{\rm (i)}] {\rm \textbf{(J1)}} holds, namely $\dd \int_0^\infty r^NJ(r)dr<\infty$, \vspace{0.2cm}
		\item[{\rm (ii)}] $\dd\int_{0}^{\yy} J_*(l) l\rd l<\yy$.
		\item[{\rm (iii)}] $\limsup_{R\to\yy}\dd\int_{0}^{R} \int_{R}^{+\yy} \td J(r,\rho)\rd \rho\rd r<\yy$.\vspace{0.2cm}
		\item[{\rm (iv)}] $\limsup_{R\to\yy}\dd\int_{0}^{R} \int_{R}^{+\yy}(r/R)^{N-1} \td J(r,\rho)\rd \rho\rd r<\yy$.
	\end{itemize} 
	Moreover, when  {\rm \textbf{(J1)}} holds, we have
	\[
	\int_{0}^{\yy} J_*(l) l\rd l=\lim_{R\to\yy}\dd\int_{0}^{R} \int_{R}^{+\yy} \td J(r,\rho)\rd \rho\rd r=\lim_{R\to\yy}\dd\int_{0}^{R} \int_{R}^{+\yy} (r/R)^{N-1}\td J(r,\rho)\rd \rho\rd r.
	\]
\end{theorem}

\begin{proof} By \eqref{J-J_*} and Lemma \ref{lemma2.7d} we see that 
	\begin{align*}
	&\textbf{(J1)}\ {\rm holds}  \Longleftrightarrow \int_{0}^{\yy} J_*(l) l\rd l<\yy,
	\nonumber\\
	&\textbf{(J1)}\ {\rm holds} \Longrightarrow \lim_{R\to\yy}\dd\int_{0}^{R} \int_{R}^{+\yy} \td  J(r,\rho)\rd \rho\rd r<\yy,
	\end{align*}
	and if {\rm \textbf{(J1)}} holds, then
	\begin{align}\label{2.13}
	\int_{0}^{\yy} J_*(l) l\rd l
	=\lim_{R\to\yy}\dd\int_{0}^{R} \int_{R}^{+\yy} \td  J(r,\rho)\rd \rho\rd r.
	\end{align}
	To finish the proof of Theorem \ref{prop1}, it remains to prove that 
	\begin{align}\label{2.14}
	\textbf{(J1)}\ {\rm holds} \Longleftarrow \limsup_{R\to\yy}\dd\int_{0}^{R} \int_{R}^{+\yy} \td  J(r,\rho)\rd \rho\rd r<\yy,
	\end{align}
	\begin{align}\label{2.15}
	&\textbf{(J1)}\ {\rm holds} \Longleftrightarrow \limsup_{R\to\yy} \dd\int_{0}^{R} (r/R)^{N-1}\int_{R}^{+\yy} \td  J(r,\rho)\rd \rho\rd r<\yy,
	\end{align}
	and 
	\begin{align}\label{2.16}
	\lim_{R\to\yy}\dd\int_{0}^{R} (r/R)^{N-1}\int_{R}^{+\yy} \td  J(r,\rho)\rd \rho\rd r =\int_{0}^{\yy} J_*(l) l\rd l \ \ {\rm if}\  \textbf{(J1)}\ {\rm holds}.
	\end{align}
	
	We now prove these in three steps.
	
	\textbf{Step 1}.  We prove \eqref{2.14}. 
	
	By \eqref{2.12} and change of order of integration,
	\begin{align*}
	\int_0^R \int_{R}^{\yy} \td J(r,\rho) \rd \rho \rd r\geq\int_{-R}^{0}\int_{0}^{\yy} J_*(r-\rho) \rd \rho\rd r\geq \int_0^R lJ_*(l)dl,
	\end{align*}
	which yields 
	\begin{align*}
	\yy> \limsup_{R\to\yy}\int_0^R \int_{R}^{\yy} \td J(r,\rho) \rd \rho \rd r\geq \int_{0}^{\yy} lJ_*(l)\rd l.
	\end{align*}
	Hence, due to \eqref{J-J_*},  \eqref{2.14} holds. 
	
	\textbf{Step 2}. We prove \eqref{2.15}. 
	
	If {\rm \textbf{(J1)}} holds, then
	\begin{align}\label{2.17}
	\limsup_{R\to\yy}\dd\int_{0}^{R} (r/R)^{N-1}\int_{R}^{+\yy} \td J(r,\rho)\rd \rho\rd r\leq 
	\limsup_{R\to\yy}\dd\int_{0}^{R} \int_{R}^{+\yy} \td J(r,\rho)\rd \rho\rd r<\yy.
	\end{align}
	
	On the other hand, if 
	\begin{align*}
	\limsup_{R\to\yy}\dd\int_{0}^{R} (r/R)^{N-1}\int_{R}^{+\yy} \td J(r,\rho)\rd \rho\rd r<\yy,
	\end{align*}
	then by \eqref{2.12},
	\begin{align*}
	\yy&>\limsup_{R\to\yy}\dd\int_{R/2}^{R} (r/R)^{N-1}\int_{R}^{+\yy} \td J(r,\rho)\rd \rho\rd r\geq 
	\limsup_{R\to\yy}2^{-(N-1)}\dd\int_{R/2}^{R}\int_{R}^{+\yy}   J_*(r-\rho)\rd \rho\rd r\\
	&=\limsup_{R\to\yy}2^{-(N-1)}\dd\int_{0}^{-R/2}\int_{0}^{+\yy}   J_*(r-\rho)\rd \rho\rd r\geq \limsup_{R\to\yy}2^{-(N-1)}\dd\int_0^{R/2}lJ_*(l)dl.
	\end{align*}
	Hence \eqref{2.15} holds. 
	
	\textbf{Step 3}.  We finally prove \eqref{2.16}.
	
	For any given $\epsilon>0$,  we have
	\begin{align*}
	\dd\int_{0}^{R} (r/R)^{N-1}\int_{R}^{+\yy} \td  J(r,\rho)\rd \rho\rd r&\geq \int_{(1-\epsilon)R}^{R} (r/R)^{N-1}\int_{R}^{+\yy} \td  J(r,\rho)\rd \rho\rd r\\
	\geq & (1-\epsilon)^{N-1}\dd\int_{(1-\epsilon)R}^{R} \int_{R}^{+\yy} \td  J(r,\rho)\rd \rho\rd r.
	\end{align*}
	By \eqref{2.12},
	\begin{align*}
	\dd\int_{(1-\epsilon)R}^{R}\int_{R}^{+\yy} \td  J(r,\rho)\rd \rho\rd r&\geq \int_{(1-\epsilon)R}^{R}\int_{R}^{\yy} J_*(r-\rho) \rd \rho\rd r\\
	&=\int_{-\epsilon R}^{0}\int_{0}^{\yy} J_*(r-\rho) \rd \rho\rd r \geq \int_0^{\epsilon R}l J_*(l)dl.
	\end{align*}
	Letting $R\to\yy$, we obtain
	\begin{align*}
	\liminf_{R\to\yy}\dd\int_{0}^{R} (r/R)^{N-1}\int_{R}^{+\yy} \td  J(r,\rho)\rd \rho\rd r\geq&(1-\epsilon)^{N-1} \int_{0}^{\yy} lJ_*(l)\rd l.
	\end{align*}
	Then by the arbitrariness of $\epsilon>0$, we see 
	\begin{align*}
	\liminf_{R\to\yy}\dd\int_{0}^{R} (r/R)^{N-1}\int_{R}^{+\yy} \td  J(r,\rho)\rd \rho\rd r\geq& \int_{0}^{\yy} lJ_*(l)\rd l.
	\end{align*}
	Combining this with \eqref{2.13} and \eqref{2.17} gives 
	\begin{align*}
	\liminf_{R\to\yy}\dd\int_{0}^{R} (r/R)^{N-1}\int_{R}^{+\yy} \td  J(r,\rho)\rd \rho\rd r= \int_{0}^{\yy} lJ_*(l)\rd l.
	\end{align*}
	The proof is now complete.
\end{proof}

\section{Well-poseness and comparison principle}
In this section we prove Theorem \ref{th1.1} and a comparison principle for \eqref{1.4}, where $J$ satisfies {\bf (J)} and $f$ satisfies \eqref{f}.

\subsection{Well-posedness}
With the preparations in Section 2, the existence and uniqueness of a global solution to \eqref{1.4} can be established by the approach in Section 2 of \cite{cdjfa} with minor modifications. 
We explain these in detail below. 

Define
\[
\hat J(\xi, \eta):=\frac 12 \tilde J(|\xi|,|\eta|) \mbox{ for } \xi,\ \eta\in \R.
\]
From the properties of $\tilde J$ we easily see that
$\hat J$ is continuous, 
\[
\hat J(\xi, 0)\equiv 0, \ \hat J(0,\eta)=\frac 12 \omega_N |\eta|^{N-1} J(0),\ \hat J(\xi,\eta)=\frac 12 \omega_N |\eta|^{N-1}[J(0)+o(1)] \mbox{ as } (\xi,\eta)\to (0,0),
\]
and by Lemma \ref{lemma2.2},
\[
\hat J(\xi,\xi)= \omega_{N-1}(2|\xi|)^{2-N}|\xi|\int_0^{2|\xi|}[4\xi^2-s^2]^{(N-3)/2}s^{N-2}J(s)ds>0 \mbox{ when } \xi\not=0.
\]
If we denote by $\hat u$ the even extension of $u$, namely
\[
\hat u(t, \xi):=u(t, |\xi|) \ \mbox{ for } \xi\in\R,
\]
then \eqref{1.4} is equivalent to
\begin{equation}\label{1.4e}
\begin{cases}
\dd  \hat u_t(t,\xi)=d \int_{-h(t)}^{h(t)} \hat J(\xi,\rho)  \hat u(t,\rho)\rd \rho-d \hat u(t,\xi)+f(t,|\xi|,  \hat u), & t>0,\; \xi \in  (-h(t),h(t)),\\
 u(t,\pm h(t))=0, & t>0,\\
\dd [h^N(t)]'= N \mu \dd\int_{-h(t)}^{h(t)} \int_{h(t)}^{+\yy} \hat J(\xi,\rho)|\xi|^{N-1} \hat u(t,\xi)\rd \rho\rd \xi, & t>0,\\
h(0)=h_0,\ u(0,\xi)= u_0(|\xi|), & \xi\in [-h_0,h_0].
\end{cases}
\end{equation}

Problem \eqref{1.4e} is close in form to the one dimensional problem in \cite{cdjfa}, with the following main differences:
\begin{itemize}
\item[(i)] The kernel function $J(|x-y|)$ in \cite{cdjfa} is replaced by $\hat J(\xi,\eta)$,
\item[(ii)] In the third equation, $h'(t)$ in \cite{cdjfa} is replaced by $[h^N(t)]'$, and the kernel function is now $\hat J(\xi,\eta)|\xi|^{N-1}$.
\end{itemize}
A close examination of the proof in \cite{cdjfa} for the existence and uniqueness results there shows that all the arguments  carry over
to \eqref{1.4e}, with only minor changes required. We indicate below the main steps and the changes needed.

In place of Lemma 2.2 in \cite{cdjfa}, we have the following result.

\begin{lemma}[Maximum principle]\label{lemma4.4} Let $T>0$, $d>0$,  and $g, h\in C([0,T])$ satisfy 
 $g(0)\leq h(0)$ and $g(t)<h(t)$ for $t\in (0, T]$.
	Denote $D_T:=\{(t,x): t\in (0, T],\; g(t)<x<h(t)\}$ and suppose that  $\phi$, $\phi_t\in C(\ol D_T)$, $c\in L^\yy(D_T)$,  and
\begin{equation*}
\begin{cases}
\dd \phi_t\geq d\int_{g(t)}^{h(t)}P(x,y)\phi(t,y) \rd y+c(t,x)\phi, & (t,x)\in D_T,\; \\
\phi(t,g(t))\geq 0, & t\in\Sigma_{\rm min}^g,\\
 \phi(t,h(t))\geq 0,  & t \in\Sigma_{\rm max}^h, \\
\phi(0,x)\geq 0, &x\in [g(0),h(0)],
\end{cases}
\end{equation*}
where 
\[\begin{cases}
\Sigma_{\rm min}^g=\left\{ t\in(0,T]: \mbox{There exists $\epsilon>0$ such that $g(t)<g(s)$ for $s\in [t-\epsilon, t)$}\right\},\\
\Sigma_{\rm max}^h=\left\{ t\in(0,T]: \mbox{There exists $\epsilon>0$ such that $h(t)>h(s)$ for $s\in [t-\epsilon, t)$}\right\},
\end{cases}
\]
and the kernel function $P$ satisfies
\begin{align*}
P\in C( \R^2)\cap L^\infty(\R^2),\ P\geq 0,\ P(x,x)>0 \mbox{ for almost all } x\in \R.
\end{align*}
Then  $\phi\geq 0$ on $\ol D_T$, and if additionally  $\phi(0,x)\not\equiv0$ in $[g(0),h(0)]$, then $\phi> 0$ in $D_T$. 
\end{lemma}
\begin{proof}
{\bf Case 1}:  $g(0)<h(0)$. Noting that $P(x,x)>0$ for $x\in \R$,  we can repeat the arguments in  \cite[Lemma 3.1]{dn2020} to show the desired conclusion.

{\bf Case 2}:  $g(0)=h(0)$ and $\phi(0,g(0))>0$.   By the continuity of $\phi$, there is $t_1>0$ such that 
\begin{align*}
\phi(t,x)>0\ \mbox{ for } \ t\in [0,t_1],\ x\in [g(t),h(t)].
\end{align*}  
Then viewing $t_1$ as the initial time, we obtain the desired  conclusion from Case 1. 

{\bf Case 3}: $g(0)=h(0)$ and $\phi(0,g(0))=0$. Let $\psi(t,x)=\phi(t,x)+\epsilon e^{At}$ for some positive constants $\epsilon$ and $A$. Then 
\begin{align*}
&\dd \psi_t(t,x)-d\int_{g(t)}^{h(t)}P(x,y)\psi(t,y) \rd y-c(t,x)\psi\\
&=\dd \phi_t(t,x)-d\int_{g(t)}^{h(t)}P(x,y)\phi(t,y) \rd y -c(t,x)\phi\\
&\ \ \ \ \ +\epsilon Ae^{At}-d\epsilon e^{At}\int_{g(t)}^{h(t)}P(x,y)\rd y-c(t,x)\epsilon e^{At}\\
&\geq \Big(A-d[h(t)-g(t)]\|P\|_{L^\yy(\R^2)}-c(t,x))\epsilon e^{At}>0\ \   {\rm for}\ (t,x)\in  D_T,
\end{align*}
provided that $A\geq d\max_{t\in [0,T]}[h(t)-g(t)]\|P\|_{L^\yy(\R^2)}+\|c\|_{L^\yy(D_T)}$. Since $\psi(0,g(0))=\psi(0,h(0))>0$, by the conclusion in Case 2, we see
\begin{align*}
\psi(t,x)> 0 \ \mbox{ for }\ \ (t,x)\in \ol D_T,
\end{align*}
which yields, by letting $\epsilon\to 0$, 
\begin{align*}
\phi(t,x)\geq 0\  \mbox{ for }\ \ (t,x)\in \ol D_T. 
\end{align*}
The proof is complete.
\end{proof}

The next result is the corresponding version of Lemma 2.3 in \cite{cdjfa}.

 \begin{lemma}\label{lem3.2}
Suppose that {\rm \bf (J)} and \eqref{f} hold, $h_0>0$ and $u_0$ satisfies \eqref{u_0}.
 Then for any $T>0$ and $h\in\mathbb H_{h_0, T}:=\Big\{h\in C([0,T])~:~h(0)=h_0,
\; \inf_{0\le t_1<t_2\le T}\frac{h(t_2)-h(t_1)}{t_2-t_1}>0\Big\}$,
 the following problem
\begin{equation}
\left\{
\begin{aligned}
&v_t=d\int_{-h(t)}^{h(t)}\hat J(r,\rho)v(t,\rho)d\rho-dv+f(t,|r|,v),
& &0<t< T,~r\in(-h(t),h(t)),\\
&v(t,\pm h(t))=0,& &0<t< T,\\
&v(0,r)=u_0(|r|),& &r\in[-h_0,h_0]
\end{aligned}
\right.
\label{201}
\end{equation}
admits a unique  solution, denoted by $V_{h}(t,r)$.
 Moreover $V_{h}$ satisfies
\begin{equation}
0<V_{h}(t,r)\le \max\left\{\max_{-h_0\le r\le h_0}u_0(|r|),
~K_0\right\}~\mbox{ for } 0<t< T,~r\in(-h(t),h(t)),
\label{v-bound}
\end{equation}
where $K_0$ is defined in  assumption \eqref{f}.
\end{lemma}
\begin{proof}
This is almost identical to the proof of Lemma 2.3 in \cite{cdjfa}; we omit the details.
\end{proof}

The following theorem shows that \eqref{1.4} is wellposed, which clearly implies Theorem \ref{th1.1}.

\begin{theorem}
Suppose that {\rm \bf (J)} and \eqref{f} hold. Then for any given $h_0>0$ and $u_0$
satisfying \eqref{u_0}, problem \eqref{1.4e} admits a unique
 solution $(u(t,r), h(t))$ defined for all $t>0$. Moreover, for any $T>0$, $
 h\in\mathbb H_{h_0, T}$ and $u\in \mathbb{X}_{u_0, h}:=\Big\{\phi\in C(\overline\Omega_h)~:~\phi\ge0~\text{in}
~\Omega_h,~\phi(0,r)=u_0(|r|)~\text{for}~r\in [-h_0,h_0]\Big\}$, where $\Omega_h:=\left\{(t,r)\in\mathbb{R}^2: 0<t\leq T,~-h(t)<r<h(t)\right\}$.
\label{thm2.3}
\end{theorem}
\begin{proof}
By Lemma \ref{lem3.2},
for any $T>0$ and $h\in\mathbb H_{h_0,T}$, we can find a
unique $V_{h}\in\mathbb{X}_{u_0,h}$ that solves (\ref{1.4e}), and it has the property
\begin{equation}\label{M_0}
0<V_{h}(t,r)\le M_0:=\max\big\{\|u_0\|_\infty,~K_0\big\} \mbox{ for  } (t,r)\in
\Omega_{h}.
\end{equation}

Using such a $V_{h}(t,r)$, we define the mapping $\hat \Gamma$ by 
\begin{equation}
\left\{
\begin{aligned}
&\hat \Gamma(h):=(\tilde h)^{1/N} \ \mbox{ for } h\in \mathbb H_{h_0,T}, \mbox{ with }\\
&\tilde h(t):= h_0^N+N\mu\int_0^t\int_{-h(\tau)}^{h(\tau)}\int_{h(\tau)
}^{+\infty}\hat J(\xi,\rho)|\xi|^{N-1}V_{h}(\tau,\rho)d\rho d\xi d\tau \ \ \mbox{ for $0<t\leq T$. }
\end{aligned}
\right.
\label{2003}
\end{equation}

To prove this theorem, we will show
that if $T$ is small enough, then $\hat\Gamma$ maps  a suitable closed subset $\Sigma_T$ of $\mathbb{H}_{h_0,T}$ into itself,
and  is a contraction mapping. This clearly implies that $\hat \Gamma$ has a unique fixed point in $\Sigma_T$,
which gives a solution $(V_{h},  h)$ of \eqref{1.4e} defined for $t\in (0, T]$. We will show that any solution $(u, h)$ of \eqref{1.4e}  with  $h\in \mathbb{H}_{h_0,T}$ must satisfy $h\in \Sigma_T$, and hence $h$ must coincide with the unique fixed point of $\hat\Gamma$ in $\Sigma_T$, which then implies that the solution $(u,h)$ of \eqref{1.4e} is unique.  We will finally  show that this unique solution
defined locally in time can be extended uniquely for all $t>0$.

This plan will be carried out in 4 steps, as in the proof of Theorem 2.1 in \cite{cdjfa}.
\medskip

\noindent
{\bf Step 1:} {\it Properties of $ \tilde h(t)$ and a closed subset of $\mathbb{H}_{h_0,T}$.}

Let $h\in\mathbb{H}_{h_0,T}$.
The definition of $\tilde h(t)$  indicates that it belongs to $C^1([0, T])$  and for $0<t\le T$,
\begin{equation}
\tilde h'(t)= N\mu \int_{-h(t)}^{h(t)}\int_{h(t)}^{+\infty}\hat J(\xi,\rho )|\xi|^{N-1}V_{h}(t,\rho)d\rho d\xi.
\label{2006}
\end{equation}
From this and the definition of $\td h$ we see that $\hat\Gamma(h)=(\tilde h)^{1/N}\in  \mathbb{H}_{h_0,T}$, but in order to show $\hat \Gamma$ is a contraction mapping, we need to
prove some further properties of  $\tilde h$, and then choose a suitable closed subset of  $\mathbb{H}_{h_0,T}$, which
is invariant under $\hat \Gamma$, and on which $\hat \Gamma$ is a contraction mapping.

Since  $v=V_{h}$ solves (\ref{201})  we  obtain by using \eqref{f} and  \eqref{M_0} that
\begin{equation}
\left\{
\begin{aligned}
&\left(V_{h}\right)_t(t,r)\ge-dV_{h}(t,r)-K(M_0)V_{h}(t,r), & &0<t\le T,~r\in(-h(t),h(t)),\\
&V_{h}(t,\pm h(t))=0,& &0<t\le T,\\
&V_{h}(0,r)=u_0(|r|),& &r\in[-h_0,h_0].
\end{aligned}
\right.
\label{2007}
\end{equation}
It follows that
\begin{equation}
\label{V>cu_0}
V_{h}(t,r)\ge e^{-(d+K(M_0))t}u_0(|r|)\ge e^{-(d+K(M_0))T}u_0(|r|) \mbox{ for } r\in [-h_0, h_0],\; t\in (0, T].
\end{equation}
By the properties of $\hat J$  there
exist constants $\epsilon_0\in (0, h_0/4)$ and $\delta_0>0$ such that
\begin{equation}
\label{J>delta_0}
\mbox{$\hat J(\xi,\rho)|\xi|^{N-1}
\ge\delta_0$ if $|\xi-\rho|\le\epsilon_0$ and $\xi,\, \rho\in [h_0-\frac{\epsilon_0}2, h_0+\frac{\epsilon_0}2]$.}
\end{equation}
Using \eqref{2006} we easily see
\[
0< \tilde h'(t)\leq N\mu M_0 h(t)^N \mbox{ for } t\in [0, T].
\]
Assume  that $h$ has the extra property that
\[\mbox{
    $h(T)\leq h_0+\frac{\epsilon_0}{4}$.}
    \]
    Then
\[
\tilde h(t)\leq h^N_0 +TN\mu M_0(h_0+\frac{\epsilon_0}{2})\leq \Big[h_0+\frac{\epsilon_0}{4}\Big]^N \mbox{ for } t\in [0, T],
\]
provided that $T>0$ is small enough, depending on $ (\mu, M_0, h_0, \epsilon_0)$.
We fix such a  $T$ and notice   that
\[
h(t)\in [h_0,  h_0+ \frac{\epsilon_0}{4}] \mbox{ for } t\in [0, T].
\]
Combining this with \eqref{V>cu_0} and \eqref{J>delta_0} we obtain, for such $T$ and $t\in (0, T]$,
\begin{align*}
\int_{-h(t)}^{h(t)}\int_{h(t)}^{+\infty}\hat J(\xi,\rho)|\xi|^{N-1}V_{h}
(t,\rho)d\rho d\xi &\ge\int_{h(t)-\frac{\epsilon_0}{2}}^{h(t)}\int_{
h(t)}^{h(t)+\frac{\epsilon_0}{2}}\hat J(\xi,\rho)|\xi|^{N-1}V_{h}
(t,\rho)d\rho d\xi \\
&\ge e^{-(d+K(M_0))T}\int_{h_0-\frac{\epsilon_0}{4}}^{h_0}\int_{
h_0+\frac{\epsilon_0}{4}}^{h_0+\frac{\epsilon_0}{2}}\hat J(\xi,\rho)|\xi|^{N-1} u_0(|\rho|)d\rho d\xi \\
&\ge\frac 14\epsilon_0\delta_0e^{-(d+K(M_0))T}\int_{h_0-\frac{
\epsilon_0}{4}}^{h_0}u_0(\rho)d\rho=:c_0>0,
\end{align*}
with $c_0$ depending only on $(J, u_0, f)$. Thus, for sufficiently small $T=T(\mu, M_0, h_0, \epsilon_0)>0$,
\begin{equation*}
\label{tilde-h'}
\tilde h'(t)\geq N\mu c_0 \mbox{ for } t\in [0, T].
\end{equation*}
Therefore
\begin{equation}
\label{hat-h'}
\frac{d}{dt}\hat \Gamma(h)(t)=\frac 1N \tilde h^{\frac{1-N}N}(t)\tilde h'(t)\geq  \td c_0:=\left(h_0+\frac\epsilon 2\right)^{1-N} \mu c_0 \mbox{ for } t\in [0, T].
\end{equation}

We now define, for $s\in (0, T_0]:=(0, T(\mu, M_0, h_0, \epsilon_0)]$,
\begin{align*}
\Sigma_s:=&\Big\{\hat h\in \mathbb H_{h_0,s}: \sup_{0\leq t_1<t_2\leq s}\frac{\hat h(t_2)-\hat h(t_1)}{t_2-t_1}\geq  \tilde c_0,\;
 \hat h(t)\leq h_0+\frac{\epsilon_0}{4} \mbox{ for } t\in [0, s]\Big\}.
\end{align*}
Our analysis above shows that
\[
\hat \Gamma (\Sigma_s)\subset \Sigma_s \mbox{ for } s\in (0, T_0].
\]

\medskip

\noindent {\bf Step 2:} {\it $\hat\Gamma$ is a
contraction mapping on $\Sigma_s$ for sufficiently small $s>0$.}

Let $s\in (0, T_0]$, $h_1, h_2\in\Sigma_s$,
 and note that $\Sigma_s$ is a complete metric space
under the metric
$$
d\left(h_1,h_2)\right)=\|h_1-h_2\|_{C([0,s])}.
$$
The analysis in Step 2 of the proof of Lemma 2.3 in \cite{cdjfa} can be repeated here to show that, for any $h_1, h_2\in \Sigma_s$,
\begin{align*}
~\|\tilde h_1-\tilde h_2\|_{C([0,s])}\le~  Cs\|h_1-h_2\|_{C([0,s])} \mbox{ for some $C>0$ independent of $h_1$ and $h_2$.}
\end{align*}
Hence
\begin{align*}
\|\hat\Gamma(h_1)-\hat\Gamma(h_2)\|_{C([0,s])}&\le \frac 1{N h_0^{(N-1)/N}} \|\tilde h_1-\tilde h_2\|_{C([0,s])}\\
&\le  \frac {Cs}{N h_0^{(N-1)/N}}\|h_1-h_2\|_{C([0,s])}\leq \frac 12 \|h_1-h_2\|_{C([0,s])},
\end{align*}
provided that $s>0$ is sufficiently small, say $s \in (0, T^*]$. Therefore $\hat\Gamma$ is a contraction mapping on $\Sigma_s$ for such $s$.

\medskip

\noindent
{\bf Step 3:} {\it Local existence and uniqueness.}

By Step 2 and the Contraction Mapping Theorem we know that \eqref{1.4e} has a solution $(u, h)$ for $t\in (0, T^*]$. If we can show that
$h\in \Sigma_{T^*}$  holds for any solution $(u,h)$ of \eqref{1.4e} defined over $t\in (0, T^*]$, then it is the unique fixed point of $\hat \Gamma$ in $\Sigma_{T^*}$ and the uniqueness of $(u,h)$ follows.

So let $(u,h)$ be an arbitrary solution of \eqref{1.4e} defined for $t\in (0, T^*]$. Then
\[
[h^N(t)]'=N\mu\int_{-h(t)}^{h(t)}\int_{h(t)}^{\infty}\hat J(\xi,\rho)|\xi|^{N-1}  u(t,\xi)\rd \rho\rd \xi \leq 2\mu M_0 h^N (t) \mbox{ for } t\in (0, T^*].
\]
We thus obtain
\begin{equation}
\label{h-g}
h^N(t)\leq h^N_0 e^{2\mu M_0 t} \mbox{ for } t\in (0, T^*].
\end{equation}
Therefore if we shrink $T^*$  if necessary so that
\[
h_0e^{\mu M_0 T^*/N}\leq h_0+\frac{\epsilon_0}{4},
\]
then
\[
h(t)\leq h_0+\frac{\epsilon_0}{4} \mbox{ for } t\in [0, T^*].
\]
Moreover, the proof of \eqref{hat-h'}  gives
\[
h'(t)\geq  \tilde c_0 \mbox{ for } t\in (0, T^*].
\]
Thus indeed $h\in\Sigma_{T^*}$, as we wanted.
This proves the local existence and uniqueness of the  solution to \eqref{1.4e}.

\medskip

\noindent
{\bf Step 4:} {\it Global existence and uniqueness.}

This is identical to the corresponding proof in \cite{cdjfa}, and we omit the details here.
\end{proof}
\medskip

\subsection{Comparison principle}
We now use Lemma \ref{lemma4.4} to obtain a comparison principle which will be useful for our later analysis.

\begin{lemma}[Comparison principle]\label{lemma3.4a}    Suppose  $(\mathbf{J})$ and \eqref{f}  hold,  and $(u,h)$ solves \eqref{1.4} for $t\in [0, T]$ with some $T>0$. For convenience we extend $ u$ by $ u(t,r)=0$ for $t\in[0, T]$ and $r> h(t)$.  Let  $r_*, h_*\in C([0,T])$ be nondecreasing functions  satisfying $0\leq r_*(t)<h_*(t)$, and 
	\begin{align*}
	\Omega_T:=\{(t,r): t\in (0,T], r \in(0,h_*(t))\},\ 	\Theta_T:=\{(t,r): t\in (0,T], r\in(r_*(t),h_*(t))\}.
	\end{align*}
 Suppose   $v\in C(\ol\Omega_T)$ is nonnegative with $v_t\in C(\ol\Theta_T)$, and
		\[
		\hat v(t,r):=\begin{cases}u(t,r) \mbox{ for }  r\in [0, r_*(t)],\ t\in [0,T], \\ 
		v(t,r) \mbox{ for }  r\in (r_*(t), h_*(t)],\ t\in [0,T].
		\end{cases}
		\]
	\begin{itemize}
		\item[{\rm (i)}]     If $(v,r_*,h_*)$ satisfy $h_*(0)\geq h(0)$, 
		\begin{equation}\label{3.8d}
		\begin{cases}
		v(0,r)\geq  u(0,r), \ \ r\in [0,h_*(0)],\\
		v(t,r)\geq   u(t,r),\ \  \  t\in [0,T],\; r\in [0,r_*(t)]
		\end{cases}
		\end{equation}
		and
		\begin{equation*}
		\begin{cases}
		v_t\geq   d \Big[\dd\int_{0}^{h_*(t)}\td J(r,\rho)\hat v(t,\rho)\rd \rho-v(t,r)\Big]+f(t,r,v), &t\in (0,T],\; r\in {(r_*(t),h_*(t))},\\[2mm]
		v(t,h_*(t))\geq  0, & t\in (0,T],\\[2mm]
		\dd h_*'(t)\geq  \frac{\mu}{h_*^{N-1}(t)}\dd\int_{0}^{h_*(t)}\int_{h_*(t)}^{+\yy} \td J(r,\rho) r^{N-1}v(t,r)\rd \rho\rd r, & t\in [0,T],
		\end{cases}
		\end{equation*}
				then  
		\begin{align*}
		 h_*(t)\geq h(t),\ \ v(t,r)\geq  u(t,r) \ \ \ \ \ {\rm for}\  t\in (0, T],\; r\in [0, h(t)].
		\end{align*}
		
		\item[{\rm (ii)}]   If $(v,r_*,h_*)$   satisfy $h_*(0)\leq h(0)$,
		\begin{align*}\begin{cases}
		v(0,r)\leq  u(0,r), \ \ r\in [0,h(0)],\\
		v(t,r)\leq   u(t,r),\ \  \  t\in [0,T],\; r\in [0,r_*(t)]
		\end{cases}
		\end{align*}
		and
		\begin{equation*}
		\begin{cases}
		v_t\leq   d \Big[\dd\int_{0}^{h_*(t)}\td J(r,\rho) \hat v(t,\rho)\rd \rho-v(t,r)\Big]+f(t,r,v), &t\in (0,T],\; r\in {(r_*(t),h_*(t))},\\[2mm]
		v(t,h_*(t))\leq  0, & t\in (0,T],\\[2mm]
		\dd h_*'(t)\leq  \frac{\mu}{h_*^{N-1}(t)}\dd\int_{0}^{h_*(t)} \int_{h_*(t)}^{+\yy} J(r,\rho)r^{N-1}v(t,r)\rd \rho\rd r, & t\in [0,T],
				\end{cases}
		\end{equation*}
		then  
		\begin{align*}
		h_*(t)\leq h(t),\ \ v(t,r)\leq  u(t,r) \ \  {\rm for }\  t\in (0, T],\; r\in [0, h_*(t)].
		\end{align*}
	\end{itemize}
\end{lemma}
\begin{proof} We just prove (i) since the proof of (ii) is similar.  For small $\epsilon>0$, 
	let $h_0^\epsilon:=(1-\epsilon)h_0$, $\mu_\epsilon:=(1-\epsilon)\mu$ and $u_0^\epsilon\in C([0,h_0])$ be a function satisfying
	 $0\leq u_0^\epsilon(r)< u(0,r)$ for $r\in [0,h_0^\epsilon]$,  	 
	 $u_0^\epsilon(r)=0$ for $r \in [h_0^\epsilon,h_0]$ and $\lim_{\epsilon\to 0}u_0^\epsilon(\cdot)= u(0,\cdot)$  in $C([0,h_0])$. Denote by  $(u_\epsilon,h_\epsilon)$  the unique solution of \eqref{1.4} with $h_0$ replaced by $h_0^\epsilon$,  $\mu$ replaced by  $\mu_\epsilon$ and $ u(0,\cdot)$ replaced by $u_0^\epsilon$. 
	 
	 We next show that 
	 \begin{align}\label{3.9d}
	 h_\epsilon(t)\leq h_*(t) \ \mbox{ for } \   t\in [0,T].
	 \end{align}
Due to  $h_\epsilon(0)=(1-\epsilon)h_0<h_0\leq h_*(0)$,    
	\begin{align*}
	t_1:=\max\{t\in (0,T]:h_\epsilon(s)<h_*(t)\ {\rm for\ all}\ s\in [0,t]\}
	\end{align*}
	is well defined. 
	If $t_1=T$, then \eqref{3.9d} immediately holds. On the other hand, if $t_1<T$, then 
	\begin{align}\label{3.10d}
	h_\epsilon(t_1)=h_*(t_1), \  h_\epsilon'(t_1)\geq h_*'(t_1),\ h_\epsilon(t)<h_*(t)\ \mbox{ for } \ t\in [0,t_1),
	\end{align}
	and	from $r_*(t)<h_*(t)$ for $t\in [0,T]$  there are two possible cases:
\begin{itemize}
	\item[{\rm (a)}]  $h_\epsilon(t)\in (r_*(t),h_*(t)]$ for $t\in [0,T]$,
	\item[{\rm (b)}] there exits  $t_2\in [0,t_1)$ such that 
	\begin{align*}
	r_*(t_2)=h_\epsilon(t_2) \  {\rm and}\  r_*(t)<h_\epsilon(t) \ \mbox{ for} \ t\in (t_2,t_1].
	\end{align*}
\end{itemize}
We thus always have 
\begin{align*}
	r_*(t_2)\leq  h_\epsilon(t_2)\  {\rm and}\  r_*(t)<h_\epsilon(t) \ \mbox{ for  all } \ t\in (t_2,t_1] \mbox{ and some $t_2\in [0,t_1)$.}
\end{align*}
  
  Note that for $t\in (t_2,t_1]$ and $r\in (r_*(t),h_\epsilon(t))\subset (r_*(t), h_*(t))$,
\begin{align*}
&v_t(t,r)\geq   d \lf[\int_{r_*(t)}^{ h_\epsilon(t)}\td J(r,\rho) v(t,\rho)\rd \rho-v(t,r)\rr]+d\Lambda(t,r)+f(t,r,v),\\
&\dd  (u_\epsilon)_t(t,r)=d\lf[\int_{r_*(t)}^{ h_\epsilon(t)} \td J(r,\rho)  u_\epsilon(t,\rho)\rd \rho- u_\epsilon(t,r)\rr]	+ d\Lambda(t,r)+f(t,r, u_\epsilon)
\end{align*}
with $\Lambda(t,r):=\dd\int_{0}^{r_*(t)}\td J(t,\rho)  u_\epsilon(t,\rho) \rd \rho$. Hence for $w:=v- u_\epsilon$, we have
\begin{align*}
w_t(t,r)\geq   d \lf[\int_{r_*(t)}^{ h_\epsilon(t)}\td J(r,\rho) w(t,\rho)\rd \rho-w(t,r)\rr]+c(t,r) w(t,r)\ \mbox{ for }  t\in (t_2,t_1],\; r\in (r_*(t),h_\epsilon(t))
\end{align*}
where 
\begin{equation*}
c(t,r):=
\begin{cases}
0,& v(t,r)= u_\epsilon(t,r),\\
\dd \frac{f(t,r, v(t,r))-f(t,r, u_\epsilon (t,r))}{v(t,r)- u_\epsilon(t,r)}, &v(t,r)\neq  u_\epsilon(t,r),
\end{cases}
\end{equation*}
is a bounded function. Besides, by our assumptions,  we have
\begin{align*}
&w(t_2,r)=w(0,r)=v(0,r)- u_\epsilon(0,r)\geq 0\ {\rm for}\ r\in [r_*(0),h_\epsilon(0)]&&  {\rm if\ Case\ (a)\ holds }, \\
&w(t_2,r)=w(t_2,r_*(t_2))=v(t_1,r_*(t_2))\geq 0\ {\rm for}\ r\in [r_*(t_2),h_\epsilon(t_2)]=\{r_*(t_2)\}&& {\rm if\ Case\ (b)\ holds},
\end{align*}
  and
\begin{align*}
w(t,r_*(t))\geq 0, \
w(t,h_\epsilon(t))=v(t,h_\epsilon(t))- u_\epsilon(t,h_\epsilon(t))=v(t,h_\epsilon(t))\geq 0\ \mbox{ for }\ t\in [t_2,t_1]. 
\end{align*}
Therefore,  we can use  Lemma \ref{lemma4.4} to conclude that 
	\begin{align*}
	 u_\epsilon(t,r)\leq  v(t,r) \ \mbox{ for } \ t\in  [t_2,t_1],\ r\in [r_*(t), h_\epsilon(t)].
	\end{align*}
	This combined with \eqref{3.8d} gives
	\begin{align*}
	 u_\epsilon (t,r)\leq   v(t,r)\ \mbox{ for } \ t\geq [t_2,t_1],\ r\in [ 0, h_\epsilon (t)].
	\end{align*}
	Thus 
	\begin{align*}
	\dd h_*'(t_1)&\geq \frac{\mu}{h_*^{N-1}(t_1)}\dd\int_{0}^{h_* (t_1)} \int_{h_*(t_1)}^{+\yy} \td J(r,\rho)r^{N-1}v(t_1,r)\rd \rho\rd r\\
	&=\frac{\mu}{h_\epsilon^{N-1}(t_1)}\dd\int_{0}^{ h_\epsilon(t_1)} \int_{ h_\epsilon(t_1)}^{+\yy} \td J(r,\rho)r^{N-1}v(t_1,r)\rd \rho\rd r\\
	&>\frac{\mu_\epsilon}{h_\epsilon^{N-1}(t_1)}\dd\int_{0}^{ h_\epsilon(t_1)} \int_{ h_\epsilon(t_1)}^{+\yy} \td J(r,\rho)r^{N-1}u_\epsilon (t_1,r)\rd \rho\rd r
=h_\epsilon'(t_1),
	\end{align*}
	which contradicts with \eqref{3.10d}. Hence $t_1=T$, and \eqref{3.9d} holds.

	 Since the unique solution of \eqref{1.4}   depends continuously on the parameters in \eqref{1.4}, the desired result then follows by letting $\epsilon\to0$.
\end{proof}

\begin{remark}\label{rmk2.6}  In Lemma \ref{lemma3.4a}, if $r_*(t)\equiv 0$, then the conclusions hold without requiring
\[\begin{cases}
v(t,r)\geq u(t,r) \mbox{ for $t\in [0, T],\ r\in [0, r_*(t)]=\{0\}$ in part $(i)$},\\
v(t,r)\leq u(t,r) \mbox{ for $t\in [0, T],\ r\in [0, r_*(t)]=\{0\}$ in part $(ii)$}.
\end{cases}
\]
\end{remark}
\begin{proof}
When $r_*(t)\equiv 0$, $\Sigma_{\rm min}^{r_*}=\emptyset$, and the conclusion follows directly from the proof of Lemma \ref{lemma3.4a} when Lemma \ref{lemma4.4} is used for $w$
over $t\in [t_2, t_1]$ and $r\in [r_*(t), h_\epsilon(t)]$.
\end{proof}

\begin{remark}\label{rmk3.6} In Lemma \ref{lemma3.4a}, 
if $v_t(t,r)$ has  jumping discontinuities over $r=\sigma_j(t)$ $(1\leq j\leq m)$, with $\sigma_1(t)<\sigma_2(t)<...<\sigma_m(t)$ continuous functions of $t$, and the inequalities involving $v_t(t,r)$
are satisfied away from these curves, then the conclusions remain valid.  A reasoning for this can be found in Remark 2.4 of \cite{dn-speed} for a similar situation.
This observation also applies to Lemma \ref{lemma4.4}.
\end{remark}

The rest of this paper will focus on \eqref{1.4} with a Fisher-KPP nonlinearity, namely $f=f(u)$ which satisfies {\bf (f)}.

\section{Spreading-vanishing dichotomy}
The purpose of this section is to prove Theorems \ref{th1.2} and \ref{th1.3}, which follows the approach of \cite{cdjfa} where the one space dimension case was treated.
\subsection{Some preparations} Let $a, d$ be positive constants, $J(r)$ a kernel function satisfying ({\bf J}), and $\Omega\subset\R^N$ a bounded domain. 
We first consider the eigenvalue problem 
\begin{align}\label{7.1}
d\int_{\Omega}J(|x-y|) \phi(y) \rd y-d\phi(x)+a\phi(x)=\lambda \phi(x),\; x\in\Omega.
\end{align}
  By  \cite{bcv2016} and \cite{lcw2017},  problem \eqref{7.1} has a principal eigenvalue $\lambda_1(\Omega)$ equipped with a positive  eigenfunction $\phi_\Omega$. 

\begin{proposition}\label{proposition3.1}
	Assume {\rm (\textbf{J})} holds. Then the following statements are true.
	\begin{itemize}
		\item[{\rm (a)}] $\lambda_1(\Omega_1)\leq \lambda_1(\Omega_2)$ if $\Omega_1\subset \Omega_2$. 
		\item[{\rm (b)}] Denote $\lambda_1(L):=\lambda_1(B_L)$ with $B_L:=\{x\in\R^N: |x|<L\}$. Then
			\begin{itemize}
			\item[{\rm (i)}] $\lambda_1(L)$ is strictly increasing and continuous with respect to $L\in (0,\yy)$,
			\item[{\rm (ii)}] $\lim_{L\to\yy} \lambda_1(L)=a$,
			\item[{\rm (iii)}]  $\lim_{L\to0} \lambda_1(L)=a-d$.
		\end{itemize}
	\end{itemize}

\end{proposition}
\begin{proof}  The conclusions in (a) and part (i) of (b) follows similarly as \cite[Proposition 3.4]{cdjfa} part (i). It remains to prove (b)(ii) and (b)(iii).
\smallskip
	
	(b)(ii).  From the variational characterization of $\lambda_1(L)$ (see, e.g., \cite{bcv2016}), we have
	\begin{align}\label{7.2}
	\dd\lambda_1(L)=\sup_{0\not\equiv\phi\in L^2(B_L)} \frac{\dd d\int_{B_L}\int_{B_L} J(|x-y|)\phi(x)\phi(y)\rd y\rd x}{\dd\int_{B_L}\phi^2(x)\rd x}-d+a.
	\end{align}
	Since
	\begin{align*}
	\dd \int_{B_L}\int_{B_L} J(|x-y|)\phi(x)\phi(y)\rd y\rd x\leq& \int_{B_L}\int_{B_L} J(|x-y|)\frac{\phi^2(x)+\phi^2(y)}{2}\rd y\rd x\\
	\leq&  \int_{B_L}\phi^2(x)\rd x,
	\end{align*}
	we immediately get
	\begin{align*}
	\lambda_1(L)\leq d-d+a=a. 
	\end{align*}
	
	For small $\epsilon>0$, it follows from $\int_{\R^N}J(|x|) \rd x=1$ that  there exists $L_\epsilon$ such that for $L\geq L_\epsilon$,
	\begin{align*}
	\int_{B_L} J(|x|) \rd x\geq 1-\epsilon. 
	\end{align*}
	It is clear that for $L\geq 2L_\epsilon$, 
	\begin{align*}
	B_{L_\epsilon}(0)\subset B_L(x) \ \ \forall x\in B_{L-L_\epsilon}(0).
	\end{align*}
	Then using  \eqref{7.2} with $\phi\equiv 1$, we deduce
	\begin{align*}
	\dd\lambda_1(L)& \geq \dd d\, |B_L|^{-1}\int_{B_L}\int_{B_L} J(|x-y|)\rd y\rd x-d+a\\
	&\geq \dd d \, |B_L|^{-1} \int_{B_{L-L_\epsilon}}\int_{B_L(x)} J(|y|)\rd y\rd x -d+a\\
	&\geq \dd d\, |B_L|^{-1} \int_{B_{L-L_\epsilon}}\int_{B_{L_\epsilon}(0)} J(|y|)\rd y\rd x-d+a\\
	&\geq d(1-\epsilon){|B_{L-L_\epsilon}|}{|B_L|^{-1}}-d+a\\
	&=d(1-\epsilon)\frac{|L-L_\epsilon|^N}{L^N}-d+a\to -\epsilon d+a \ \ \ \ {\rm as }\ L\to\yy,
	\end{align*}
	which implies 
	\begin{align*}
	\liminf_{L\to\yy}\lambda_1(L)\geq -\epsilon d+a. 
	\end{align*}
	The arbitrariness of $\epsilon$ yields $\liminf_{L\to\yy}\lambda_1(L)\geq a$, which combined with $\lambda_1(L)\leq a$ gives 
	\[
	\lim_{L\to\yy}\lambda_1(L)= a.\]
	
	(b)(iii). Since the supremum in \eqref{7.2} is attained by $\phi_{B_L}$, we deduce
	\begin{align*}
	|\dd\lambda_1(L)+d-a|     &= \frac{\dd d\int_{B_L}\int_{B_L} J(|x-y|)\phi_{B_L}(x)\phi_{B_L}(y)\rd y\rd x}{\dd\int_{B_L}\phi_{B_L}^2(x)\rd x}\\
	&\leq  d\|J\|_{L^\yy(\R)}  \frac{\lf(\dd \int_{B_L}\phi_{B_L}(x)\rd x\rr)^2}{\dd\int_{B_L}\phi_{B_L}^2(x)\rd x}\leq d\|J\|_{L^\yy(\R)}  \frac{|B_L|\dd \int_{B_L}\phi_{B_L}^2(x)\rd x}{\dd\int_{B_L}\phi_{B_L}^2(x)\rd x}\\
	&= d\|J\|_{L^\yy(\R)}|B_L|\to0\ \ \ {\rm as}\ L\to 0.
	\end{align*}
	The proof is complete.
\end{proof}
\begin{corollary}\label{coro3.2}
Assume {\rm(\textbf{J})} holds. Let $\lambda_1(L)$ be given by Proposition \ref{proposition3.1}. Then 
\begin{itemize}
	\item[{\rm (i)}]  $\lambda_1(L)>0$ for all $L>0$ if $a\geq d$. 
	\item[{\rm (ii)}] If $0<a< d$, then there exists  $L_*>0$ such that $\lambda_1(L_*)=0$, $\lambda_1(L)<0$ for $0<L<L_*$ and $\lambda_1(L)>0$ for $L>L_*$.
\end{itemize}
\end{corollary}

We next consider the fixed boundary problem
\begin{equation}\label{fix1}
\begin{cases}
\dd w_t=d \int_{\Omega}J(|x-y|)w(t,y)\rd y-dw(t,x)+f(w), & t>0,\; x\in \Omega,\\
w(0,x)=w_0(x)\geq,\not\equiv 0, & x\in \Omega.
\end{cases}
\end{equation}
The following two results are well known (see, for example, \cite{bz2007}).

\begin{lemma}\label{lemma3.4}
Assume   {\rm (\textbf{J})}  holds, $f$ satisfies {\rm (\bf{f})},   and $w_i$, $\partial_t w_i\in C([0,\yy)\times \ol\Omega)$ for $i=1,2$. If 
 \begin{equation*}
 \begin{cases}
 \dd \partial_t w_1\leq d \int_{\Omega}J(|x-y|)w_1(t,y)\rd y-dw_1(t,x)+f(w_1), & t>0,\; x\in \Omega,\\[3mm]
  \dd \partial_t w_2\geq  d \int_{\Omega}J(|x-y|)w_2(t,y)\rd y-dw_2(t,x)+f(w_2), & t>0,\; x\in \Omega,\\
 w_1(0,x)\leq w_2(0,x), & x\in \ol\Omega,
 \end{cases}
 \end{equation*}
 then $w_1(t,x)\leq w_1(t,x)$ for $(t,x)\in [0,\yy)\times \ol\Omega$. If additionally $w_1(0,x)\leq, \not\equiv w_2(0,x)$ for $x\in \ol\Omega$, then $w_1(t,x)< w_1(t,x)$ for $(t,x)\in  (0,\yy)\times \Omega$.
\end{lemma}

\begin{proposition}\label{proposition3.3}
	Suppose  {\rm (\textbf{J})}  holds and $f$ satisfies {\rm (\bf{f})}.     Then  problem \eqref{fix1} admits a unique  solution $w(t,x)$ defined for all $t>0$. Moreover, if $\lambda_1(\Omega)$ is the principal eigenvalue of \eqref{7.1} wit $a=f'(0)$, then the following statements hold.
	\begin{itemize}
		\item[{\rm (i)}]  Problem \eqref{fix1} has a unique positive steady state $w^*\in C(\ol\Omega)$ if and only if $\lambda_1(\Omega)> 0$. 
		\item[{\rm (ii)}]  If $\lambda_1(\Omega)\leq  0$, then $w(t,x)$ converges to $0$ as $t\to\yy$ uniformly for $x\in \ol\Omega$.
		\item[{\rm (iii)}]  If $\lambda_1(\Omega)>0$, then $w(t,x)$ converges to $w^*$ as $t\to\yy$ uniformly for $x\in \ol\Omega$.
	\end{itemize}
\end{proposition}

\begin{remark}\label{remar3.4}
	If $\Omega=B_L$ for some $L>0$, we then denote $w$ and $w^*$ by $w_L$ and $w_L^*$, respectively. Since $J$ is radially symmetric, we see that $w_L^*$ is radially symmetric in $ B_L$. 
\end{remark}

\begin{lemma}\label{lemma3.6}
	Suppose  {\rm (\textbf{J})}  holds and $f$ satisfies {\rm (\bf{f})}.  Then
	\begin{align*}
	\lim_{L\to\yy}w_L^*=u^*\ \ {\rm locally\ uniformly\ in}\ \mathbb{R^N}.
	\end{align*}
\end{lemma}
\begin{proof} To emphasize the dependence on $\Omega$, we rewrite $w^*(x)$ as  $w^*(x; \Omega)$. 
	
	{\bf Step 1}. We show that $w^*(x;\Omega_1)\leq w^*(x;\Omega_2)$ if $\ol B_{L_*}\subset \Omega_1\subset \Omega_2$, which then implies
	$ w_{L_1}^*\leq  w_{L_2}^*$ for $L_*<L_1\leq L_2$, where  $w_{L_i}^*$ is defined as in Remark \ref{remar3.4}. 
	
	Let  $C>0$ be a constant. Denote by $w_i$  the positive solution of \eqref{fix1} with $\Omega=\Omega_i$ and initial function  $w_{i}(0,x)=C$ for $i=1,2$.
	Since
	\begin{align*}
	\int_{\Omega_2}J(x-y)w_{2}(t,y)dy\geq 	\int_{\Omega_1}J(x-y)w_{2}(t,y)dy,
	\end{align*}
	we see that the restriction of $w_{2}$ over $\Omega_1$ is an upper solution, and by Lemma \ref{lemma3.4}, $w_{1}(t,x)\leq w_{2}(t,x)$ for $t\geq 0$ and $x\in \Omega_1$. Then from  Proposition \ref{proposition3.3} (iii), 
	\begin{align*}
	w^*(x; \Omega_1)\leq w^*(x;\Omega_2)\ \ \  \ \ \ {\rm for}\ x\in \Omega_1.
	\end{align*}
	This completes Step 1. 
	
	Due to the monotonicity of $w_L^*$ in $L$, 	we could  define
	\begin{align*}
	{w}_\yy^*(x):=\lim_{L\to \yy} w_L^*(x).
	\end{align*}
	By the dominated convergence theorem,
	it is easy to see that $w_\yy^*$ satisfies
	\begin{align}\label{7.4}
	d \int_{\R^N}J(|x-y|)w_\yy^*(y)\rd y-dw_\yy^*(x)+f(w_\yy^*)=0 \mbox{ for } x\in\R^N.
	\end{align}
	
	{\bf Step 2}. We show that $w_\yy^*$ is a positive constant. 
	
	It suffices to prove that $w_\yy^*(x_0)=w_\yy^*(0)$ for any given $x_0\in \R^N$. 
	Denote   $L_0:=|x_0|$. Then for $L\gg L_0$, 
	\begin{align*}
	B_{L-L_0}\subset  B_L(-x_0)\subset B_{L+L_0},
	\end{align*}
	and from Step 1, 
	\begin{align*}
	w^*_{L-L_0}(x)\leq w^*(x; B_L(-x_0))\leq w^*_{L+L_0}(x). 
	\end{align*}
	We claim that $w^*(x; B_L(-x_0))=w^*_L(x+x_0)$. In fact,  $\td w(x):= w^*(x-x_0; B_L(-x_0))$ for $x\in B_L$ 
	also satisfies 
	\begin{align}\label{7.5}
	d \int_{B_L}J(|x-y|)w(y)\rd y-dw(x)+f(w)=0,\ \ x\in B_L,
	\end{align}
	and then $\td w=w^*_L$ follows from the uniqueness of the positive solution to \eqref{7.5}, which implies $w^*(x; B_L(-x_0))=w^*_L(x+x_0)$.  Hence
	\begin{align*}
	w^*_{L-L_0}(x)\leq w^*_L(x+x_0)\leq w^*_{L+L_0}(x). 
	\end{align*}
	Letting $x=0$ and $L\to \yy$, we obtain
	\begin{align*}
	w_\yy^*(0)\leq w_\yy^*(x_0)\leq w_\yy^*(0),
	\end{align*}
	which gives $w_\yy^*(x_0)=w_\yy^*(0)$.

	{\bf Step 3.} Since $u^*$ is the only constant positive solution of \eqref{7.4}, the conclusion in Step 2 clearly implies
	$w_\yy^*=u^*$, and so $\lim_{L\to\yy} w_L(x)=u^*$.
	By Dini's theorem, the convergence is locally uniform in $x\in\R$.
\end{proof}

\subsection{Proof of Theorem \ref{th1.2} and Theorem \ref{th1.3}}
Throughout this subsection, we assume that   {\rm (\textbf{J})}  holds and $f$ satisfies {\rm (\bf{f})}. Let $(u, h)$ be the unique solution of \eqref{1.4}.

\begin{lemma}\label{lemma3.7}  If $\lim_{t\to\yy} h(t)=h_{\yy}<\yy$, then
	\begin{align}\label{7.6}
	\lim_{t\to\yy}\|u\|_{C[-h(t),h(t)]}=0.
	\end{align}
\end{lemma}
\begin{proof}
	We first show that 
	\begin{align*}
	\lambda_1({h_\yy})\leq 0,
	\end{align*}
	where $\lambda_1({h_\yy})$ is the principal eigenvalue of \eqref{7.1} with $\Omega=B_{h_\yy}$ and $a=f'(0)$. Suppose, on the contrary, $\lambda_1(h_\yy)>0$. 
	By Proposition \ref{proposition3.1} and $\lim_{t\to\yy} h(t)=h_\yy$, for any small $\epsilon>0$, there exits a  large constant $T=T_\epsilon>0$ such that for all $t\geq T$, 
	\begin{align}\label{7.7a}
	h(t)\geq h_\yy-\epsilon\ {\rm and}\  \lambda_1({h(t)})>0.
	\end{align}
	
	Let $w_1(t,x)$ be the solution of \eqref{fix1} with $\Omega$ replaced by $B_{h(T)}$,
	and initial function $w_1(0,l)=  u(T,x)$ for $x\in B_{h(T)}$.   By the comparison principle in Lemma \ref{lemma3.4}  we obtain
	\begin{align*}
	w_1(t,x)\leq 	 u(t+T,x)	\ \   \ \  {\rm for}\  (t,x)\in [0,\yy)\times B_{h(T)},
	\end{align*}
	Recalling that $\lambda_1(h(T))>0$, we can use Proposition \ref{proposition3.3} to conclude that 
	\begin{align*}
	w_1^*(x):=\lim_{t\to\yy} w_1(t,x)>0\ \ {\rm uniformly\ for\ } x\in \ol B_{h(T)}.
	\end{align*}
	Hence
	\begin{align*}
	  \inf_{x\in B_{h(T)}} w_1^*(x)\leq \liminf_{t\to\yy} \inf_{x\in B_{h(T)}}u(t,x),
	\end{align*}
	and there exists $T_1\geq T$ such that for $t\geq T_1$,
	\begin{align}\label{7.7}
	0<C_1:=\frac 12 \min_{x\in B_{h(T)}}w_1^*(x)	\leq \frac 12 w_1^*(x)< u(t,x) \ \ \ \ \mbox{ for } x\in B_{h(T)}.
	\end{align}
	Due to $J(0)>0$, by choosing $ \epsilon$ small enough we may assume that 
	\begin{align*}
	C_2:=\min_{r\in [0, 4\epsilon]}J(r)>0.
	\end{align*}
	We also have 
	\begin{align*}
	h'(t)&= \frac{\mu} {|\partial B_{h(t)}|} \dd  \int_{B_{h(t)}} u(t,x)\lf[\int_{\R^N\backslash B_{h(t)}}J(|x-y|)\rd y\rr]\rd x\\
	&\geq \frac{\mu}{|\partial B_{h_\yy}|} \dd \int_{B_{h(t)-\epsilon} \backslash B_{h(t)-2\epsilon}} u(t,x)\lf[\int_{B_{h(t)+\epsilon}\backslash B_{h(t)}}J(|x-y|)\rd y\rr]\rd x.
	\end{align*}
	Denote $	\Omega_x(t):=[B_{h(t)+\epsilon}\backslash B_{h(t)}]\cap B_{4\epsilon}(x)$. It is clear that for $x\in B_{h(t)-\epsilon} \backslash B_{h(t)-2\epsilon}$ and $t\geq T$,
	\begin{align*}
	|\Omega_x(t)|\geq C_3>0
	\end{align*}
	for some $C_3$ depending on $\epsilon$ but not on $x$ or $t$.
	Hence for $t\geq T_1$,
	\begin{equation}\begin{aligned}\label{7.8}
	h'(t)& \geq \frac{\mu}{|\partial B_{h_\yy}|}\dd  \int_{B_{h(t)-\epsilon} \backslash B_{h(t)-2\epsilon}} u(t,x)\lf[\int_{\Omega_x(t)}J(|x-y|)\rd y\rr]\rd x\\
	&\geq\frac{\mu}{|\partial B_{h_\yy}|}  \dd \int_{B_{h(t)-\epsilon} \backslash B_{h(t)-2\epsilon}} u(t,x)C_2C_3\rd x\\
	&\geq \frac{\mu}{|\partial B_{h_\yy}|} \dd \mu |B_{h(t)-\epsilon} \backslash B_{h(t)-2\epsilon}| C_1C_2C_3\\
	&\to \frac{\mu}{|\partial B_{h_\yy}|} \dd \mu |B_{h_\infty-\epsilon} \backslash B_{h_\infty-2\epsilon}| C_1C_2C_3>0 \mbox{ as } t\to\infty.
	\end{aligned}
	\end{equation}
		However, \eqref{7.8} contradicts with the fact $h_\yy<\yy $. Therefore, $\lambda_1(h_\yy)\leq 0$.
	
	Let $w_2$ be the solution of \eqref{fix1} with $\Omega=B_{h_\yy}$ and $w_2(0,x)=u_0(x)$ for $x\in B_{h_\yy}$, where we extend the domain of $u_0(x)$  to  $B_{h_\yy}$ by defining $u_0(x)=0$ for $x\in B_{h_\yy}\backslash B_{h_0}$. 	By Lemma \ref{lemma4.4} it is easily seen that
	\begin{align*}
	u(t,x)\leq w_2(t,x),\ \ \ \ t\geq 0,\; x\in B_{h(t)}.
	\end{align*}
	Because of  $\lambda_1(h_\infty)\leq  0$, it follows from  Proposition \ref{proposition3.3} that 
	$\lim_{t\to\infty}w_2(t,x)=0$ uniformly for $x\in \ol B_{h_\yy}$, which implies 
	\eqref{7.6}. 
\end{proof}

\begin{lemma}\label{lemma3.8} 
	If $\lambda_1(h_0)<0$ and  $\mu$ is sufficiently small, then
	\begin{align}\label{7.10a}
	\lim_{t\to\yy}\|u\|_{C(\ol B_{h(t)})}=0.
	\end{align}
	
\end{lemma}
\begin{proof} By Lemma \ref{lemma3.7}, it suffices  to show $h_\yy<\yy$. 
	Since $\lambda_1(L)$ is continuous in $L$ by Proposition \ref{proposition3.3}, we obtain from $\lambda_1(h_0)<0$ that  $\lambda_1(h_1)<0$ for some $h_1:=h_0+\epsilon$
	with $\epsilon>0$ small. Let $\phi$ be a positive eigenfunction  corresponding to $\lambda_1(h_1)<0$.  Define
	\begin{align*}
	&\delta:=-\lambda_1(h_1)/2, \; C:= (h_1-h_0)/\mu,\; M:= \delta C|\partial B_{h_0}|\Big(\!\int_{B_{h_1}}\phi(x)\rd x\Big)^{-1},\\
	&\ol h(t):=h_0+\mu C [1-e^{-\delta t}],\ \ \ol u(t,x):=Me^{-\delta t} \phi(x),\ \ \ \ t\geq 0,\; x\in B_{h_1}.
	\end{align*}
	Clearly, $\ol h(t)\in [h_0, h_1)$ for $t\geq 0$.  By ({\bf f}\,) we have
	\[
	f(\bar u)\leq f'(0) \bar u.
	\]
	
	We next use Remark \ref{rmk2.6} to show that 
	\begin{align}\label{7.10}
	B_{h(t)}\subset B_{\bar h(t)}, \ \  \ t\geq 0
	\end{align}
	for all small $\mu>0$, which then implies
	\begin{align*}
	h(t)\leq \bar h(t)\leq h_1<\yy. 
	\end{align*}
	Using the equation satisfied by $\phi$, we  see
	\begin{align*}
	&\dd \bar u_t-d \int_{B_{\bar h(t)}}J(|x-y|)\bar u(t,y)\rd y+d\bar u(t,x)-f(\bar u)\\
	=&Me^{-\delta t}\lf[ -\delta \phi(x)-d \int_{B_{\bar h(t)}}J(|x-y|) \phi(y)\rd y+d \phi(x)\rr]-f(\bar u)\\
	\geq &Me^{-\delta t}[f'(0)-\lambda_1(h_1)-\delta]\phi(x)-f(\bar u)\\
	=& \delta Me^{-\delta t}\phi(x)+f'(0)\bar u -f(\bar u)
	\geq 0\ \ \ \ \ \ \ {\rm for}\  t>0, \; x\in  B_{\bar h(t)}.
	\end{align*}
	Since   $h_0\leq \bar h(t)\leq h_1$ and $\dd\int_{\R^N}J(|y|)\rd y=1$, we have, for $t>0$,
	\begin{align*}
	&\frac{\mu}{|\partial B_{ h(t)}|}\dd \int_{B_{\bar h(t)}}\bar  u(t,x)\lf[\int_{\R^N\backslash B_{\bar h(t)}}J(|x-y|)\rd y\rr]\rd x\\
	\leq& \frac{\mu}{|\partial B_{ h_0}|} \dd \int_{B_{ h_1}}\bar  u(t,x)\rd x =\frac{\mu Me^{-\delta t}} {|\partial B_{ h_0}|}\int_{B_{ h_1}}  \phi(x)\rd x=\mu C \delta e^{-\delta t}=\bar h'(t).
	\end{align*}
	It is clear that 
	\begin{align*}
	\bar u(t,x)>0 \ \mbox{ for } \ t>0,\;  x\in \partial B_{\bar h(t)}.
	\end{align*}
	Moreover, we may choose $\mu>0$ small so that $M$ is large enough such that  	
	\begin{align*}
	u_0(x)\leq M\phi(x)= \bar u(0,x) \ \mbox{ for } \ x\in B_{h_0},
	\end{align*}
	which would allow us to apply Remark \ref{rmk2.6}  to obtain \eqref{7.10}. 
\end{proof}

\begin{lemma}\label{lemma3.9}
	If $\lambda_1(h(t_0))\geq 0$ for some $t_0\geq 0$, then  $h_\yy=+\yy$ and 
	\begin{align}\label{7.12}
	\lim_{t\to\yy}u(t,r)=u^*\ \ \ \ {\rm locally\ uniformly\ in}\ \R_+,
	\end{align}
	where $\lambda_1(h(t_0))$ is the eigenvalue of \eqref{7.1} with $a=f'(0)$ and $\Omega=B_{h(t_0)}$.
\end{lemma}
\begin{proof}
	We first show that $h_\yy=\yy$. Otherwise,  $h_\yy<\yy$ and by Lemma \ref{lemma3.7}, $\lambda_1(h_\yy)\leq 0$, which contradicts with 
	$\lambda_1(h_\yy)>\lambda_1(h(t_0))\geq 0$. Thus  $h_\yy=\yy$.

	It remains to verify \eqref{7.12}, which would follow if we can show  
	\begin{align}
	&\limsup_{t\to\yy} \max_{r\in [0, h(t)]} u(t,x)\leq u^*,\label{7.13}\\
	&\liminf_{t\to\yy} \inf_{r\in [0, R]} u(t,r)\geq u_* \ \ {\rm for\ any\ } R>0.\label{7.14}
	\end{align}
	
	 Let $Z(t)$ is the unique solution of the ODE problem $Z'=f(Z)$ with  $Z(0)=\|u_0\|_\infty$. By Remark \ref{rmk2.6} we have 
	 $u(t,r)\leq Z(t)$ for $t>0$ and $r\in [0, h(t)]$, and so \eqref{7.13} follows from  $\lim_{t\to\yy} Z(t)=u^*$.
	
	By Lemma \ref{lemma3.6}, 
	\begin{align*}
	\lim_{L\to\yy}w_L^*(r)=u^*\ \ \ \ {\rm locally\ uniformly\ in}\ \mathbb{R}_+.
	\end{align*}
	Therefore in order to show \eqref{7.14}, it suffices to show that for any given large constant $K>h(t_0)$, 
	\begin{align}\label{7.15b}
	\liminf_{t\to\yy} \inf_{r\in [0, K]} u(t,r)\geq \inf_{r\in [0, K]}w_L^*(r)\ \ \ \  {\rm for\ any\ } L\geq K.
	\end{align}
	In fact,  for such $L$ there is $t_L>0$ such that $L=h(t_L)$,  and then from  $\lambda_1(L)>\lambda_1(h(t_0))\geq 0$ and Proposition \ref{proposition3.3}, we have
	\begin{align*}
	\lim_{t\to\yy} w_L(t,r)= w_L^* \ \ {\rm uniformly\ for}\ r\in [0, K],
	\end{align*}
	where  $w_L$ is the solution of \eqref{fix1} with initial function $w_L(0,r)=u(t_L,r)$  for $r\in [0, L]=[0, h(t_L)]$. By Lemma \ref{lemma3.4}, we obtain
	\begin{align*}
	u(t+t_L, r)\geq w_L(t,r) \ \mbox{ for } \ t\geq 0,\; r\in [0, K],
	\end{align*}
	which yields \eqref{7.15b}. 
\end{proof}

\noindent {\bf Proof of Theorem \ref{th1.2}\,:}   If $\lim_{t\to\yy}h(t)=h_\yy=\yy$, by Corollary  \ref{coro3.2},  there is $t_0\geq 0$ such that    $\lambda_1(h(t_0))>0$.  By  Lemma \ref{lemma3.9}, we see that spreading happens. 

If $h_\infty<\infty$, then \eqref{7.6} holds by Lemma \ref{lemma3.7}, which implies that vanishing happens. 
$\hfill \Box$

\noindent {\bf Proof of Theorem \ref{th1.3}\,:} (1) By Proposition \ref{proposition3.1}, $\lambda_1(h_0)> 0$ for any $h_0>0$, and the conclusion in part (1) follows directly from Lemma \ref{lemma3.9}.

(2) From Corollary  \ref{coro3.2}, $\lambda_1({h_0})\geq 0$ for $h_0\geq L_*$, and so we can use Lemma \ref{lemma3.9} to conclude that  spreading happens.  This proves (2)(i).

Next we consider (2)(ii). Under the assumptions for this case,  by Corollary  \ref{coro3.2} we have $\lambda_1(h_0)<0$.  From  Lemma  \ref{lemma3.8}, for any given admissible initial function  $u_0$, vanishing happens for all small $\mu>0$, say $\mu\in (0, \underline\mu]$.

To stress the dependence of the unique positive solution $(u, h)$ of \eqref{1.4} on $\mu$, we will denote it by
$(u_\mu,h_\mu)$. 
We show next that there exists $\ol\mu>0$ such that spreading happens for $\mu\geq \ol\mu$. To this end, we first prove that there exists $\mu_0>0$ such that
\bes\label{7.16b}
h_{\mu_0}(t_0)> L_* \mbox{ for  some } t_0>0.
\ees
 If \eqref{7.16b} is not true, then
\begin{align*}
h_{\mu}(t)\leq L_*\ {\rm for\ all}\ t>0\ {\rm and}\ \mu>0. 
\end{align*}
By Remark  \ref{rmk2.6}, both $u_\mu$ and $h_\mu$ are nondecreasing in $\mu$. Hence 
$h_\yy(t):=\lim_{\mu\to \yy}h_\yy(t)$ is well-defined and $h_\yy(t)\leq L_*$. Besides, since for each $\mu>0$,  $h_\mu(t)$ is nondecreasing in $t$, the function $h_\yy(t)$ is also nondecreasing in $t$. 
Define  $\td h_\yy=\lim_{t\to\yy}h_\yy(t)$. Then from the monotonicity of $h_\mu(t)$ in $\mu$ and $t$, for small $\epsilon>0$ there are $t_1>0$ and $\mu_1>0$ such that 
\begin{align*}
h_\mu(t)\in (\td h_\yy-\epsilon, \td h_\yy] \mbox{ for } \ t\geq t_1,\; \mu\geq \mu_1,
\end{align*}
and from the monotonicity of $u_\mu$ with respect to $\mu$, we obtain for all $\mu\geq \mu_1$,
\begin{align*}
0<u_{\mu_1}(t,r)\leq  u_{\mu}(t,r) \ \ \ \ \mbox{ for } t\geq t_1,\ r\in [0, \td h_\yy-2\epsilon]\subset  [0, h_\mu(t_1)-\epsilon]\subset  [0, h_{\mu_1}(t_1)].
\end{align*}
Let
\begin{align*}
C_1:={\rm min}_{\{t\in [t_1,t_1+1],\; r\in[0, \td h_\yy-2\epsilon]\}}u_{\mu_1}(t,r)>0.
\end{align*}
Due to $J(0)>0$, if $\epsilon$ is sufficiently small, we have
\begin{align*}
C_2:=\min_{r\in [0,4\epsilon]}J(r)>0.
\end{align*}
Then for  $\mu\geq \mu_1$,
\begin{align*}
h_\mu(t_1+1)-h_\mu(t_1)&=\dd \int_{t_1}^{t_1+1}\frac{\mu}{|\partial B_{h(s)}|}\int_{B_{h(s)}} u(s,|x|)\lf[\int_{\R^N\backslash B_{h(s)}}J(|x-y|)\rd y\rr]\rd x\rd s\\
&\geq \frac{\mu}{|\partial B_{\td h_\yy}|} \int_{t_1}^{t_1+1}\int_{B_{\td h_\yy-2\epsilon} \backslash B_{\td h_\yy-3\epsilon}} u(s,x)\lf[\int_{B_{\td h_\yy+\epsilon}\backslash B_{\td h_\yy}}J(|x-y|)\rd y\rr]\rd x\rd s\\
&\geq \frac{ \mu C_1}{|\partial B_{\td h_\yy}|} \int_{B_{\td h_\yy-2\epsilon} \backslash B_{\td h_\yy-3\epsilon}} \lf[\int_{B_{\td h_\yy+\epsilon}\backslash B_{\td h_\yy}}J(|x-y|)\rd y\rr]\rd x.
\end{align*}
It is clear that for $x\in B_{\td h_\yy-2\epsilon} \backslash B_{\td h_\yy-3\epsilon}$, there exists $C_3>0$ independent of $x$ such that
\begin{align*}
|\Omega_x|\geq C_3, \ \mbox{where $\Omega_x:=[ B_{\td h_\yy+\epsilon}\backslash B_{\td h_\yy}]\cap B_{4\epsilon}(x)$.}
\end{align*}
Hence
\begin{align}
h_\mu(t_1+1)-h_\mu(t_1)&\geq  \frac{ \mu C_1}{|\partial B_{\td h_\yy}|}  \int_{B_{\td h_\yy-2\epsilon} \backslash B_{\td h_\yy-3\epsilon}} C_2C_3\rd x\to \yy\ {\rm as}\ \mu\to\yy.
\end{align}
which  contradicts with the fact $h_\mu(t_1+1)-h_\mu(t_1)<\td h_\yy\leq L_*$. 
This proves \eqref{7.16b}.

Making use of \eqref{7.16b} and Corollary  \ref{coro3.2},  we see $\lambda_1(h_{\mu_0}(t_0))>0$, and then by Lemma \ref{lemma3.9} we see that  spreading happens when $\mu=\mu_0$. 

Note that  both $u_\mu$ and $h_\mu$ are nonincreasing in $\mu>0$.   The above proved facts then indicate that   vanishing happens for all small $\mu>0$ and  spreading happens for all large $\mu>0$. Define
\[
\mu_*:=\sup\{\mu^0>0: \mbox{ vanishing happens for } \mu\in (0, \mu^0]\}.
\]
Then we can follow the simple argument in the proof of  \cite[Theorem 3.14] {cdjfa} to deduce that vanishing happens for $0<\mu\leq \mu_*$ and spreading happens for $\mu>\mu_*$. The proof is now complete.
$\hfill \Box$

\section{Spreading speed }
In this section, we prove Theorem \ref{th1.6}. The analysis is presented in three subsections.
\subsection{Semi-wave }
If $f$ satisfies ({\bf f}), then it is easily seen that for all small $\sigma>0$, say $\sigma\in (-\epsilon_0,\epsilon_0)$,
\[
f_\sigma (u):=\sigma  u+f(u)
\]
also satisfies ({\bf f}), with $u^*$ replaced by some uniquely determined $u^*_\sigma$.

The  conclusions below about semi-wave solutions will be used frequently in the rest of this paper. We will use the following assumptions.
\begin{enumerate}[leftmargin=3.8em] 
	\item[\textbf{(P):}] $P\in C(\mathbb R)\cap L^\infty(\mathbb R)$,\; $P$ is even and nonnegative,\ $P(0)>0$ and $d(\|P\|_{L^1}-1)\in (-\epsilon_0,\epsilon_0)$.
	\item[\textbf{(P1):}] $\dd\int_{0}^{\yy} xP(x)\rd x<\yy$.
\end{enumerate}
\begin{proposition}\label{prop4.1} Let  $d$  and $\mu$ be positive constants and $f$ be a function satisfying {\rm ({\bf f})}, with $\epsilon_0>0$ given as above.
	\begin{itemize} 
		\item[{\rm (1)}] Assume $P$ satisfies {\rm \textbf{(P)}}. Then the following problem 
		\begin{equation}\label{semi1}
		\begin{cases}
		\dd d\int_{-\yy}^{0}P(x-y) \phi(y) {\rm d}y-d\phi+c\phi'(x)+f(\phi) =0,&
		x<0,\\[2mm]
		\dd	\phi(-\yy)=\hat u^*,\ \ \phi(0)={0},
		\end{cases}
		\end{equation}
		with
		\begin{align}\label{semi2}
		c=\mu\int_{-\yy}^{0} \phi(x)\int_{0}^{\yy}P(x-y)\rd y\rd x. 
		\end{align}
		admits a unique solution $(c,\phi)$ with $\phi$ monotone
		 if and only if  {\rm \textbf{(P1)}} holds, where $\hat u^*$ is the unique positive root of the equation $\hat f(u):=d(\|P\|_{L^1}-1)u+f(u)=0$.
		\item[{\rm (2)}] Let $\{P_n\}_{n=1}^{\yy}$ be a  sequence  with each  $P_n$ satisfying {\rm \textbf{(P)} }and {\rm \textbf{(P1)}}. Denote by $(c_n,\phi_n)$  the unique solution of \eqref{semi1}-\eqref{semi2} with $P=P_n$. Suppose $P_n(x)$  converges to a function $\wtd P(x)$ satisfying {\rm \textbf{(P)}}, locally uniformly for $x\in \R$. 
		\begin{itemize}
			\item[{\rm (i)}]    If $\wtd P(x)$ satisfies {\rm \textbf{(P1)}} {and  there is $Q\in L^1(\R)$ satisfying $\dd\int_{0}^\yy Q(x)x\rd x<\yy$ such that $P_n\leq Q$},
			then 
			\begin{align*}
			\lim_{n\to\yy} \phi_n=\td\phi,\ \lim_{n\to \yy} c_n=\td c.
			\end{align*} 
			where $(\td c,\td \phi)$ is the unique solution of \eqref{semi1}-\eqref{semi2} with $P$ replaced by  $\wtd P$.
			\item[{\rm (ii)}] If  $\wtd P$  does not  satisfy {\rm \textbf{(P1)}},  then 
			\begin{align*}
			\lim_{n\to\yy} c_n=\yy.
			\end{align*}
		\end{itemize}
	\end{itemize}
\end{proposition}

To prove  Proposition \ref{prop4.1}, we will need \cite[Lemma 2.8]{dn-speed}, which we restate below.
\begin{lemma}\label{lemma5.2}
	Let  $(c_1,\phi_1,P_1)$ and $(c_2,\phi_2,P_2)$ be given as in Proposition \ref{prop4.1} {\rm (2)}.   If $P_1\leq P_2$, then 
	\begin{align*}
c_1\leq c_2\ {\rm and}\ \phi_1\leq \phi_2.
	\end{align*} 
\end{lemma}

\begin{proof}[{\bf Proof of Proposition \ref{prop4.1}:}]  The conclusion in part  (1)  follows from  \cite[Theorem 1.2]{dlz2019} once we rewrite the first equation in \eqref{semi1} as
\[
\hat d\int_{-\yy}^{0}\hat P(x-y) \phi(y) {\rm d}y-\hat d\phi+c\phi'(x)+\hat f(\phi) =0,
\]
with $\hat d:=d\|P\|_{L^1},\; \hat P(x):=P(x)/\|P\|_{L^1}$ and $\hat f(u):=d(\|P\|_{L^1}-1)+f(u)$.
	
	To prove part (2), we first show
	\begin{align}\label{3.3d}
	\inf_{n\geq 1} c_n>0.
	\end{align}
	Fix $\epsilon>0$ sufficiently small and let $P_0$ be a continuous  function satisfying 
	\begin{align*}
	&P_0(x)=\max\{\wtd P(x)-\epsilon^2, 0\}\ {\rm for}\ |x|\leq 1/\epsilon,\ \ \ \ 
	P_0(x)=0\ {\rm for}\ |x|\geq  1/\epsilon+1,\\
		&P_0(x)\leq \max\{\wtd P(x)-\epsilon^2, 0\}\ {\rm for}\  x\in \R.
	\end{align*}
 Then
	\begin{align*}
\|P_0\|_{L^1(\R)}\geq& \int_{-1/\epsilon}^{1/\epsilon} P_0(x) \rd x\geq  \int_{-1/\epsilon}^{1/\epsilon} \wtd P(x) \rd x-\int_{-1/\epsilon}^{1/\epsilon} \epsilon^2 \rd x\\
=&\int^{1/\epsilon}_{-1/\epsilon} \wtd P(x) \rd x-2\epsilon \to \|\wtd P\|_{L^1(\R)} \mbox{ as } \epsilon\to 0,
	\end{align*}
	and $\|P_0\|_{L^1(\R)}\leq \|\wtd P\|_{L^1(\R)}$. Hence $P_0$ satisfies \textbf{(P)} for such $\epsilon$. Since	$P_n$  converges to  $\wtd P$  locally uniformly for $x\in \R$,  there is $n_\epsilon>0$ such that for $n\geq n_\epsilon$, 
	\begin{align*}
P_n\geq \max\{P-\epsilon^2,0\}\geq P_0 \ \mbox{ for} \ |x|\leq 1/\epsilon+1.  
	\end{align*}
Let $(\hat c,\hat \phi)$ be the solution of \eqref{semi1}-\eqref{semi2} with $P=P_0$. Then by Lemma \ref{lemma5.2}, we have
\begin{align}\label{3.4d}
c_n\geq \hat c,\ \phi_n\geq \hat \phi,\ \ \ \ \ n\geq n_\epsilon.
\end{align} 
Thus 
\begin{align*}
c_n\geq \min\{c_1,c_2,\cdots,c_{n_\epsilon}, \hat c\}>0,
\end{align*}
which proves \eqref{3.3d}. 
	
By \eqref{semi1} and \eqref{3.3d},  
\begin{align*}
\sup_{n\geq 1}\phi_n'(x)\leq&\frac{1}{(\inf_{n\geq 1}c_n)}\sup_{n\geq 1} \lf(d[\|P_n\|_{L^\yy(\R)}+1]\phi_n(-\yy)+\max_{u\in [0,\phi_n(-\yy)]}f(u)\rr)<\yy.
\end{align*}	
Assume that 
\begin{align}\label{4.5a}
K:=\sup_{n\geq 1} c_n<\yy. 
\end{align}
Then by the Arzela-Ascoli Theorem and a strand argument including a diagonal process of choosing subsequences, there are $\phi_\yy\in C(\R_-)$ and a subsequence of $\{\phi_n\}_{n\geq 1}$, still denoted by itself, such that $\phi_{n}$ converges to $ \phi_\yy$ locally uniformly in $\R_-$. Moreover, $ \phi_\yy(x)$ is nonincreasing in $x$.
	
By \eqref{4.5a}, without loss of generality, we assume that $c_n\to c_\yy$ as $n\to\yy$. Then we claim that  $(c_\yy,\phi_\yy)$ satisfies 
	\begin{align}\label{3.6d}
	d \int_{-\yy}^{0}\wtd  P (x-y) \phi_\yy(y) {\rm d}y-d\phi_\yy(x)+c_\yy\phi_\yy'(x)+f(\phi_\yy(x))=0,\ \ \ \ x\in \R_-.
	\end{align}
	In fact, from \eqref{semi1},
	\begin{align*}
	c_n\phi_{n}(x)-c_n\phi_{n}(0)=&- \int_{0}^{x} \lf[d\int_{-\yy}^{0}{P_n} (\xi -y) \phi_{n}(y) {\rm d}y-d\phi_{n}(\xi)+f(\phi_{n}(\xi))\rr] {\rm d}\xi.
	\end{align*}
	For given $x\in \R_-$, it then follows from the dominated convergence theorem that 
	\begin{align*}
	c_\yy\phi_\yy(x)-c_\yy\phi_\yy(0)=&- \int_{0}^{x} \lf[d\int_{-\yy}^{0}{\wtd P} (\xi -y) \phi_\yy(y) {\rm d}y-d\phi_\yy(\xi)+f(\phi_\yy(\xi))\rr] {\rm d}\xi,
	\end{align*}
	and hence \eqref{3.6d} holds by differentiating this equation.  
	
We claim that $\phi_\yy(-\yy)>0$ satisfies
\begin{align}\label{3.7b}
f(\phi_\yy(-\yy))+d(\|\wtd P\|_{L^{1}(\R)}-1)\phi_\yy(-\yy)=0.
\end{align}

  	Due to the monotonicity and boundedness of $\phi_\yy(x)$, one could easily get
	\begin{align*}
	\lim_{x\to -\yy} \lf[d \int_{-\yy}^{\yy}\wtd P(x-y) \phi_\yy(y) {\rm d}y-d\|\wtd P\|_{L^{1}(\R)}\phi_\yy(x)\rr]=0,
	\end{align*}
	and so 
	\begin{align*}
	\lim_{x\to -\yy} [c\phi_\yy'(x)+f(\phi_\yy(x))+d(\|\wtd P\|_{L^{1}(\R)}-1)\phi_\yy(x)]=0.
	\end{align*}
 From again the monotonicity and boundedness  of $\phi_\yy(x)$ in $x$,  we deduce that $\lim_{x\to -\yy}\phi_\yy'(x)=\lim_{x\to-\yy}$
 $ f(\phi_\yy(x))+d(\|\wtd P\|_{L^{1}(\R)}-1)\phi_\yy(x)=0$. Recalling that  $\phi_n\geq \hat \phi$ for $n\geq n_\epsilon$ and $\phi_n\to \phi_\yy$ as $n\to\yy$, we see that $\phi_\yy(-\yy)>0$. Hence, \eqref{3.7b} holds. 
 
(i)  We now show that 
\begin{align*}
c_\yy=\lim_{n\to \yy} c_n=\lim_{n\to \yy}\mu\int_{-\yy}^{0} \phi_n(x)\int_{0}^{\yy} P_n(x-y)\rd y\rd x=\mu\int_{-\yy}^{0} \phi_\yy(x)\int_{0}^{\yy}\wtd P(x-y)\rd y\rd x.
\end{align*}
In fact, due to
\begin{align*}
&\int_{-\yy}^{0} \phi_n(x)\int_{0}^{\yy} P_n(x-y)\rd y\rd x=\int_{0}^\yy P_n(y)\int_{-y}^0 \phi_n(x)\rd x\rd y,\\
&\int_{-\yy}^{0} \phi_\yy(x)\int_{0}^{\yy}\wtd P(x-y)\rd y\rd x=\int_{0}^\yy \wtd P(y)\int_{-y}^0 \phi_\yy(x)\rd x\rd y,
\end{align*}
and $P_n\leq Q$, the desired identity follows from
 the dominated convergence theorem. Therefore,  $(c_\yy, \phi_\yy)$ is a solution of \eqref{semi1} and \eqref{semi2} with $P=\wtd P$.    It then follows from the conclusion in part (1) that $(c_\yy,\phi_\yy)=(\td c,\td \phi)$. The above discussions  also indicate that 
\begin{align*}
\lim_{n\to\yy} c_n=\td c,\   \lim_{n\to\yy} \phi_n=\td \phi \mbox{ for the entire original sequence } (c_n,\phi_n).
\end{align*}

(ii) If $\wtd P$  does not  satisfy {\rm \textbf{(P1)}}, then for any constants  $L>L_0>0$, we have 
\begin{align*}
\lim_{n\to \yy}\int_{-L}^{0} \phi_n(x)\int_{0}^{L} P_n(x-y)\rd y\rd x&=\int_{-L}^{0} \phi_\yy(x)\int_{0}^{L}\wtd P(x-y)\rd y\rd x\\
&\geq \int_{-L/2}^{-L_0/2}\phi_\infty(x)\int_{-x}^{L-x}\wtd P(y)dydx\\
&\geq c_0   \int_{-L/2}^{-L_0/2} \int_{-x}^{L-x}\wtd P(y)dydx\\
&\geq c_0 \int_{L_0/2}^{L/2}(y-\frac{L_0}2)\wtd P(y)dy \to\infty \mbox{ as } L\to\infty,
\end{align*}
where $c_0:=\inf_{x\leq -L_0/2} \phi_\infty(x)>0$  due to $\phi_\infty\geq \hat\phi $.
Hence,
\begin{align*}
\liminf_{n\to \yy} c_n&=\mu \liminf_{n\to \yy}\int_{-\yy}^{0} \phi_n(x)\int_{0}^{\yy} P_n(x-y)\rd y\rd x\\
&\geq \mu \lim_{n\to \yy}\int_{-L}^{0} \phi_n(x)\int_{0}^{L} P_n(x-y)\rd y\rd x\to\infty \mbox{ as } L\to\infty,
\end{align*}
which implies $\lim_{n\to\yy} c_n=\yy$.
\end{proof}

In the rest of this section, we make use of Proposition \ref{prop4.1} to determine the spreading speed of \eqref{1.4} according to the behaviour of the kernel function $J$.
 
\subsection{Infinite speed}

\begin{theorem}\label{thm5.3}
Suppose that {\rm {\bf (J)}}, {\rm {\bf (f)}} and \eqref{u_0} hold, and spreading happens to the unique positive solution $(u, h)$ of \eqref{1.4}.
	 If  {\rm \textbf{(J1)}} is not satisfied, then 
	\begin{align}\label{4.1d}
	\lim_{t\to\yy} \frac{h(t)}{t}=\yy.
	\end{align}
\end{theorem}

To prove this theorem we will use the following lemma.

\begin{lemma}\label{lemma3.8a}
	If in Theorem \ref{thm5.3} the kernel function $J$ has compact support $($and so {\rm {\bf (J1)}} is satisfied$)$, then 
	\begin{align}\label{3.16d}
	\liminf_{t\to\yy}\frac{h(t)}{t}\geq   c_0,
	\end{align}
	where $c_0>0$ is given by Proposition \ref{prop1.4}. 
\end{lemma}
\begin{proof} Suppose the supporting set of $J$ is contained in $B_{K_*}$.
	Let $\{\epsilon_n\}_{n=1}^\yy$ be a sequence satisfying $0<\epsilon_n\to 0$ as $n\to\yy$.	Define
	\begin{align*}
	J_n(l):=\max\{J_*(l)-\xi_n(l),0\},
	\end{align*}
	where $\xi_n$ is given by 
	\begin{equation*}
\xi_n(l):=\begin{cases}
\epsilon_n,& |l|\leq K_*,\\
\epsilon_n(K_*+1-|l|), & K_*\leq |l|\leq  K_*+1,\\
=0,& |l|\geq   K_*+1.
\end{cases}
	\end{equation*}
	Clearly $J_n$ has compact support and $J_n\to J_*$ locally uniformly in $\R$ and in $L^1(\R)$.
	
	For  fixed $\epsilon_n$,  by Proposition \ref{coro2.5} and the definition of $\xi_n$, there is $L_n\gg K_*$ such that  
	\begin{align}\label{3.17d}
	J_n(l-\rho)\leq \td J(l,\rho) \mbox{ for } \rho>0,\ l\geq L_n,
	\end{align} 
	and
	\begin{align}\label{3.18d}
	\lf(1-\frac{K_*}{L_n}\rr)^{N-1}>1-\epsilon_n
	\end{align}
	By Proposition \ref{prop4.1},  for all large $n$, problem \eqref{semi1}-\eqref{semi2} with $P=J_n$ admits a solution $(c_n,\phi_n)$ satisfying
	\begin{align*}
	\dd\phi_n(-\yy)=:u_n^*< u^*,
	\
	\lim_{n\to\yy} c_n=c_0,
	\end{align*}
	where we have used  $c_n\leq c_{n+1}\leq c_0$, which implies that $\{c_n\}$  is bounded. 
		
	For fixed large $n$ define 
	\begin{align*}
	&\underline h(t):=c_n(1-2\epsilon_n) t+2L_n,  &&\ \ \ t\geq 0,\\
	&\underline u(t,r):=(1-\epsilon_n) \phi_n(r-\bar h(t)), &&\ \ \ t\geq 0,\; L_n\leq r\leq \bar h(t).
	\end{align*}
	We show that there is $t_1>0$ such that 
	\begin{equation}\label{3.20d}
	\begin{cases}
	\dd \underline u_t(t,r)\leq  d \int_{0}^{\underline h(t)} \td J(r,\rho) \underline u(t,\rho)\rd \rho-d\underline u(t,r)+f(\rho,\underline u), & t>0,\; r\in  (L_n,\underline h(t)),\\[3mm]
	\dd \underline h'(t)\leq  \frac{\mu}{\underline h^{N-1}(t)}\dd\int_{0}^{\underline h(t)} r^{N-1}\underline u(t,r)\int_{\underline h(t)}^{+\yy} \td J(r,\rho)\rd \rho\rd r, & t>0,\\
	\underline u(t,r)\leq  u(t+t_1,r), \; \underline u(t,\underline h(t))=0, & t>0,\; r\in [0,L_n],\\
	\underline u(0,r)\leq \ u(t_1,r),\ h(t_1)>\underline h(0), & r\in [0,\underline h(0)].
	\end{cases}
	\end{equation}
	
	Since spreading happens, there is $t_1>0$ such that $h(t)\geq 3L_n$ for $t\geq t_1$, and 
	\begin{align*}
	 u(t,r)\geq u_n^*,\ \ \ \ (t,r)\in [t_1,\yy)\times [0,3L_n]. 
	\end{align*}
	Hence, 
	\begin{align*}
	\underline u(t,r)\leq (1-\epsilon_n) u_n^*\leq u_n^*\leq  u(t+t_1,r) \ \mbox{ for }\  (t,r)\in [0,\yy)\times [0,L_n].
	\end{align*}
	and 
	\begin{align*}
	\underline u(0,r)\leq (1-\epsilon_n) u_n^*\leq u_n^*\leq  u(t_1,r ) \mbox{ for }\  r\in [0,\underline h(0)]. 
	\end{align*}
	It is clear that  $\underline u(t,\underline h(t))=0$. Hence, the last two inequalities  of \eqref{3.20d} are satisfied.
	
	Next, we check the first two inequalities of \eqref{3.20d}.  Making use of \eqref{3.17d}, \eqref{3.18d} and $\td J(r,\rho)=0$ for $|r-\rho|\geq K_*$,   
	we have
	\begin{align*}
	&\frac{\mu}{\underline h^{N-1}(t)}\dd\int_{0}^{\underline h(t)} r^{N-1}\underline u(t,r)\int_{\underline h(t)}^{+\yy} \td J(r,\rho)\rd \rho\rd r\\
	=&\frac{\mu}{\underline h^{N-1}(t)}\dd\int_{\underline h(t)-K_*}^{\underline h(t)} r^{N-1}(1-\epsilon_n) \phi(r-\underline h(t))\int_{\underline h(t)}^{+\yy} \td J(r,\rho)\rd \rho\rd r\\
	\geq  &\frac{\mu}{\underline h^{N-1}(t)}\dd\int_{\underline h(t)-K_*}^{\underline h(t)} r^{N-1}(1-\epsilon_n) \phi(r-\underline h(t))\int_{\underline h(t)}^{+\yy}  J_n(r-\rho)\rd \rho\rd r\\
	= &\frac{\mu}{\underline h^{N-1}(t)}\dd\int_{-K_*}^{0} (r+\underline h(t))^{N-1}(1-\epsilon_n) \phi(r)\int_{0}^{+\yy}  J_n(r-\rho)\rd \rho\rd r\\
	\geq  &\lf(1-\frac{K_*}{\underline h(t)}\rr)^{N-1}\mu\dd\int_{-K_*}^{0} (1-\epsilon_n) \phi(r)\int_{0}^{+\yy}  J_n(r-\rho)\rd \rho\rd r\\
	= &\lf(1-\frac{K_*}{\underline h(t)}\rr)^{N-1}\mu\dd\int_{-\yy}^{0} (1-\epsilon_n) \phi(r)\int_{0}^{+\yy}  J_n(r-\rho)\rd \rho\rd r\\
	=&(1-\epsilon_n)\lf(1-\frac{K_*}{\underline h(t)}\rr)^{N-1} c_n\geq (1-\epsilon_n)\lf(1-\frac{K_*}{L_n}\rr)^{N-1} c_n\\
	\geq &(1-\epsilon_n)^2c_n\geq (1-2\epsilon_n)c_n=\underline h'(t).
	\end{align*}
	From the equation satisfied by $\phi$, \eqref{3.17d} and
	\begin{align*}
	J_n(r-\rho)=0 \ \mbox{ for } \ r\in [L_n,\underline h(t)],\; \rho\leq 0,
	\end{align*}
	we deduce for $t>0$ and $r\in [L_n,\underline h(t)]$,
	\begin{align*}
	\underline u_t(t,r)&=-(1-\epsilon_n)c_n(1-2\epsilon_n)\phi_n'(r-\underline h(t))\leq -(1-\epsilon_n)c_n\phi_n'(r-\underline h(t))\\ 
	&= (1-\epsilon_n) \lf[d  \int_{- \yy}^{\underline h(t)}    J_n (r-\rho)   \phi_n(\rho-\underline h(t))\rd \rho -d \phi_n(r-\underline h(t))+ f(\phi_n(r-\underline h(t)))\rr]\\ 
	&=d  \int_{0}^{\underline h(t)}    J_{n} (r-\rho)  \underline u(t,\rho)\rd \rho -d  \underline u(t,r)+(1-\epsilon_n) f(\phi_n(r-\underline h(t)))\\
	&\leq   d  \int_{0}^{\underline h(t)}    J_{n} (r-\rho)  \underline u(t,\rho)\rd \rho -d \underline u(t,r)+ f(\underline u(t,r))\\
	&\leq  d  \int_{0}^{\underline h(t)}   \td J (r,\rho)  \underline u(t,\rho)\rd \rho -d \underline u(t,r)+ f(\underline u(t,r)).
	\end{align*}
	Hence \eqref{3.20d} holds.

	By Lemma \ref{lemma4.4} (2) with $r_*(t)=L_n$, we obtain
	\begin{align*}
	& h(t+t_1)\geq  \underline h(t) \ \ \ \ \ \ \ \ \mbox{ for } t\geq 0,\\
	& u(t+t_1,r)\geq  \underline u(t,r)\ \ \ \mbox{ for } t\geq 0,\ r\in [ 0,\underline h(t)].
	\end{align*}
	Hence 
	\begin{align*}
	\liminf_{t\to\yy} \frac{h(t)}{t}\geq    \liminf_{t\to\yy} \frac{\underline h(t-t_1)}{t}=c_n(1-2\epsilon_n),
	\end{align*}
	which  gives \eqref{3.16d} by letting $n\to\infty$.  
\end{proof}

\begin{proof}[{\bf Proof of Theorem \ref{thm5.3}}]	Define 
	\begin{align*}
	J_n(r):=\zeta\big(\frac{r}{n}\big) J(r),\ \ n=1,2,\cdots, 
	\end{align*}
where  
\begin{align*}
\zeta(\xi)=
\begin{cases}
0,&|\xi|\geq 2,\\
2-|\xi|,&1\leq|\xi|\leq 2,\\
1,&|\xi|\leq 1.
\end{cases}
\end{align*}
Clearly,   $J_n$ has compact support,  is nondecreasing in $n$, $J_n\leq J$ and 
\begin{align}\label{4.2d}
\lim\limits_{n\rightarrow\infty}J_n|x|)=J(|x|)\ \text{\ locally uniformly in\ } \R^N.
\end{align}

Similarly to $\td J$ and $J_*$, we  define
\begin{align*}
&\td J^n(r,\rho)=\int_{\partial B_\rho} J_n(|x-y|) \rd y\ \ {\rm with }\ |x|=r,\\
&J^n_*(l):=\int_{\R^{N-1}} J_n(|(l,x')|)dx'.
\end{align*}
For large $n$ and some $T>0$ to be determined, let $( u_n,h_n)$ be the solution of \eqref{1.4} with $\td J$ replaced by $\td J^n$ and
\begin{align*}
h_n(0)=h(T),\   u_n(0,r)= u(T,r) \mbox{ for }  r\in [0,h(t)].
\end{align*}
 Due to $\td J^n\leq \td J$ and the comparison principle in Lemma \ref{lemma3.4a} (2) with $r_*(t)\equiv 0$, we have
\begin{align*}
h_n(t)\leq h(t+T),\ u_n(t,r)\leq  u(t+T,r) \ \mbox{ for }\ t\in [0,\yy),\; r\in [0,h_n(t)].
\end{align*}

By Theorem \ref{th1.2}, spreading happens for $(u_n,h_n)$ if $h_n(0)$ is greater than a constant $L_n$ determined by  $J_n$.  Since spreading happens for \eqref{1.4},  there is a constant $t_n>0$ such that  $h(t)>L_n$ for all $t\geq t_n$. Choosing $T=t_n$,  we then have
\begin{align*}
\lim_{t\to\yy} h_n(t)=\yy.
\end{align*}

Since $J_n$ has compact support,  by Lemma \ref{lemma3.8a}  we have
\begin{align}\label{5.14}
\liminf_{t\to\yy} \frac{h_n(t)}{t}\geq c_n.
\end{align}
where $c_n$, associated with a function $\phi_n$, is the unique solution of \eqref{semi1}-\eqref{semi2} with $P=J^n_*$.
By the monotonicity of $J_n$ and \eqref{4.2d}, we see that $J_n^*$ is nondecreasing in $n$ and 
\begin{align*}
\lim_{n\to\yy}  J^n_*(l)=J_*(l)\ \ {\rm locally\ uniformly\ for}\ l\in \R.
\end{align*} 
Therefore, in view of the assumption that {\bf (J1)} is not satisfied,  we can use Proposition \ref{prop4.1} to conclude that
\begin{align}\label{4.3d}
\lim_{n\to\yy} c_n=\yy.
\end{align}
Since
\begin{align*}
\liminf_{t\to\yy} \frac{h(t)}{t}\geq \liminf_{t\to\yy} \frac{h_n(t-t_n)}{t}\geq c_n,
\end{align*}
we see that \eqref{4.1d} follows from \eqref{4.3d}.
\end{proof}

\subsection{Finite speed} In this subsection, we prove the following result.

\begin{theorem}\label{thm5.5}
Suppose that {\rm {\bf (J)}}, {\rm {\bf (f)}} and \eqref{u_0} hold, and spreading happens to the unique positive solution $(u, h)$ of \eqref{1.4}.
	 If  {\rm \textbf{(J1)}} is satisfied, then 
	\begin{align}\label{4.1d}
	\lim_{t\to\yy} \frac{h(t)}{t}=c_0,
	\end{align}
	where $c_0$ is given by Proposition \ref{prop1.4}.
\end{theorem}

\begin{proof}

We first show that
\begin{align}\label{geq-c_0}
	\liminf_{t\to\yy}\frac{h(t)}{t}\geq  c_0.
	\end{align}
Note that we can obtain \eqref{5.14} by repeating the argument in the proof of Theorem \ref{thm5.3}. Since {\bf (J1)} holds, now we have, by Proposition \ref{prop4.1}, $\lim_{n\to\infty} c_n=c_0$, and so
\eqref{geq-c_0} follows by letting $n\to\infty$ in \eqref{5.14}.

To prove \eqref{4.1d}, it remains to show
	\begin{align}\label{leq-c_0}
	\limsup_{t\to\yy}\frac{h(t)}{t}\leq  c_0.
	\end{align}
		
	Let $\{\epsilon_n\}$ be a sequence with $0<\epsilon_n\ll 1$ and $\epsilon_n\to 0$ as $n\to\yy$,  and let the function $J_{\epsilon_n}$ be given by \eqref{2.5d} with $\epsilon=\epsilon_n$. Define 
	\begin{align*}
	J_n:=(1+\sqrt{\epsilon_n})J_{\epsilon_n}.
	\end{align*}
From the definition of $J_{\epsilon_n}$, it is clear that $J_n\leq 3(1+\sqrt{\epsilon_n})J_*\leq 4J_*$, and so each $J_n$  satisfies \textbf{(J1)}. 
By Proposition \ref{prop4.1},  the problem \eqref{semi1}-\eqref{semi2} with $P=J_n$ admits a unique solution $(c_n,\phi_n)$.  Since $J_n$ converges to $J$ and   $J\leq J_n\leq 4J_*$, we could apply Proposition \ref{prop4.1} to conclude 
	\begin{equation}\label{4.5d}
	\begin{cases}
		\dd\phi_n(-\yy)=:u_n^*>u^*,\\
			c_n\geq c_0,\ \ \lim_{n\to\yy} c_n=c_0.
	\end{cases}
	\end{equation}

For fixed $n\geq 1$, define 	
\begin{align*}\begin{cases}
	\bar g(t):=c_n(1+2\epsilon_n) t+K,\ \bar h(t):=(1+\epsilon_n/2)\bar g(t),  & t\geq 0,\\
	\ol u(t,r):=(1+\epsilon_n) \phi_n(r-\bar h(t)), & t\geq 0,\;0\leq r\leq \bar h(t),
	\end{cases}
	\end{align*}
	where  $K>0$  is a large constant to be determined. 
	For convenience, we extend $u(t,r)$ to $[0,\yy)\times [0,\yy)$ by defining $ u(t,r)=0$ for $t\in [0,\yy)$ and $r\in (h(t),\yy)$. 
	
	Next we show that for large fixed $n$,  there are $K:=K_n$  and $T:=T_n>0$ such that $(\bar u,\bar h)$ satisfies
	\begin{equation}\label{4.6d}
	\begin{cases}
	\dd \bar u_t(t,r)\geq d \int_{0}^{2\bar h(t)/3} \td J(r,\rho)  u(t,\rho)\rd \rho+d \int_{2\bar h(t)/3}^{\bar h(t)} \td J(r,\rho) \bar u(t,\rho)\rd \rho\\
\ \ \ \ \ \  \ \ \ \ \ \  \ -d\bar u(t,r)	+f(\bar u), & t>0,\; r\in  (\bar g(t),\bar h(t)),\\[3mm]
	\dd \bar h'(t)\geq \frac{\mu}{\bar h^{N-1}(t)}  \dd\int_{0}^{\bar h(t)} r^{N-1}\bar u(t,r)\int_{\bar h(t)}^{+\yy} \td J(r,\rho)\rd \rho\rd r, & t>0,\\
	\bar u(t,r)\geq  u(t+T,r), \; \bar u(t,\bar h(t))=0, & t\geq 0,\; r\in [0,\bar g(t)],\\
	\bar u(0,r)\geq  u(T,r),\ \bar h(0)\geq h(T), & r\in [0,h(T)].
	\end{cases}
	\end{equation}
	Let us note that, since $\bar g(t)>2\bar h(t)/3$, if we define 
	\[
	\hat u(t,r):=\begin{cases} u(t+T, r), & t>0,\ r\in [0, \bar g(t)],\\
	\bar u(t,r), & t>0,\ r\in (\bar g(t), \bar h(t)],
	\end{cases}
	\]
	then the third inequality in \eqref{4.6d} implies
	\[
	d \int_{0}^{2\bar h(t)/3} \td J(r,\rho)  u(t,\rho)\rd \rho+d \int_{2\bar h(t)/3}^{\bar h(t)} \td J(r,\rho) \bar u(t,\rho)\rd \rho\geq d \int_{0}^{\bar h(t)} \td J(r,\rho)  \hat u(t,\rho)\rd \rho.
	\]
	Therefore, when \eqref{4.6d} holds, we can apply Lemma \ref{lemma3.4a} to conclude that 
	\begin{align}\label{4.9d}
	h(t+T)\leq \bar h(t) \mbox{ for all } t>0,
	\end{align}
which yields
	 \begin{align*}
	 	\limsup_{t\to\yy} \frac{h(t)}{t}\leq    \limsup_{t\to\yy} \frac{\bar h(t-T)}{t}=(1+\epsilon_n/2)c_n(1+2\epsilon_n),
	 \end{align*}
	 and \eqref{leq-c_0} then follows by letting $n\to\infty$.
	 
	 Therefore, to complete the proof of the theorem, it suffices to prove \eqref{4.6d}, which is carried out in the following three steps.
	
	\textbf{Step 1}. We check the  last two inequalities of \eqref{4.6d}.
	
	Since $\phi(-\yy)=u_n^* $, there is $K_0>0$ such that for any $K\geq K_0$, 
	\begin{align*}
(1+\epsilon_n)	\phi_n(-K)\geq u_n^*.
	\end{align*}
	Hence for $K\geq \frac{4K_0}{\epsilon_n}$, we have for $t\in [0,\yy)$ and $r\in [0,[\bar g(t)+\bar h(t)]/2]$,
	\begin{align}\label{4.7d}
	\ol u(t,r)&=(1+\epsilon_n) \phi_n(r-\bar h(t))\geq (1+\epsilon_n) \phi\lf(\frac{\bar g(t)+\bar h(t)}{2}-\bar h(t)\rr)\nonumber\\
	&=(1+\epsilon_n) \phi_n\lf(-\frac{\epsilon_n}{4}\bar  g(t)\rr)\geq (1+\epsilon_n) \phi_n\lf(-\frac{\epsilon_n}{4}K\rr)\geq (1+\epsilon_n) \phi(-K_0)\\
	&\geq u_n^*.\nonumber
	\end{align}
It is easily seen that
	\begin{align*}
	\limsup_{t\to\yy} \max_{r\in [0,h(t)]}  u(t,r)\leq {u}^*<u_n^*.
	\end{align*}
	Since   $h(t)\to\yy$ as $t\to\yy$ by assumption, there is a  $T>0$ such that 
	\begin{align*}
	 u(t+T,r)\leq  u_n^*\ \mbox{ for }   t\geq 0,\; r\in [0,h(t)].
	\end{align*}
By further enlarging $K$, we may assume that $K>h(T)$ and hence, from \eqref{4.7d}  we see
	\begin{align*}
&\bar u(0,r)\geq u_n^*\geq   u(T,r)\ \mbox{ for } \  r\in [0,h(T)].
	\end{align*}
	Clearly, $\bar u(t,\bar h(t))=0$. Therefore,  the last two inequalities of \eqref{4.6d} hold. 
	
		\textbf{Step 2}. We verify the  first   inequality of \eqref{4.6d}. 
		
From the equation satisfied by $\phi_n$, 
	we deduce for $t>0$ and $r\in (\bar g(t),\bar h(t))$,
	\begin{align*}
	\bar u_t(t,r)&=-(1+\epsilon_n)c_n(1+2\epsilon_n)\phi_n'(r-\bar h(t))\geq -(1+\epsilon_n)c_n\phi_n'(r-\bar h(t))\\ 
	&= (1+\epsilon_n) \lf[d  \int_{- \yy}^{\bar h(t)}    J_n (r-\rho)   \phi_n(\rho-\bar h(t))\rd \rho -d \phi_n(r-\bar h(t))+ f(\phi_n(r-\bar h(t)))\rr]\\ 
	&= d  \int_{-\yy}^{\bar h(t)}    J_{n} (r-\rho)  \bar u(t,\rho)\rd \rho -d  \bar u(t,r)+(1+\epsilon_n) f(\phi_n(r-\bar h(t)))\\
	&\geq   d  \int_{-\yy}^{\bar h(t)}    J_{n} (r-\rho)  \bar u(t,\rho)\rd \rho -d \bar u(t,r)+ f(\bar u(t,r)).
	\end{align*}
	In order to get the first inequality of \eqref{4.6d}, it remains to prove for $t>0$ and $r\in  (\bar g(t),\bar h(t))$,
	\begin{align}\label{4.8d}
	\int_{-\yy}^{\bar h(t)}    J_n (r-\rho)  \bar u(t,\rho)\rd \rho-\int_{0}^{2\bar h(t)/3}   \td J (r,\rho)   u(t,\rho)\rd \rho-  \int_{2\bar h(t)/3}^{\bar h(t)}   \td J (r,\rho)  \bar u(t,\rho)\rd \rho\geq 0.
	\end{align}
	A direct computation gives
	\begin{align*}
	&\int_{-\yy}^{\bar h(t)}    J_n (r-\rho)  \bar u(t,\rho)\rd \rho-  \int_{2\bar h(t)/3}^{\bar h(t)}   \td J (r,\rho)  \bar u(t,\rho)\rd \rho-\int_0^{2\bar h(t)/3}   \td J (r,\rho)  u(t,\rho)\rd \rho\\
	=&\int_{-\yy}^{\bar h(t)}    J_n (r-\rho)  \bar u(t,\rho)\rd \rho- \int_{2\bar h(t)/3}^{\bar h(t)}  [ \td J_+ (r,\rho)+\td J_- (r,\rho)]  \bar u(t,\rho)\rd \rho
-\int_0^{2\bar h(t)/3}   \td J (r,\rho)  \td u(t,\rho)\rd \rho\\
=&\ Q_1+Q_2,
\end{align*}
with
\begin{align*}
Q_1:&=\int_{2\bar h(t)/3}^{\bar h(t)}    [J_n (r-\rho)-\td J_+ (r,\rho) ]  \bar u(t,\rho)\rd \rho\\
Q_2:&=\int_{-\yy}^{2\bar h(t)/3}    J_n (r-\rho)  \bar u(t,\rho)\rd \rho-  \int_{0}^{2\bar h(t)/3}   \td J (r,\rho)   u(t,\rho)\rd \rho-\int_{2\bar h(t)/3}^{\bar h(t)} \td J_-(r,\rho)  \bar u(t,\rho)\rd \rho.
	\end{align*}
	where $\td J_+$ and $\td J_-$ are defined as in Lemma \ref{lemma2.2}.

For $r\in (\bar g(t),\bar h(t))$ and $\rho \in (2\bar h(t)/3,\bar h(t))$ we have $r\geq \bar g(t)\geq K$ and
	\begin{align*}
	\frac{\rho}{r}<\frac{\bar h(t)}{\bar g(t)}=1+\epsilon_n/2, \ \    \frac{\rho}{r}\geq \frac{2\bar h(t)/3}{\bar h(t)}=\frac{2}{3}>\frac{1}{2},
	\end{align*}
	which allows us to apply Lemma \ref{lemma2.4} to conclude that 
	\begin{align*}
	J_n (r-\rho)-\td J_+ (r,\rho)\geq 0 \mbox{ for }  r\in (g(t),h(t)),\; \rho \in (2\bar h(t)/3,\bar h(t)),
	\end{align*}
	provided that $K$ is sufficiently large, say $K\geq L_{\epsilon_n}$; and so $Q_1\geq 0$.
	
	We now examine $Q_2$.  Using the facts that $\bar u(t,\rho)$ is decreasing in $\rho\leq \bar h(t)$ and
	\begin{align*}
	(1+\epsilon_n)u_n^*\geq \bar u(t,\rho)\geq \bar u(t,2\bar h(t)/3) \geq u_n^*\geq  u(t,\rho) \ \mbox{ for } \ 0\leq \rho\leq \frac{2\bar h(t)}{3},
	\end{align*}
	we obtain 
	\begin{align*}
	Q_2&\geq u_n^*\lf[\int_{-\yy}^{2\bar h(t)/3}    J_n (r-\rho)  \rd \rho-  \int_{0}^{2\bar h(t)/3}   \td J (r,\rho)  \rd \rho-(1+\epsilon_n)\int_{2\bar h(t)/3}^{\bar h(t)} \td J_-(r,\rho)  \rd \rho\rr]\\ 
	&= u_n^*\lf[\int_{-\yy}^{2\bar h(t)/3}    (1+\epsilon_n)J_{\epsilon_n} (r-\rho)  \rd \rho-  \int_{0}^{2\bar h(t)/3}   \td J (r,\rho)  \rd \rho-
	(1+\epsilon_n)\int_{2\bar h(t)/3}^{\bar h(t)} \td J_-(r,\rho)  \rd \rho\rr].
	\end{align*}
	From the definition of $J_{\epsilon_n}$, we have 
	\begin{align*}
	J_{\epsilon_n}(l)=3J(l) \ \mbox{ when }  |l|\geq 1+\epsilon_n^{-1}.
	\end{align*}
	Without loss of generality we may assume that $L_{\epsilon_n}\geq 1+\epsilon_n^{-1}$. Choosing $K\geq 4 L_{\epsilon_n}$, then for all large $n$ we obtain
	\begin{align*}
	r-3\bar h(t)/2>\lf(1-\frac{2(1+\epsilon_n)}{3}\rr)\bar g(t)>\frac{1}{6}\bar g(t)\geq \frac{K}{4}\geq L_{\epsilon_n} \ \ \forall\ r\in (\bar g(t),\bar h(t)),\ t\geq 0,
	\end{align*}
	and so, for such $r$ and $t$, 
	\begin{align*}
	Q_2&\geq u_n^*\lf[\int_{-\yy}^{2\bar h(t)/3}   \frac{1}{2}J_{\epsilon_n}(r-\rho)  \rd \rho-  \int_{0}^{2\bar h(t)/3}   \td J (r,\rho)  \rd \rho\rr]\\
	&\ \ \ \ \ \ +u_n^*(1+\epsilon_n)\lf[\int_{-\yy}^{2\bar h(t)/3}   \frac{1}{2}J_{\epsilon_n}(r-\rho)  \rd \rho-\int_{2\bar h(t)/3}^{\bar h(t)} \td J_-(r,\rho) \rd \rho\rr]\\
	&= u_n^*\lf[\int_{\Omega_1}  \frac 32 J(|x_1^r-y|) \rd y- \int_{\Omega_2}  J(|x_1^r-y|) \rd y \rr] \\
	&\ \ \ \ \ \ +u_n^*(1+\epsilon_n)\lf[\int_{\Omega_1} \frac 32J(|x_1^r-y|) \rd y- \int_{\Omega_3}  J(|x_1^r-y|) \rd y \rr] 
	\end{align*}
where $x_1^r:=(r,0,\cdots,0)$ and
	\begin{align*}
	&\Omega_1:=\{z=(z_1,z_2,\cdots,z_N):z_1<{2\bar h(t)}/{3}\},\ \Omega_2:=\{z: |z|<{2\bar h(t)}/{3}\},\\
	&\Omega_3:=\{z=(z_1,z_2,\cdots,z_N): {2\bar h(t)}/{3}< |z|<\bar h(t),\ z_1<0\}.
	\end{align*}
	Clearly, $\Omega_2\subset \Omega_1$ and $\Omega_3\subset \Omega_1$. Hence $Q_2\geq 0$. Therefore, \eqref{4.8d} holds.

	\textbf{Step 3}. We verify the  second   inequality of \eqref{4.6d}. 
			
	By Lemmas \ref{lemma2.6d} and \ref{lemma2.7d}, 
	we have for fixed large $n$ and all large $K>0$,
	\begin{align*}
	&\frac{\mu}{\bar h^{N-1}(t)}\dd\int_{0}^{\bar h(t)} r^{N-1}\bar u(t,r)\int_{\bar h(t)}^{+\yy} \td J(r,\rho)\rd \rho\rd r\\
	\leq &\ \mu\dd\int_{0}^{\bar h(t)} \bar u(t,r)\int_{\bar h(t)}^{+\yy} \td J(r,\rho)\rd \rho\rd r\\
	= &\ \mu \int_{(1-\epsilon_n/4)\bar h(t)}^{\bar h(t) }\bar u(t,r) \int_{\bar h(t)}^{(1+\epsilon_n/8)\bar h(t)} \td J_+(r,\rho) \rd \rho \rd r+o_K(1),
	\end{align*}
where $o_K(1)\to 0$ as $K\to\infty$.
This, together with Lemma \ref{lemma2.4}, gives  
	\begin{align*}
&	\frac{\mu}{\bar h^{N-1}(t)}\dd\int_{0}^{\bar h(t)} r^{N-1}\bar u(t,r)\int_{\bar h(t)}^{+\yy} \td J(r,\rho)\rd \rho\rd r\\
\leq &\ \mu \int_{(1-\epsilon_n/4)\bar h(t)}^{\bar h(t)} \bar u(t,r)\int_{\bar h(t)}^{(1+\epsilon_n/8)\bar h(t)} (1+\sqrt{\epsilon_n})J_{\epsilon_n}(r-\rho) \rd \rho \rd r+o_K(1)\\
=&\ \mu (1+\epsilon_n) \int_{(1-\epsilon_n/4)\bar h(t)}^{\bar h(t)} \phi_n(r-\bar h(t))\int_{\bar h(t)}^{(1+\epsilon/8)\bar h(t)} J_n(r-\rho) \rd \rho \rd r+o_K(1)\\
=&\ \mu (1+\epsilon_n) \int_{-\epsilon_n \bar h(t)/4}^{0} \phi_n(r)\int_{0}^{\epsilon_n\bar h(t)/8} J_n(r-\rho) \rd \rho \rd r+o_K(1)\\
\leq &\ \mu (1+\epsilon_n)  \int_{-\yy}^{0} \phi_n(r)\int_{0}^{\yy} J_n(r-\rho) \rd \rho \rd r+o_K(1)\\
	=&\ c_n(1+\epsilon_n)+o_K(1)\leq c_n(1+2\epsilon_n)= \bar h'(t),
	\end{align*}
	provided that $K$ is sufficiently large.
This proves the second inequality of \eqref{4.6d}. 	
\end{proof}

Clearly Theorem \ref{th1.6} follows directly from Theorems \ref{thm5.3} and \ref{thm5.5}.

\section{Logarithmic shift}
In this section, we prove Theorem \ref{th1.7}. So throughout this section, we assume that the kernel function $J$ satisfies {\bf (J)} and has compact support contained in the ball $B_{K_*}$, the function $f$ satisfies {\bf (f)} and is $C^2$, the initial function satisfies \eqref{u_0}, and spreading happens for the unique positive solution $(u,h)$ of \eqref{1.4}. Hence by Theorem \ref{th1.6} we have $h(t)/t\to c_0$ as $t\to\infty$. We are going to show that
$c_0t-h(t)\approx \ln t$ as $t\to\infty$.

\subsection{Upper bound of $h(t)-c_0t$}
\begin{lemma}\label{lemma3.10}
	There exists $C>0$ such that
	\begin{align}\label{3.20}
		h(t)-c_0t\leq  -C\ln t \mbox{ for } t\gg 1, 
	\end{align}
	where $c_0>0$ is given by Proposition \ref{prop1.4}. 
\end{lemma}
\begin{proof}
Let $(c_0,\phi_0)$ be the solution of the semi-wave problem given in Proposition \ref{prop1.4}.	We define
	\begin{align*}\begin{cases}
	\bar h(t):=c_0 t+\delta(t),  & t\geq 0,\\
	\bar u(t,r):=(1+\epsilon(t)) \phi_0(r-\bar h(t)), & t\geq 0,\  0\leq r\leq \bar h(t),
	\end{cases}
	\end{align*}
	with 
	\begin{align*}
	\epsilon(t):=K_1(t+\theta)^{-1}, \ 	\ \  \delta(t) :=c_0\theta -K_2[\ln (t+\theta)-\ln \theta]
	\end{align*}
	for some positive constants $\theta$, $K_1$ and $K_2\in (0,1)$ to be  determined.   Clearly, for large $\theta>0$,
	\begin{align}\label{3.21}
c_0(t+\theta)	 \geq \bar h(t)\geq \frac{c_0}{2}(t+\theta) \ \mbox{ for all } \ t\geq 0.
	\end{align}

	Next we choose suitable $\theta$,  $K_1$, $K_2$ and $t_0>0$ such that $(\bar u,\bar h)$ satisfies
	\begin{equation}\label{3.22}
	\begin{cases}
	\dd \bar u_t(t,r)\geq d \int_{0}^{\bar h(t)} \td J(r,\rho) \bar u(t,\rho)\rd \rho-d\bar u(t,r)+f(\bar u(t,r)), & t>0,\; r\in  (\bar h(t)/2,\bar h(t)),\\[3mm]
	\dd \bar h'(t)\geq  \frac{\mu}{\bar h^{N-1}(t)}\dd\int_{0}^{\bar h(t)} r^{N-1}\bar u(t,r)\int_{\bar h(t)}^{+\yy} \td J(r,\rho)\rd \rho\rd r, & t>0,\\
	\bar u(t,r)\geq  u(t+t_0,r), \; \bar u(t,\bar h(t))=0, & t>0,\; r\in [0,\bar h(t)/2],\\
	\bar u(0,r)\geq  u(t_0,r),\ h(t_0)\leq \bar h(0), & r\in [0,h(t_0)].
	\end{cases}
	\end{equation}
	
	If \eqref{3.22} is proved, then we can use Lemma \ref{lemma3.4a} to obtain
	\begin{align*}\begin{cases}
	 h(t+t_0)\leq  \bar h(t) & \mbox{ for } t\geq 0,\\
	 u(t+t_0,r)\leq  \ol u(t,r)& \mbox{ for } t\geq 0,\ r\in [ 0,h(t+t_0)],
	 \end{cases}
	\end{align*}
which implies \eqref{3.20}.

Therefore to prove the lemma, it suffices to show \eqref{3.22}. For clarity we break the proof of \eqref{3.22}  into several steps.
	
	{\bf Step 1}. We choose $t_0=t_0(\theta)$ and $K_2$ such that the last two inequalities of \eqref{3.22} hold. 
	
	For the ODE problem
	\begin{align*}
	v'=f(v),\ \ \ v(0)=u^*+\epsilon_1
	\end{align*}
	with small $\epsilon_1>0$, from $f'(u^*)<0$, we see that  
	\begin{align*}
	u^*<v(t)\leq u^*+\epsilon_1e^{\wtd F t} \mbox{ for all } t\geq 0,
	\end{align*} 
	with $\wtd F=\max_{u\in [u^*,u^*+\epsilon_1]}f'(u)<0$. A simple comparison argument shows that there is $t_*>0$ such that $ u(t,r)\leq u^*+\epsilon_1$ for $t\geq t_*$ and $r\in [0, h(t)]$. Using comparison again we obtain 
	\begin{align*}
	 u(t+t_*,r)\leq v(t)\leq u^*+\epsilon_1e^{\wtd F t}\ \mbox{ for all } t\geq 0,\ r\in [0, h(t)].
	\end{align*}
On the other hand, by \cite[Theorem 1.7 (iii)]{dn-speed}, there is $\beta>0$ and $C_1>0$ such that 
\begin{align*}
u^*-\phi_0(r)<C_1e^{\beta \, r} \mbox{ for } r<0,
\end{align*}	
and hence, by \eqref{3.21}, for $t\geq 0$ and $r\in [0,(\bar h(t)/2]$,
\begin{align*}
\bar u(t,r)& =(1+\epsilon(t)) \phi_0(r-\bar h(t))\geq (1+\epsilon(t)) \phi_0(-\bar h(t)/2)\\
&\geq (1+K_1 (t+\theta)^{-1})(u^*-C_1e^{ -\beta\bar h(t)/2})\\
&\geq  (1+K_1 (t+\theta)^{-1})u^*-2C_1e^{ -\beta\bar h(t)/2}\\
&\geq(1+K_1 (t+\theta)^{-1})u^*-2C_1e^{ -\beta c_0(t+\theta)/4}\\
&\geq u^*+\epsilon_1 e^{\wtd F (t+t_0-t_*)}
\geq \td u(t+t_0,r)
\end{align*}
 provided 
\begin{align}\label{3.23}
2C_1e^{ -\beta c_0(t+\theta)/4}+\epsilon_1e^{\wtd F (t+t_0-t_*)}\leq K_1 (t+\theta)^{-1}u^* \mbox{ for all } t\geq 0,
\end{align}
which holds true when $\theta$ and $t_0$ are sufficiently large. We may at the same time also require
 \begin{align}\label{3.24}
 h(t_0)\leq \bar h(0)/2=c_0\theta/2.
 \end{align}
  So in particular we have
\begin{align*}
\bar u(0,r)\geq u(t_0,r) \mbox{ for } r \in [0,h(t_0)]. 
\end{align*}

To be more precise, 
by Theorem \ref{th1.6}, there is $C_2>0$ such that $h(t)\leq 2c_0 t+C_2$ for all $t\geq 0$.
Then \eqref{3.23} and \eqref{3.24} hold if 
\begin{align}
t_0=\frac{c_0\theta-2C_2}{4c_0}
\end{align}
and $\theta$ is sufficient large.

{\bf Step 2}. We check the second inequality of \eqref{3.22}.  

Using  Lemma \ref{lemma2.2} and \eqref{1.10},  for $r+\rho> K_*$, we have
\begin{equation}\label{3.26}
\begin{cases}
\dd\td J(r,\rho)\leq \lf(\frac{\rho}{r}\rr)^{(N-1)/2}J_*(r-\rho),& N\geq 3,\\
\dd \td J(r,\rho)\leq \lf(\frac{\rho}{r}\rr)^{1/2}\lf[\frac{(\rho+r)^2-K_*^2}{4r\rho }\rr]^{-1/2}J_*(r-\rho),& N=2.
\end{cases}
\end{equation}
Using \eqref{3.26},
	we deduce for $N\geq 3$,
	\begin{align*}
	&\frac{\mu}{\bar h^{N-1}(t)}\dd\int_{0}^{\bar h(t)} r^{N-1}\bar u(t,r)\int_{\bar h(t)}^{+\yy} \td J(r,\rho)\rd \rho\rd r\\
	&=\frac{\mu}{\bar h^{N-1}(t)}\dd\int_{0}^{\bar h(t)} (1+\epsilon(t)) r^{N-1}\phi_0(r-\bar h(t))\int_{\bar h}^{+\yy} \td J(r,\rho)\rd \rho\rd r\\
	&\leq \frac{\mu (1+\epsilon(t))}{\bar h^{N-1}(t)}\dd\int_{0}^{\bar h(t)}  r^{(N-1)/2}\phi_0(r-\bar h(t))\int_{\bar h(t)}^{+\yy}  \rho^{(N-1)/2} J_*(r-\rho)\rd \rho\rd r\\
	&=\frac{\mu(1+\epsilon(t))}{\bar h^{N-1}(t)}\dd\int_{-\bar h(t)}^{0} (r+\bar h(t))^{(N-1)/2}\phi_0(r)\int_{0}^{+\yy}  (\rho+\bar h(t))^{(N-1)/2} J_*(r-\rho)\rd \rho\rd r\\
	&=\mu(1+\epsilon(t))\dd\int_{-\bar h(t)}^{0} \int_{0}^{+\yy} \phi_0(r) J_*(r-\rho)\rd \rho\rd r\\
	&\ \ \ \ \ \ +\mu(1+\epsilon(t))\dd\int_{-\bar h(t)}^{0} \int_{0}^{+\yy}  \lf[\Big(1+\frac r{\bar h(t)}\Big)^{(N-1)/2}\Big(1+\frac{\rho}{\bar h(t)}\Big)^{(N-1)/2}-1\rr]\phi_0(r)  J_*(r-\rho)\rd \rho\rd r\\
	&=:\mu(1+\epsilon(t))\left[\dd\int_{-\bar h(t)}^{0}\int_{0}^{+\yy}    \phi_0(r)J_*(r-\rho)\rd \rho\rd r+A(N)\right]\\
	&\leq (1+\epsilon(t))[c_0+\mu A(N)].
	\end{align*}
	
	For $N=2$, we similarly obtain
	\begin{align*}
		&\frac{\mu}{\bar h^{N-1}(t)}\dd\int_{0}^{\bar h(t)} l^{N-1}\bar u(t,r)\int_{\bar h(t)}^{+\yy} \td J(r,\rho)\rd \rho\rd r\\
		\leq&\ (1+\epsilon(t))\left[c_0+ \mu A(2)\right],
			\end{align*}
			with
			\begin{align*}
		 A(2):=\!\!\dd\int_{-\bar h(t)}^{0}  \int_{0}^{+\yy}\!\!  \lf(\Big[1+\frac r{\bar h(t)}\Big]^{\frac 12}\Big[1+\frac \rho{\bar h(t)}\Big]^{\frac 12}\lf[\frac{(\rho+r+2\bar h(t))^2-K_*^2}{4(r+\bar h(t))(\rho+\bar h(t)) }\rr]^{-\frac 12}\!\!\!-\!1\!\rr)\phi_0(r) J_*(r-\rho)\rd \rho\rd r.
	\end{align*}
	
	{\bf Claim 1}.  There exist constants $C_{i,j}\geq 0$, with $C_{0,1}=C_{1,0}>0$, such that  for   $x, y\in \R$ close to 0,
	\begin{align}\label{3.27}
	(1+x)^{(N-1)/2}(1+y)^{(N-1)/2}-1\leq \sum_{1\leq i+j\leq N-1}C_{i,j}x^iy^j.
	\end{align}
	If $N\geq 2$ is an odd integer, \eqref{3.27} follows by expanding the product in its left side.
	 When $N\geq 2$ is an even integer,  for  $x\in \R$ and $y\in \R$ close to 0, we have
	\begin{align*}
	(1+x)^{(N-1)/2}(1+y)^{(N-1)/2}-1
	&=(1+x)^{(N-2)/2}(1+y)^{(N-2)/2}(1+x)^{1/2}(1+y)^{1/2}-1\\
	&\leq(1+x)^{(N-2)/2}(1+y)^{(N-2)/2}(1+x/2)(1+y/2) -1,
	\end{align*}
	since 
	\[
	(1+x)^{1/2}=1+\frac 12 x-\frac 18 (1+\xi)^{-3/2}x^2\leq 1+\frac 12 x
	\]
	for some $\xi$ satisfying $|\xi|\leq |x|$.
	Then \eqref{3.27} is obtained by expanding the last term of the earlier inequality.

	{\bf Claim 2}.  For integer $N\geq 2$, there are  constants $C_4$ and $C_5$ depending on $N$ such that 
	\begin{align*}
	A(N)\leq \frac{C_4}{\bar h^{2}(t)}+\frac{C_5B}{\bar h(t)} \mbox{ for all large $\theta$ and $t\geq 0$},
	\end{align*}
	where
	\begin{align*}
	B:=\dd\int_{-K_*}^{0}\int_{0}^{K_*}  (r+\rho) \phi_0(r) J_*(r-\rho)\rd \rho\rd r.
	\end{align*}
	
By our assumption, the supporting set of $J_*$ in contained in $ [-K_*,K_*]$, and so  for $N\geq 3$ and large $\bar h(t)$, by \eqref{3.27} we have
	\begin{align*}
	A(N)\leq &\dd\sum_{1\leq  i+j\leq N-1}C_{i,j}\int_{-\bar h(t)}^{0} \frac{ r^{i}\rho^{j}}{\bar h^{i+j}(t)}\phi_0(r)\int_{0}^{+\yy}  J_*(r-\rho)\rd \rho\rd r\\
	=&\dd\sum_{1\leq  i+j\leq N-1}C_{i,j}\int_{-K_*}^{0}\int_{0}^{K_*} \frac{ r^{i}\rho^{j}}{\bar h^{i+j}(t)}\phi_0(r) J_*(r-\rho)\rd \rho\rd r\\
	=&\dd\sum_{2\leq  i+j\leq N-1}C_{i,j}\int_{-K_*}^{0}\int_{0}^{K_*}\frac{ r^{i}\rho^{j}}{\bar h^{i+j}(t)}\phi_0(r)J_*(r-\rho)\rd \rho\rd r\\
	&+C_{0,1}\int_{-K_*}^{0}\int_{0}^{K_*} \frac{ r+\rho}{\bar h(t)}\phi_0(r)\int_{0}^{+\yy}  J_*(r-\rho)\rd \rho\rd r\\
	\leq&\sum_{2\leq  i+j\leq N-1}C_{i,j}\frac{K_*^{i+j+2}\|J_*\|_{L^\yy}u^*}{\bar h^{2}(t)}+C_{0,1}\dd\int_{-K_*}^{0}\int_{0}^{K_*}  \frac{(r+\rho)  }{\bar h(t)}\phi_0(r) J_*(r-\rho)\rd \rho\rd r\\
	=:&\frac{C_4}{\bar h^{2}(t)}+\frac{C_5B}{\bar h(t)}.
	\end{align*}

Next we consider $A(2)$. Clearly, for small $x\geq 0$, there is $\xi\in [0,x]$ such that 
	\begin{align}\label{3.28}
	(1-x)^{-1/2}=1+\frac{1}{2} (1-\xi)^{-3/2}x \leq 1+\frac{1}{2} (1-1/2)^{-3/2}x= 1+\sqrt{2} x.
	\end{align}
	Using this and \eqref{3.27} we obtain, for large $\bar h(t)$,
	\begin{align*}
	 A(2)=&\int_{-K_*}^{0}\int_{0}^{K_*}  \lf(\Big[1+\frac r{\bar h(t)}\Big]^{\frac 12}\Big[1+\frac \rho{\bar h(t)}\Big]^{\frac 12} \lf[\frac{(\rho+r+2\bar h)^2-K_*^2}{4(r+\bar h)(\rho+\bar h) }\rr]^{-1/2}-1\rr)\phi_0(r)J_*(r-\rho)\rd \rho\rd r\\
	\leq&\dd\int_{-K_*}^{0}\int_{0}^{K_*} \lf(\Big[1+\frac r{\bar h(t)}\Big]^{\frac 12}\Big[1+\frac \rho{\bar h(t)}\Big]^{\frac 12}\lf[1-\frac{K_*^2}{4(r+\bar h(t))(\rho+\bar h(t)) }\rr]^{-1/2}-1\rr)
	\phi_0(r)  J_*(r-\rho)\rd \rho\rd r\\
	\leq&\dd\int_{-K_*}^{0}\int_{0}^{K_*} \lf(\Big[1+\frac r{2\bar h(t)}\Big]\Big[1+\frac \rho{2\bar h(t)}\Big] \lf[1+\frac{\sqrt{2}K_*^2}{4(r+\bar h(t))(\rho+\bar h(t)) }\rr]-1\rr)\phi_0(r)  J_*(r-\rho)\rd \rho\rd r\\
	=&\dd\int_{-K_*}^{0}\int_{0}^{K_*} \lf(\frac{r+\rho}{2\bar h(t)}+\frac{r\rho}{4\bar h^2(t)}\rr)\phi_0(r))  J_*(r-\rho)\rd \rho\rd r\\
	&+\dd\int_{-K_*}^{0}\int_{0}^{K_*} \Big[1+\frac r{2\bar h(t)}\Big]\Big[1+\frac\rho{2\bar h(t)}\Big]\frac{\sqrt{2}K_*^2}{4(r+\bar h(t))(\rho+\bar h(t)) }\phi_0(r) J_*(r-\rho)\rd \rho\rd r\\
	\leq&\ \frac{B}{2\bar h(t)}+\frac{K_*^{4}\|J_*\|_{L^\yy}u^*}{4\bar h^{2}(t)}+\dd\int_{-K_*}^{0}\int_{0}^{K_*} \frac{\sqrt{2}K_*^2}{2(r+\bar h(t))(\rho+\bar h(t)) }\phi_0(r)
	 J_*(r-\rho)\rd \rho\rd r\\
	\leq&\ \frac{B}{2\bar h(t)}+\frac{K_*^{4}\|J_*\|_{L^\yy}u^*}{4\bar h^{2}(t)}+\frac{\sqrt{2}K_*^{4}\|J_*\|_{L^\yy}u^*}{\bar h^{2}(t)}\\
	=&\ \frac{(1+4\sqrt{2})K_*^{4}\|J_*\|_{L^\yy}u^*}{4\bar h^{2}(t)}+\frac{B}{2\bar h(t)}=:\frac{C_4}{\bar h^{2}(t)}+\frac{C_5B}{\bar h(t)}.
	\end{align*}
	Thus Claim 2 is proved.
	
	{\bf Claim 3}. $B<0$. 

Since $J_*$ is even, we have
\begin{align*}
	B&=\dd\int_{-K_*}^{0}\int_{0}^{K_*} (r+\rho) \phi_0(r) J_*(r-\rho)\rd \rho\rd r\\
	&=\int_0^{K_*}\int_0^{K_*}(\rho-r)\phi_0(-r)J_*(\rho+r)d\rho dr\\
	&=\int_0^{K_*}\int_0^{r}(\rho-r)\phi_0(-r)J_*(\rho+r)d\rho dr+\int_0^{K_*}\int_r^{K_*}(\rho-r)\phi_0(-r)J_*(\rho+r)d\rho dr\\
	&=\int_0^{K_*}\int_\rho^{K_*}(\rho-r)\phi_0(-r)J_*(\rho+r)dr d\rho +\int_0^{K_*}\int_r^{K_*}(\rho-r)\phi_0(-r)J_*(\rho+r)d\rho dr\\
	&=\int_0^{K_*}\int_r^{K_*}(\rho-r)[\phi_0(-r)-\phi_0(-\rho)]J_*(\rho+r)d\rho dr<0,
\end{align*}  
since $r\to\phi_0(-r)$ is strictly increasing. Claim 3 is thus proved.

In view of $\bar h(t)\geq c_0\theta$, from Claims 2 and 3 we obtain
	\[
	A(N)\leq \frac{ C_5 B}{2\bar h(t)} \mbox{ for all large $\theta$.}
	\]
	 It follows that
	\begin{align*}
	\frac{\mu}{\bar h^{N-1}(t)}\dd\int_{0}^{\bar h(t)} r^{N-1}\bar u(t,r)\int_{\bar h(t)}^{+\yy} \td J(r,\rho)\rd \rho\rd r
	\leq(1+\epsilon(t))c_0+\frac{\mu C_5 B}{2\bar h(t)}
\end{align*}
provided that $\theta$ is chosen large enough.
  Then from $\bar h(t)\leq c_0(t+\theta)$, we obtain
\begin{align*}
&\frac{\mu}{\bar h^{N-1}(t)}\dd\int_{0}^{\bar h(t)} r^{N-1}\bar u(t,r)\int_{\bar h(t)}^{+\yy} \td J(r,\rho)\rd \rho\rd r
\leq c_0+\frac{K_1c_0}{t+\theta}+\frac{\mu C_5 B}{2c_0(t+\theta)}\\
\leq&c_0-K_2(t+\theta)^{-1}=\bar h'(t)
\end{align*}
if $K_1$ and $K_2$ are  small such that 
\begin{align*}
K_1c_0+K_2\leq \frac{-\mu C_5 B}{2c_0}.
\end{align*}
 
	{\bf Step 3}. We verify the first inequality of \eqref{3.22}, namely,  for $t>0$ and $r\in (\overline h(t)/2,\overline h(t))$, 
	\begin{align}\label{3.30}
	\bar u_t(t,r)\geq d \int_{0}^{\bar h(t)} \td J(r,\rho) \bar u(t,\rho)\rd \rho -d\bar u(t,x)+f(\bar u(t,r)).
	\end{align}
	
We start with a claim.

{\bf Claim 4.} There exist positive constants $C_6$ and $C_7$ such that for  all large $\theta$ and $r\in (\bar h(t)/2,\bar h(t))$, $t>0$,
\begin{equation}\begin{aligned}\label{3.31}
\int_{0}^{\bar h(t)}  \td J(r,\rho) \phi_0(\rho-\bar h(t))\rd \rho\leq &\ \int_{0}^{\bar h(t)}  J_*(r-\rho) \phi_0(\rho-\bar h(t))\rd \rho \\ & +	\frac{C_6}{r}\int_{-K_*}^{K_*}  \rho J_*(\rho) 
\phi_0(\rho+r-\bar h(t))\rd \rho+\frac{C_7}{r^2 }.
\end{aligned}
\end{equation}

We prove \eqref{3.31} for the cases $N\geq 3$ and $N=2$ separately.  Note that $r\geq \bar h(t)/2\geq c_0\theta/2$ is large for all large $\theta$.

For $N\geq 3$, $r\in (\underline h(t)/2,\underline h(t))$ and $t>0$,  by \eqref{3.26}, when $\theta$ is chosen sufficiently large,
	\begin{align*}
	&\int_{0}^{\bar h(t)}  \td J(r,\rho) \phi_0(\rho-\bar h(t))\rd \rho\leq 	\int_{0}^{\bar h(t)} \lf(\frac{\rho}{r}\rr)^{(N-1)/2}  J_*(r-\rho) \phi_0(\rho-\bar h(t))\rd \rho\\
	=&\int_{0}^{\bar h(t)} J_*(r-\rho) \phi_0(\rho-\bar h(t))\rd \rho+\int_{0}^{\bar h(t)} \lf[\lf(\frac{\rho}{r}\rr)^{(N-1)/2}-1\rr] J_*(r-\rho) \phi_0(\rho-\bar h(t))\rd \rho\\
	=&\int_{0}^{\bar h(t)} J_*(r-\rho) \phi_0(\rho-\bar h(t))\rd \rho+\int_{-r}^{\bar h(t)-r} \lf[\lf(1+\frac{\rho}{r}\rr)^{(N-1)/2}-1\rr]J_*(\rho) \phi_0(\rho+r-\bar h(t))\rd \rho\\
	=&\int_{0}^{\bar h(t)}  J_*(r-\rho) \phi_0(\rho-\bar h(t))\rd \rho+\int_{-K_*}^{\min\{K_*,\bar h(t)-r\}} \lf[\lf(1+\frac{\rho}{r}\rr)^{(N-1)/2}-1\rr] J_*(\rho) \phi_0(\rho+r-\bar h(t))\rd \rho\\
	=&\int_{0}^{\bar h(t)}  J_*(r-\rho) \phi_0(\rho-\bar h(t))\rd \rho+\int_{-K_*}^{K_*} \lf[\lf(1+\frac{\rho}{r}\rr)^{(N-1)/2}-1\rr] J_*(\rho) \phi_0(\rho+r-\bar h(t))\rd \rho,
	\end{align*}
where we have assumed that $\phi$ is extended by $\phi(r)\equiv 0$ for $r>0$.  By elementary calculus, 
there exist positive constants $D_1$ and $D_2$, depending on $N$, such that for all $x\in \R$ close to 0,
	\begin{align*}
\lf(1+x\rr)^{(N-1)/2}-1\leq  D_1 x+D_2 x^2.
	\end{align*}
	 Hence,  for all large $\theta$ we have
	\begin{align*}
	&\int_{-K_*}^{K_*} \lf[\lf(1+\frac{\rho}{r}\rr)^{(N-1)/2}-1\rr] J_*(\rho) \phi_0(\rho+r-\bar h(t))\rd \rho\\
	\leq& \ \frac{D_1}{r} 	\int_{-K_*}^{K_*} \rho J_*(\rho) \phi_0(\rho+r-\bar h(t))\rd \rho+\frac{2D_2K_*^3\|J_*\|_{L^\yy}u^*}{r^2}.
	\end{align*}
Therefore \eqref{3.31} holds when $N\geq 3$.
	
	When $N=2$,  from \eqref{3.26}, we have for $r\in (\bar h(t)/2,\bar h(t))$, $t>0$ and large $\theta$,
	\begin{align*}
	&\int_{0}^{\bar h(t)}  \td J(r,\rho) \phi_0(\rho-\bar h(t))\rd \rho\leq 	\int_{0}^{\bar h(t)}\lf(\frac{\rho}{r}\rr)^{1/2}\lf[\frac{(\rho+r)^2-K_*^2}{4r\rho }\rr]^{-1/2} J_*(r-\rho) 
	\phi_0(\rho-\bar h(t))\rd \rho\\
	=& \int_{\bar h(t)/2-K_*}^{\bar h(t)}\lf(\frac{\rho}{r}\rr)^{1/2}\lf[\frac{(\rho+r)^2-K_*^2}{4r\rho }\rr]^{-1/2} J_*(r-\rho) 
	\phi_0(\rho-\bar h(t))\rd \rho\\
	\leq& \int_{\bar h(t)/2-K_*}^{\bar h(t)}\lf(\frac{\rho}{r}\rr)^{1/2}\lf(1-\frac{K_*^2}{4r\rho }\rr)^{-1/2} J_*(r-\rho) \phi_0(\rho-\bar h(t))\rd \rho\\
	=&\int_{0}^{\bar h(t)}J_*(r-\rho) \phi_0(\rho-\bar h(t))\rd \rho+\int_{0}^{\bar h(t)}\lf[\lf(\frac{\rho}{r}\rr)^{1/2}\lf(1-\frac{K_*^2}{4r\rho }\rr)^{-1/2}-1\rr]J_*(r-\rho) \phi_0(\rho-\bar h(t))\rd \rho\\
	=&\int_{0}^{\bar h(t)}J_*(r-\rho) \phi_0(\rho-\bar h(t))\rd \rho\\
	&+\int_{-K_*}^{K_*}  \lf[\lf(1+\frac{\rho}{r}\rr)^{1/2}\lf(1-\frac{K_*^2}{4r(\rho+r) }\rr)^{-1/2}-1\rr] J_*(\rho) \phi_0(\rho+r-\bar h(t))\rd \rho.
	\end{align*}
	Thus, by \eqref{3.28}, for $r\in (\underline h(t)/2,\underline h(t))$, $t>0$ and large $\theta$,
	\begin{align*}
	&\int_{-K_*}^{K_*}  \lf[\lf(1+\frac{\rho}{r}\rr)^{1/2}\lf(1-\frac{K_*^2}{4r(\rho+r) }\rr)^{-1/2}-1\rr] J_*(\rho) \phi_0(\rho+r-\bar h(t))\rd \rho\\
	\leq& \int_{-K_*}^{K_*}  \lf[\lf(1+\frac{\rho}{2r}\rr)\lf(1+\frac{\sqrt{2}K_*^2}{4r(\rho+r) }\rr)-1\rr] J_*(\rho) \phi_0(\rho+r-\bar h(t))\rd \rho\\
	\leq&\int_{-K_*}^{K_*}  \lf[\lf(1+\frac{\rho}{2r}\rr)\lf(1+\frac{\sqrt{2}K_*^2}{2r^2 }\rr)-1\rr] J_*(\rho) \phi_0(\rho+r-\bar h(t))\rd \rho\\
	=&\int_{-K_*}^{K_*}  \lf[\frac\rho {2r}+\frac{\sqrt{2}K_*^2}{2r^2 }\left(1+\frac{\rho}{2r}\rr)\rr] J_*(\rho) \phi_0(\rho+r-\bar h(t))\rd \rho\\
	\leq &\ \frac{1}{2r}\int_{-K_*}^{K_*}  \rho J_*(\rho) \phi_0(\rho+r-\bar h(t))\rd \rho+\frac{2\sqrt{2}K_*^3\|J_*\|_{L^\yy}u^*}{r^2 },
	\end{align*}
	which gives \eqref{3.31} for $N=2$. Claim 4 is thus proved.
	
With the above estimates, we are ready to prove \eqref{3.30}.	By the definition of $\bar u$ and \eqref{3.31}, we have
	\begin{align*}
	\bar u_t(t,r)=&-(1+\epsilon(t))[c_0+\delta'(t)]\phi_0'(r-\bar h(t))+\epsilon'(t)\phi_0(r-\underline h(t)),
	\end{align*}
	and for $t>0$, $r\in (\bar h(t)/2,\bar h(t))$ and large $\theta$,
	\begin{align*}
	& -(1+\epsilon(t))c_0\phi'(r-\bar h(t))\\ 
	= & (1+\epsilon(t)) \lf[d \int_{-\yy}^{\bar h(t)}  J_*(r-\rho)  \phi_0(\rho-\bar h(t))\rd \rho-d\phi(r-\bar h(t))+f(\phi_0(r-\bar h(t)))\rr]\\ 
	\geq & (1+\epsilon(t)) \bigg[d \int_{0}^{\bar h(t)}  \td J(r,\rho)  \phi_0(\rho-\bar h(t))\rd \rho-d\phi_0(r-\bar h(t))+f(\phi_0(r-\bar h(t)))\\
	&\ \ \ \ \ \ \ \ -\frac{C_6}{r}\int_{-K_*}^{K_*}  \rho J_*(\rho) \phi_0(\rho+r-\bar h(t))\rd \rho-\frac{C_7}{r^2 }\bigg]\\
	= & \  d \int_{0}^{\bar h(t)}  \td J(r,\rho)  \bar u(t,\rho)\rd \rho-d\bar u(t,r)+(1+\epsilon(t))f(\phi_0(r-\bar h(t)))\\
	&\ -(1+\epsilon(t))\left[\frac{C_6}{r}\int_{-K_*}^{K_*}  \rho J_*(\rho) \phi_0(\rho+r-\bar h(t))\rd \rho+\frac{C_7}{r^2 }\right].
	\end{align*}
	Therefore
	\begin{align*}
	\bar u_t(t,r)\geq \ d \int_{0}^{\bar h(t)}  \td J(r,\rho)  \bar u(t,\rho)\rd \rho-d \bar u(t,r)+f(\bar u(t,r))+E,
	\end{align*}
	with
		\begin{align*}	
		E:=&(1+\epsilon(t))f(\phi_0(r-\bar h(t)))-f(\bar u(t,r))-(1+\epsilon(t))\left[\frac{C_6}{r}\int_{-K_*}^{K_*}  \rho J_*(\rho) \phi_0(\rho+r-\bar h(t))\rd \rho+\frac{C_7}{r^2 }\right]
		\\
		&-(1+\epsilon(t))\delta'(t)\phi_0'(r-\bar h(t))+\epsilon'(t)\phi_0(r-\underline h(t)).
	\end{align*}
	
	Clearly to complete the proof of \eqref{3.30}, it suffices to show the following claim.

{\bf Claim 5}. For suitably chosen small $K_1$ and $K_2$, and large $\theta$,
\begin{align*}
E\geq 0 \mbox{ for } r\in (\underline h(t)/2,\underline h(t)),\ t>0. 
\end{align*}

 Define
\begin{align*}
G(u):=(1+\epsilon)f(u)- f((1+\epsilon)u).
\end{align*}
Then for $u\in  [{0},{u}^*]$, there exists some $\td u\in [u,{u}^*]$ such that
\begin{align*}
G(u)=&G({u}^*)+ G'(\td u) (u-{u}^*)\\
=&-f((1+\epsilon){u}^*)+(1+\epsilon) f'(\td u) (u-{u}^*)-(1+\epsilon) f'((1+\epsilon)\td u) (u-{u}^*)\\
=&-f((1+\epsilon){u}^*)+(1+\epsilon)\bigg[ f'(\td u)- f'((1+\epsilon)\td u)\bigg] (u-{u}^*).
\end{align*}
 Since $f\in C^2$, there exists $C_f>0$ such that 
 \[
  f'(\td u)- f'((1+\epsilon)\td u\geq - C_f\td u\epsilon\geq -C_f u^*\epsilon.
  \]
  Hence,
\begin{align*}
G(u)\geq& -f((1+\epsilon){u}^*)-(1+\epsilon)C_f u^*\epsilon |u-{u}^*|\\
\geq&-\epsilon f'(u^*) u^*+o(\epsilon)-2 C_f u^* ({u}^*-u)\epsilon.
\end{align*}
Therefore, in view of $f'(u^*)<0$ and $\phi_0(-\infty)=u^*$, there exists large $L>0$ such that for $\xi\leq -L$ and $\theta$ large (and hence $\epsilon(t)$ small),
\begin{align}\label{3.33a}
(1+\epsilon(t))f(\phi_0(\xi))- f((1+\epsilon(t))\phi_0(\xi))\geq \frac{- f'(u^*) u^*}{2}\epsilon(t)>0.
\end{align}

Since $\phi_0$ is decreasing and $J_*$ is even, for $r\in [\bar h(t)/2, \bar h(t)]$ and $t>0$,
\begin{align}\label{6.13}
\int_{-K_*}^{K_*}  \rho J_*(\rho) \phi_0(\rho+r-\bar h(t))\rd \rho=\int_0^{K_*}\rho J_*(\rho)[\phi_0(\rho+r-\bar h(t))-\phi_0(-\rho+r-\bar h(t))]d\rho<0.
\end{align}
Denote
\begin{align*}
&C_8:=\dd\max_{r\leq 0}|\phi_0'(r)|,\\
&\tilde\phi_0(\rho):=\max_{\xi \in [-L, 0]}[\phi_0(\rho+\xi)-\phi_0(-\rho+\xi)].
\end{align*}
Then $\tilde\phi_0(\rho)$ is continuous and $\tilde\phi_0(\rho)<\tilde\phi_0(0)=0$ for $\rho\in (0, K_*]$. Therefore
\begin{align*}
\int_{-K_*}^{K_*}  \rho J_*(\rho) \phi_0(\rho+r-\bar h(t))\rd \rho&=\int_0^{K_*}\rho J_*(\rho)[\phi_0(\rho+r-\bar h(t))-\phi_0(-\rho+r-\bar h(t))]d\rho\\
&\leq \int_0^{K_*}\rho J_*(\rho)\tilde \phi_0(\rho)d\rho=:-C_9<0 \mbox{ for } r\in [\bar h(t)-L, \bar h(t)].
\end{align*}
We may now use $(1+\epsilon(t))f(\phi_0(r-\bar h(t)))\geq f(\bar u(t,r))$  and \eqref{3.21} to obtain, for $r \in [\bar h(t)-L,\bar h(t)]$, $t>0$ and large $\theta$,
	\begin{align*}
	E\geq &\ \frac{C_6C_9}{r}-\frac{2C_7}{r^2 }-2\delta'(t)C_8+u^*\epsilon'(t)\\
	\geq &\ 	\frac{C_6C_9}{\bar h(t)}-\frac{4C_7}{\bar h^2(t) }-2K_2C_8(t+\theta)^{-1}-u^*K_1(t+\theta)^{-2}\\
	\geq&\  \frac{C_6C_9}{c_0(t+\theta)}-\frac{16C_7}{c_0^2(t+\theta)^2 }-2K_2C_8(t+\theta)^{-1}-u^*K_1(t+\theta)^{-2}\\
	=&\ \frac{1}{t+\theta}[C_6C_9/c_0-16C_7c_0^{-2}(t+\theta)^{-1} -2K_2C_8-u^*K_1(t+\theta)^{-1}]\geq 0
	\end{align*}
if $K_1$ and $K_2$ are small and $\theta$ is large.

We next estimate $E$ for $r \in [\bar h(t)/2,\bar h(t)-L]$. For such $r$ and $t>0$, by \eqref{6.13} and \eqref{3.33a}, we obtain
\begin{align*}
E\geq& \ (1+\epsilon(t) )f(\phi_0(r-\bar h(t)))-f(\bar u(t,r))-\frac{4C_7}{\bar h^2(t) }-2K_2C_8(t+\theta)^{-1}-u^*K_1(t+\theta)^{-2}\\
\geq&\ \frac{- f'(u^*) u^*}{2} K_1(t+\theta)^{-1}-\frac{16C_7}{c_0^2(t+\theta)^2 }-2K_2C_8(t+\theta)^{-1}-u^*K_1(t+\theta)^{-2}\\
=&\ (t+\theta)^{-1}[- f'(u^*) u^*K_1/2-16C_7c_0^{-2}(t+\theta)^{-1} -2K_2C_8-u^*K_1(t+\theta)^{-1}]\geq 0
\end{align*}
if $\theta$ is large and
\begin{align*}
K_1=\frac{8C_8K_2}{- f'(u^*) u^*}.
\end{align*}
This finishes the proof of Claim 5 and hence the lemma.	
\end{proof}

\subsection{Lower bound of $h(t)-c_0t$}

This subsection is devoted to the proof of the following lemma, which, combined with Lemma \ref{lemma3.10}, gives Theorem \ref{th1.7}.

\begin{lemma}\label{lem-lb}
		There exists $\td C>0$ such that for $t\gg 0$,
		\begin{align}\label{3.36a}
		h(t)-c_0t \geq -\td C \ln t,
		\end{align}
		where $c_0>0$ is given by Proposition \ref{prop1.4}. 
\end{lemma}

The proof of Lemma \ref{lem-lb} is rather involved, and requires some preliminary results, given in the following two lemmas.
\begin{lemma}\label{lemma3.11}
There exists $H_*>0$ large so that for $h\geq r\geq H_*$,  the inequality
\begin{equation}\begin{aligned}\label{3.32}
&\int_{0}^{h}  \td J(r,\rho) \psi(\rho)\rd \rho- \int_{0}^{ h}  J_*(r-\rho) \psi(\rho)\rd \rho\\
&\geq \frac{N-1}r\int_{-K_*}^{\min\{h-r, K_*\}} 
\rho
J_*(\rho) \psi(\rho+r)\rd \rho -\frac{NK_*^2\|\psi\|_{L^\infty}}{r^2}
\end{aligned}
\end{equation}
holds for any nonnegative function $\psi\in C([0,h])$, where $N\geq 2$ is the dimension.
\end{lemma}
\begin{proof}  By Lemma \ref{lemma2.2} and \eqref{1.10}, we easily deduce, for $r+\rho> K_*$, 
	\begin{equation}\label{3.33}
\begin{cases}
\dd	\td J(r,\rho)\geq \lf(\frac{\rho}{r}\rr)^{(N-1)/2}\lf[\frac{(\rho+r)^2-K_*^2}{4r\rho }\rr]^{(N-3)/2}J_*(r-\rho),&N\geq 3,\\[2mm]
\dd \td J(r,\rho)\geq \lf(\frac{\rho}{r}\rr)^{1/2}J_*(r-\rho),&N=2.
\end{cases}
	\end{equation}
Therefore, for $N\geq 3$,
\begin{align*}
&\int_{0}^{ h}  \td J(r,\rho) \psi(\rho)\rd \rho\geq  	\int_{0}^{ h} \lf(\frac{\rho}{r}\rr)^{(N-1)/2}  
\lf[\frac{(\rho+r)^2-K_*^2}{4r\rho }\rr]^{(N-3)/2}
J_*(r-\rho) \psi(\rho)\rd \rho\\
=&\int_{0}^{ h} J_*(r-\rho) \psi(\rho)\rd \rho\\
&+\int_{0}^{ h} 
\lf[\lf(\frac{\rho}{r}\rr)^{(N-1)/2}  
\lf(\frac{(\rho+r)^2-K_*^2}{4l\rho }\rr)^{(N-3)/2}-1\rr]
J_*(r-\rho) \psi(\rho)\rd \rho\\
=&\int_{0}^{ h} J_*(r-\rho) \psi(\rho)\rd \rho\\
&+\int_{-K_*}^{\min\{h-r, K_*\}} 
\lf[\lf(1+\frac{\rho}{r}\rr)^{(N-1)/2}  
\lf(\frac{(\rho+2r)^2-K_*^2}{4r(r+\rho) }\rr)^{(N-3)/2}-1\rr]
J_*(\rho) \psi(\rho+r)\rd \rho.
\end{align*}
 A simple calculation gives, for $N\geq 3$ and $x\in  \R$ close to 0,
\begin{align*}
\lf(1+x\rr)^{(N-1)/2}=1+\frac{N-1}{2}x+\frac{(N-1)(N-3)}{8}(1+\xi)^{(N-5)/2}x^2\geq 1+\frac{N-1}{2}x
\end{align*}
for some $\xi$ lying between $0$ and $x$; and for $N\geq 3$ and small $x\geq 0$,
\begin{align*}
(1-x)^{(N-3)/2}= 1-\frac{N-3}{2}(1-\eta)^{(N-5)/2}x \geq 1-(N-3)x
\end{align*}
for some  $\eta\in [0,x]$. Moreover, for $N\geq 3$, $r\gg 1$ and $\rho\in [-K_*, K_*]$, 
\begin{align*}
\lf[\frac{(\rho+2r)^2-K_*^2}{4r(r+\rho) }\rr]^{(N-3)/2}\geq \lf[1-\frac{K_*^2}{4r(r+\rho) }\rr]^{(N-3)/2}.
\end{align*}
Therefore, for such $N, r$ and $\rho$, we have
\begin{align*}
&\lf(1+\frac{\rho}{r}\rr)^{(N-1)/2}  
\lf[\frac{(\rho+2l)^2-K_*^2}{4r(r+\rho) }\rr]^{(N-3)/2}-1\\
\geq& \lf(1+\frac{N-1}{2} \frac{\rho}{r}\rr)\lf(1-\frac{(N-3)K_*^2}{4r(r+\rho) }\rr)-1\\
=&\frac{N-1}{2} \frac{\rho}{r}-\lf(1+\frac{N-1}{2} \frac{\rho}{r}\rr)\frac{(N-3)K_*^2}{4r(r+\rho)}\\
\geq&\frac{N-1}{2} \frac{\rho}{r}-\frac{(N-3)K_*^2}{r^2}.
\end{align*}
It follows that, for $h\geq r\gg 1$,
\begin{align*}
&\int_{-K_*}^{\min\{h-r, K_*\}} 
\lf(\lf(1+\frac{\rho}{r}\rr)^{(N-1)/2}  
\lf[\frac{(\rho+2r)^2-K_*^2}{4r(r+\rho) }\rr]^{(N-3)/2}-1\rr)
J_*(\rho) \psi(\rho+r)\rd \rho\\
\geq&\int_{-K_*}^{\min\{h-r, K_*\}} 
\lf(\frac{N-1}{2} \frac{\rho}{r}-\frac{(N-3)K_*^2}{r^2}\rr)
J_*(\rho) \psi(\rho+r)\rd \rho\\
=&\ \frac{N-1}{2}\int_{-K_*}^{\min\{h-r, K_*\}} 
\frac{\rho}{r}
J_*(\rho) \psi(\rho+r)\rd \rho-\frac{(N-3)K_*^2}{r^2}\int_{-K_*}^{\min\{h-r, K_*\}} J_*(\rho) \psi(\rho+r)\rd \rho,\\
\geq &\ \frac{N-1}{2r}\int_{-K_*}^{\min\{h-r, K_*\}} 
{\rho}J_*(\rho) \psi(\rho+r)\rd \rho-\frac{(N-3)K_*^2\|\psi\|_{L^\yy}}{r^2},
\end{align*}
which implies \eqref{3.32} (for $N\geq 3$).

When $N=2$,  from \eqref{3.33} we obtain, for $h\geq r\gg 1$, 
\begin{align*}
&\int_{0}^{ h}  \td J(r,\rho) \psi(\rho)\rd \rho\geq  	\int_{0}^{ h}\lf(\frac{\rho}{r}\rr)^{1/2} J_*(r-\rho) \psi(\rho)\rd \rho\\
=&\int_{0}^{ h}J_*(r-\rho) \psi(\rho)\rd \rho+\int_{0}^{ h}\lf[\lf(\frac{\rho}{r}\rr)^{1/2}-1\rr]J_*(r-\rho) \psi(\rho)\rd \rho\\
=&\int_{0}^{ h}J_*(r-\rho) \psi(\rho)\rd \rho+\int_{-K_*}^{\min\{h-r, K_*\}}  \lf[\lf(1+\frac{\rho}{r}\rr)^{1/2}-1\rr] J_*(\rho) \psi(\rho+r)\rd \rho.
\end{align*}

For $x\in \R$ close to 0, we have
\begin{align*}
(1+x)^{1/2}=1+\frac{1}{2}x-\frac{1}{8}(1+\xi)^{-3/2}x^2\geq 1+\frac{1}{2}x-\frac{x^2}{4}
\end{align*}
for some $\xi$ lying between $0$ and $x$. Therefore, for $h\geq r\gg 1$, 
\begin{align*}
& \int_{-K_*}^{\min\{h-r, K_*\}}  \lf[\lf(1+\frac{\rho}{r}\rr)^{1/2}-1\rr] J_*(\rho) \psi(\rho+r)\rd \rho\\
&\geq \int_{-K_*}^{\min\{h-r, K_*\}} \lf[\frac{\rho}{2r}-\left(\frac{\rho}{2r}\right)^2\rr] J_*(\rho) \psi(\rho+r)\rd \rho\\
&\geq\frac{1}{2r}\int_{-K_*}^{\min\{h-r, K_*\}} \rho J_*(\rho) \psi(\rho+r)\rd \rho
-\frac{K_*^2\|\psi\|_{L^\yy}}{4r^2}.
\end{align*}
 This finishes the proof of the Lemma.
\end{proof}

Let us note that by the assumption \textbf{(f)}, there is $D_*>0$ such that 
\begin{align}\label{3.36}
f(u)\geq D_* \min\{u,u^*-u\} \mbox{ for } u\in [0,u^*].
\end{align}

Our next lemma gives a crucial first estimate for the solution $u(t,r)$ of \eqref{1.4}.
\begin{lemma}\label{lemma3.12}  
	Suppose spreading happens to the solution $(u,h)$ of \eqref{1.4}. Then there exist positive constants $E_1$,  $E_2$ and $\theta_1$  such that for any $\theta\geq \theta_1$, we can find $t_0>0$ depending on $\theta$,  such that
\begin{align*}
u(t+t_0,r)\geq u^*-\frac{E_2}{t+\theta}\ \mbox{ for all } \ t\geq 0,\ r\in [0,E_1(t+\theta)].
\end{align*}
\end{lemma}
\begin{proof}
Define, for some positive constants $E_1$, $\theta$ and $D_1\in (0, 2E_1\theta u^*)$    to be determined, 
\begin{align*}
&\underline h(t):=2E_1(t+\theta) \mbox{ for } t\geq 0,\\
&\underline u(t,r):=\begin{cases}
\dd u^*-\frac{D_1}{\underline h(t)},& r\in [0,\underline h(t)/2],\ t\geq 0,\\[2mm]
\dd 2\lf(u^*-\frac{D_1}{\underline h(t)}\rr)\lf(1-\frac{r}{\underline h(t)}\rr),& r\in [\underline h(t)/2,\underline h(t)],\ t\geq 0.
\end{cases}
\end{align*}
Clearly $\underline u(t,r)$ is continuous, nonnegative, and nonincreasing in $r$.

We will show that $\underline u$ satisfies 
\begin{align}\label{3.35}
	\dd \underline u_t(t,r)\leq d \int_{0}^{\underline h(t)} \td J(r,\rho) \underline u(t,\rho)\rd \rho-d\underline u(t,r)+f(\underline u(t,r))\ \mbox{ for } \  t>0,\; r\in  (0,\underline h(t))\setminus
	\left\{\frac{\underline h(t)}2\right\}.
\end{align}
Clearly $\underline u(0,r)\leq u^*-\frac{D_1}{2E_1\theta}$.
Since $ u(t,r)$ converges to $u^*$ locally uniformly for $r\in [0, \infty)$ as $t\to \yy$, and $\lim_{t\to \yy}h(t)/t=c_0$, if we choose $E_1$ small so that $2E_1<c_0$,
then we can find $t_0>0$, depending on $E_1$ and $\theta$, such that
\[\begin{cases}
h(t_0+t)>2E_1(t+\theta)=\underline h(t) \mbox{ for all } t\geq 0,\\ u(t_0, r)\geq \underline u(0,r) \mbox{ for } r\in [0, \underline h(0)].
\end{cases}
\]
We also have $u(t,\underline h(t))>0=\underline u(t_0+t,\underline h(t))$ for all $t\geq 0$.
Therefore, if \eqref{3.35} holds true, then we can  use the comparison principle (and Remark \ref{rmk3.6})
 over the region $\{(t,r): t\geq 0,\; r\in [0, \underline h(t)]\}$ to obtain
\begin{align*}
u(t_0+t,r)\geq \underline u(t,r) \ \mbox{ for } \ t\geq 0,\ r\in [0,\underline h(t)],
\end{align*}
and the desired estimate thus follows with $E_2:=\frac{D_1}{2E_1}$.

Thus, to complete the proof of the lemma, it remains to prove
 \eqref{3.35}. We do so   according to the following four cases:
\begin{align*}
&{\rm (i)}\ r\in \Big[0,\frac 12\underline h(t)-K_*\Big],\ \ \ \ \ \ \  \ \ \ \ \ \ \ \ \ \ {\rm (ii)}\ r\in \Big[ \frac 12\underline h(t)-K_*,\frac 12\underline h(t)+K_*\Big]\backslash \Big\{\frac 12\underline h(t)\Big\},\\
&{\rm (iii)}\ r\in \Big[\frac 12\underline h(t)+K_*, \underline h(t)-K_*\Big],\ \ \ \ {\rm (iv)}\ r\in [\underline h(t)-K_*,\underline h(t)].
\end{align*}

\underline{Case (i)} For  $r\in \Big[0,\frac 12\underline h(t)-K_*\Big]$ and $t>0$, since the supporting set of $J$ is contained in $B_{K_*}$, we easily see, with  $\underline h(t)\geq 2E_1\theta\gg 1$,
\begin{align*}
\int_{0}^{\underline h(t)} \td J(r,\rho) \rd \rho=\int_{0}^{+\yy} \td J(r,\rho) \rd \rho=1,
\end{align*}
and so by \eqref{3.36}, for such $r$ and $t$, we obtain
\begin{align*}
d \int_{0}^{\underline h(t)} \td J(r,\rho) \underline u(t,\rho)\rd \rho-d\underline u(t,r)+f(\underline u(t,r))=f(\underline u(t,r))\geq \frac{D_1D_*}{\underline h(t)}.
\end{align*}
Clearly
\begin{align*}
\underline u_t(t,r)=& \frac{D_1 \underline h'(t)}{\underline h^2(t)}=\frac{D_1 }{2E_1(t+\theta)^2}\leq  \frac{D_1D_*}{\underline h(t)} \mbox{ for } r \in \left[0,\frac 12\underline h(t)-K_*\right]
\end{align*}
provided $\theta$ is large enough. Thus \eqref{3.35} holds in this case if $\theta$ is chosen sufficiently large.

\underline{Case (ii)} For $r\in \Big[ \frac 12\underline h(t)-K_*,\frac 12\underline h(t)\Big]$ and $t>0$, with $\underline h(t)\geq 2E_1\theta\gg1$,
\begin{align*}
&\int_{0}^{\underline h(t)}  J_*(r-\rho) \underline u(t,\rho)\rd \rho-\underline u(t,r )=\int_{-K_*}^{K_*}  J_*(\rho) \underline u(\rho+r)\rd \rho-\lf(u^*-\frac{D_1}{\underline h(t)}\rr)\\
=&\int_{-K_*}^{\underline h(t)/2-r}  J_*(\rho) \underline u(t,\rho+r)\rd \rho+\int_{\underline h(t)/2-r}^{K_*}  J_*(\rho) \underline u(t,\rho+l)\rd \rho-\int_{-K_*}^{K_*}  J_*(\rho)\lf(u^*-\frac{D_1}{\underline h(t)}\rr)d\rho\\
=&\int_{\underline h(t)/2-r}^{K_*}  J_*(\rho) \underline u(t, \rho+r)\rd \rho-\int^{K_*}_{\underline h(t)/2-r}  J_*(\rho) \lf(u^*-\frac{D_1}{\underline h(t)}\rr)\rd \rho\\
=&\ \lf(u^*-\frac{D_1}{\underline h(t)}\rr)\int_{\underline h(t)/2-r}^{K_*}  J_*(\rho) \lf[2\lf(1-\frac{\rho+r}{\underline h(t)}\rr)-1\rr]\rd \rho\\
\geq &\lf(u^*-\frac{D_1}{\underline h(t)}\rr)\int_{\underline h(t)/2-r}^{K_*}  J_*(\rho) \lf[2\lf(1-\frac{K_*+\underline h(t)/2}{\underline h(t)}\rr)-1\rr]\rd \rho\\
\geq&\frac{-2K_* u^*}{\underline h(t)}.
\end{align*}

For $r\in \left[\frac 12 \underline h(t),\frac 12 \underline h(t)+K_*\right]$ and $t>0$, with $\underline h(t)\geq 2E_1\theta\gg1$, we have
\begin{align*}
\underline u(t,r)&=2\lf(u^*-\frac{D_1}{\underline h(t)}\rr)\lf(1-\frac{r}{\underline h(t)}\rr)\\
&=2\lf(u^*-\frac{D_1}{\underline h(t)}\rr)\lf(1-\frac{r}{\underline h(t)}\rr)\int_{-K_*}^{K_*}  J_*(\rho) \rd \rho-2\lf(u^*-\frac{D_1}{\underline h(t)}\rr)\int_{-K_*}^{K_*}  J_*(\rho) \frac{\rho}{\underline h(t)}\rd \rho\\
&=\int_{-K_*}^{K_*}  J_*(\rho) 2\lf(u^*-\frac{D_1}{\underline h(t)}\rr)\lf(1-\frac{\rho+r}{\underline h(t)}\rr)\rd \rho.
\end{align*}
Therefore,
\begin{align*}
&\int_{0}^{\underline h(t)}  J_*(r-\rho) \underline u(\rho)\rd \rho-\underline u(t,r)\\
=&\int_{-K_*}^{\underline h(t)/2-r}\!\!\! \!\! J_*(\rho) \underline u(\rho+r)\rd \rho+\!\!\!\int_{\underline h(t)/2-r}^{K_*} \!\!\! J_*(\rho) \underline u(\rho+r)\rd \rho-\!\!\!\int_{-K_*}^{K_*} \!\!\! J_*(\rho) 
2\lf(u^*-\frac{D_1}{\underline h(t)}\rr)\!\!\lf(1-\frac{r}{\underline h(t)}\rr)\rd \rho \\
=&\int_{-K_*}^{\underline h(t)/2-r}  J_*(\rho) \lf[\lf(u^*-\frac{D_1}{\underline h(t)}\rr)-2\lf(u^*-\frac{D_1}{\underline h(t)}\rr)\lf(1-\frac{\rho+r}{\underline h(t)}\rr)\rr]\rd \rho\\
=&\ \lf(u^*-\frac{D_1}{\underline h(t)}\rr)\int_{-K_*}^{\underline h(t)/2-r}  J_*(\rho) \lf[1-2\lf(1-\frac{\rho+r}{\underline h(t)}\rr)\rr]\rd \rho\\
\geq&\ \frac{-2K_* u^*}{\underline h(t)}.
\end{align*}

We have now proved that for  $r\in \left[\frac 12 \underline h(t)-K_*,\frac 12 \underline h(t)+K_*\right]$ and $t>0$, with $\underline h(t)\geq 2E_1\theta\gg1$,
\[
\int_{0}^{\underline h(t)}  J_*(r-\rho) \underline u(\rho)\rd \rho-\underline u(t,r)\geq \frac{-2K_* u^*}{\underline h(t)},
\]
which combined with \eqref{3.32} gives
\begin{align*}
&d \int_{0}^{\underline h(t)} \td J(r,\rho) \underline u(t,\rho)\rd \rho-d\underline u(t,r)+f(\underline u(t,r))\\
&\geq d\left[ \int_{0}^{\underline h(t)}  J_*(r-\rho) \underline u(\rho)\rd \rho-\underline u(t,r)\right]+f(\underline u(t,r))\\
& \ \ \ \ \ \ +d\left[\frac{N-1}r\int_{-K_*}^{\min\{h(t)-r, K_*\}} 
\rho
J_*(\rho) \underline u(t, \rho+r)\rd \rho -\frac{NK_*^2\|\underline u\|_{L^\infty}}{r^2}\right]\\
&\geq -\frac{3K_* u^*}{\underline h(t)}d+f(\underline u(t,r))+d\frac{N-1}r\int_{-K_*}^{\min\{h(t)-r, K_*\}} 
\rho
J_*(\rho) \underline u(t, \rho+r)\rd \rho.
\end{align*}

A simple computation gives, for  $r\in \left[\frac 12 \underline h(t)-K_*,\frac 12 \underline h(t)+K_*\right]$ and $t>0$, with $\underline h(t)\geq 2E_1\theta\gg1$,
\begin{align*}
d\frac{N-1}{r}\int_{-K_*}^{\min\{\underline h(t)-r,K_*\}} \rho
J_*(\rho) \underline u(t, \rho+r)\rd \rho\\
\geq d\frac{N-1}{r}(-K_*u^*)\geq -\frac{3d(N-1)K_*u^*}{\underline h(t)}.
\end{align*}
Moreover, since 
\begin{align*}
\min\{u^*,u^*-\underline u(t,r)\}= u^*-\underline u(t,r)\geq u^*-\underline u(t, \underline h(t)/2-K_*)=D_1/\underline h(t),
\end{align*}
by \eqref{3.36} we obtain
\[
f(\underline u(t,r))\geq D_*\min\{u^*,u^*-\underline u(t,r)\}\geq \frac{D_*D_1}{\underline h(t)}.
\]
Therefore, for  $r\in \left[\frac 12 \underline h(t)-K_*,\frac 12 \underline h(t)+K_*\right]$ and $t>0$, with $\underline h(t)\geq 2E_1\theta\gg1$,
\begin{align*}
&d \int_{0}^{\underline h(t)} \td J(r,\rho) \underline u(t,\rho)\rd \rho-d\underline u(t,r)+f(\underline u(t,r))\\
&\geq -\frac{3K_* u^*}{\underline h(t)}d+ \frac{D_*D_1}{\underline h(t)}-\frac{3d(N-1)K_*u^*}{\underline h(t)}\\
&= \frac{D_*D_1-3dNK_*u^*}{\underline h(t)}.
\end{align*}

From the definition of $\underline u$ we obtain, for $r\in [\underline h(t)/2-K_*, \underline h(t)/2)$,
\[
\underline u_t(t,r)=\frac{D_1\underline h'(t)}{\underline h^2(t)}=\frac{2D_1E_1}{\underline h^2(t)},
\]
and for  $r\in (\underline h(t)/2, \underline h(t)/2+K_*]$,
\[
\underline u_t(t,r)=2\frac{D_1\underline h'(t)}{\underline h^2(t)}\left(1-\frac r{\underline h(t)}\right)+2\left(u^*-\frac{D_1}{\underline h(t)}\right)\frac{r \underline h'(t)}{\underline h^2(t)}
\leq \frac{4D_1E_1}{\underline h^2(t)}+\frac{4u^*E_1}{\underline h(t)}.
\]

Thus it is easily seen that \eqref{3.35} holds for Case (ii) 
provided that $D_1\in (0, 2E_1\theta u^*)$ is chosen large enough so that
\[
D_*D_1\geq 4dNK_*u^*+4u^*E_1,
\]
which is possible since $\theta\gg1$.

\underline{Case (iii)} For $r\in [\underline h(t)/2+K_*,\underline h(t)-K_*]$, $t>0$, with $\underline h(t)\geq 2E_1\theta\gg1$,
\begin{align*}
&\int_{0}^{\underline h(t)}  J_*(r-\rho) \underline u(t, \rho)\rd \rho-\underline u(t,r)=\int_{-K_*}^{K_*}  J_*(\rho) \underline u(t,\rho+r)\rd \rho-\underline u(t,r)\\
&=\int_{-K_*}^{K_*}  J_*(\rho) 2\lf(u^*-\frac{D_1}{\underline h(t)}\rr)\lf(1-\frac{\rho+r}{\underline h(t)}\rr)\rd \rho-2\lf(u^*-\frac{D_1}{\underline h(t)}\rr)\lf(1-\frac{r}{\underline h(t)}\rr)\\
=&\int_{-K_*}^{K_*}  J_*(\rho) 2\lf(u^*-\frac{D_1}{\underline h(t)}\rr)\frac{\rho}{\underline h(t)}\rd \rho=0,
\end{align*}
and
\begin{align*}
&\int_{-K_*}^{\min\{\underline h(t)-r,K_*\}} 
{\rho}J_*(\rho) \underline u(t, \rho+r)\rd \rho=\int_{-K_*}^{K_*}  {\rho}J_*(\rho) 2\lf(u^*-\frac{D_1}{\underline h(t)}\rr)\lf(1-\frac{\rho+r}{\underline h(t)}\rr)\rd \rho\\
=&2\lf(u^*-\frac{D_1}{\underline h(t)}\rr)\lf(1-\frac{r}{\underline h(t)}\rr)\int_{-K_*}^{K_*}  {\rho}J_*(\rho) \rd \rho-\int_{-K_*}^{K_*}  {\rho}J_*(\rho) 2\lf(u^*-\frac{D_1}{\underline h(t)}\rr)\frac{\rho}{\underline h(t)}\rd \rho\\
=&-2\lf(u^*-\frac{D_1}{\underline h(t)}\rr)\int_{-K_*}^{K_*}  \frac{\rho^2}{\underline h(t)}J_*(\rho) \rd \rho\geq -\frac{2u^*K_*^2}{\underline h(t)}.
\end{align*}
Thus, by  \eqref{3.32},  for such $r, t$ and $\theta$,
\begin{align*}
&d \int_{0}^{\underline h(t)} \td J(r,\rho) \underline u(t,\rho)\rd \rho-d\underline u(t, r)+f(\underline u(t,r)\\
\geq &-\frac{d(N-1)}{r}\frac{2u^*K_*^2}{\underline h(t)}-\frac{dNK_*^2\|\underline u\|_{L^\infty}}{r^2}+f(\underline u(t,r))\\
\geq& -\frac{16 dNu^*K_*^2}{\underline h^2(t)}+f(\underline u(t,r)).
\end{align*}
Since for such $r, t$ and $\theta$,
\begin{align*}
& \min\{\underline u(t,r)), (u^*-\underline u(t,r))\}\geq \min\{\underline u(t, \underline h(t)-K_*)), (u^*-\underline u(t, \underline h(t)/2))\}\\
 =&\min\lf\{2\lf(u^*-\frac{D_1}{\underline h(t)}\rr)\frac{K_*}{\underline h(t)}, \frac{D_1}{\underline h(t)}\rr\}\geq \min\lf\{\frac{u^*K_*}{\underline h(t)}, \frac{D_1}{\underline h(t)}\rr\}=\frac{u^*K_*}{\underline h(t)}
\end{align*}
provided $D_1\geq u^*K_*$, by  \eqref{3.36} we obtain
\[
f(\underline u(t,r))\geq D_* \min\{\underline u(t,r)), (u^*-\underline u(t,r))\}\geq \frac{D_* u^*K_*}{\underline h(t)}.
\]
Therefore, for $r\in [\underline h(t)/2+K_*,\underline h(t)-K_*]$, $t>0$, with $\underline h(t)\geq 2E_1\theta\gg1$,
\begin{align*}
&d \int_{0}^{\underline h(t)} \td J(r,\rho) \underline u(t,\rho)\rd \rho-d\underline u(t, r)+f(\underline u(t,r)\\
\geq& -\frac{16 dNu^*K_*^2}{\underline h^2(t)}+\frac{D_* u^*K_*}{\underline h(t)}\geq \frac{D_* u^*K_*}{2\underline h(t)}.
\end{align*}

By simple calculation, we have, for $r\in [\underline h(t)/2+K_*,\underline h(t)-K_*]$, $t>0$, with $\underline h(t)\geq 2E_1\theta\gg1$,
\begin{align*}
\underline u_t(t,r)
=&\ 2\frac{D_1 \underline h'(t)}{\underline h^2(t)}\lf(1-\frac{r}{\underline h(t)}\rr)+2\lf(u^*-\frac{D_1}{\underline h(t)}\rr)\frac{r \underline h'(t)}{\underline h^2(t)}\\
\leq&\ \frac{4E_1D_1 }{\underline h^2(t)}+\frac{3E_1u^*}{\underline h(t)}\leq \frac{4E_1u^*}{\underline h(t)}.
\end{align*}
Thus \eqref{3.35} holds for such $r, t$ and $\theta$ if $E_1\in (0, D_*K_*/8]$.

\underline{Case (iv)} For $r\in [\underline h(t)-K_*,\underline h(t)]$, $t>0$, with $\underline h(t)\geq 2E_1\theta\gg1$,
\begin{align*}
&\int_{0}^{\underline h(t)}  J_*(r-\rho) \underline u(t,\rho)\rd \rho-\underline u(t,r)=\int_{-K_*}^{\underline h(t)-r}  J_*(\rho) \underline u(\rho+r)\rd \rho-\underline u(t,r)\\
&=\int_{-K_*}^{K_*}  J_*(\rho) 2\lf(u^*-\frac{D_1}{\underline h(t)}\rr)\lf(1-\frac{\rho+r}{\underline h(t)}\rr)\rd \rho-2\lf(u^*-\frac{D_1}{\underline h(t)}\rr)\lf(1-\frac{r}{\underline h(t)}\rr)\\
&\ \ \ -\int_{\underline h(t)-r}^{K_*}  J_*(\rho) 2\lf(u^*-\frac{D_1}{\underline h(t)}\rr)\lf(1-\frac{\rho+r}{\underline h(t)}\rr)\rd \rho\\
&=-2\lf(u^*-\frac{D_1}{\underline h(t)}\rr)\int_{\underline h(t)-r}^{K_*}  J_*(\rho) \lf(1-\frac{\rho+r}{\underline h(t)}\rr)\rd \rho\\
&\geq -u^*\int_{\underline h(t)-r}^{K_*}  J_*(\rho) \lf(1-\frac{\rho+r}{\underline h(t)}\rr)\rd \rho\geq 0,
\end{align*}
and 
\begin{align*}
&\int_{-K_*}^{\min\{\underline h(t)-r,K_*\}} 
{\rho}
J_*(\rho) \underline u(\rho+r)\rd \rho=\int_{-K_*}^{\underline h(t)-r}  {\rho}J_*(\rho) 2\lf(u^*-\frac{D_1}{\underline h(t)}\rr)\lf(1-\frac{\rho+r}{\underline h(t)}\rr)\rd \rho\\
=&\lf[\int_{-K_*}^{K_*} -\int_{\underline h(t)-r}^{K_*}\rr] {\rho}J_*(\rho) 2\lf(u^*-\frac{D_1}{\underline h(t)}\rr)\lf(1-\frac{\rho+r}{\underline h(t)}\rr)\rd \rho\\
\geq&\int_{-K_*}^{K_*}  {\rho}J_*(\rho) 2\lf(u^*-\frac{D_1}{\underline h(t)}\rr)\lf(1-\frac{\rho+r}{\underline h(t)}\rr)\rd \rho\\
=&\ 2\lf(u^*-\frac{D_1}{\underline h(t)}\rr)\lf(1-\frac{r}{\underline h(t)}\rr)\int_{-K_*}^{K_*}  {\rho}J_*(\rho) \rd \rho-\int_{-K_*}^{K_*}  {\rho}J_*(\rho) 2\lf(u^*-\frac{D_1}{\underline h(t)}\rr)\frac{\rho}{\underline h(t)}\rd \rho\\
=&-2\lf(u^*-\frac{D_1}{\underline h(t)}\rr)\int_{-K_*}^{K_*}  \frac{\rho^2}{\underline h(t)}J_*(\rho) \rd \rho\geq -\frac{2u^*K_*^2}{\underline h(t)}.
\end{align*}

Therefore, by  \eqref{3.32}, for such $r$, $t$ and $\theta$, we obtain
\begin{align*}
&d \int_{0}^{\underline h(t)} \td J(r,\rho) \underline u(t,\rho)\rd \rho-d\underline u(t,r)+f(\underline u(t,r)\\
\geq &-du^*\int_{\underline h(t)-r}^{K_*}  J_*(\rho) \lf(1-\frac{\rho+r}{\underline h(t)}\rr)\rd \rho-d\left[\frac{N-1}r\frac{2u^*K_*^2}{\underline h(t)}+\frac{NK_*^2\|\underline u\|_{L^\infty}}{r^2}\right]+f( \underline u(t,r))\\
\geq &-du^*\int_{\underline h(t)-r}^{K_*}  J_*(\rho) \lf(1-\frac{\rho+r}{\underline h(t)}\rr)\rd \rho-\frac{4dNK_*^2u^*}{\underline h^2(t)}+f( \underline u(t,r)).
\end{align*}
Moreover, for such $r$, $t$ and $\theta$,
\begin{align*}
1\gg \underline u(t,r)= 2\lf(u^*-\frac{D_1}{\underline h(t)}\rr)\lf(1-\frac{r}{\underline h(t)}\rr)\geq u^*\lf(1-\frac{r}{\underline h(t)}\rr),
\end{align*}
which implies, by  \eqref{3.36}
\[
f( \underline u(t,r))\geq D_*u^*\lf(1-\frac{r}{\underline h(t)}\rr).
\]

For  $r\in [\underline h(t)-K_*,\underline h(t)-K_*/2]$, $t>0$ and $\theta\gg1$, we have
\begin{align*}
&-du^*\int_{\underline h(t)-r}^{K_*}  J_*(\rho) \lf(1-\frac{\rho+r}{\underline h(t)}\rr)\rd \rho+D_*u^*\lf(1-\frac{r}{\underline h(t)}\rr)\\
\geq &D_*u^*\lf(1-\frac{r}{\underline h(t)}\rr)\geq \frac{D_*u^*K_*}{2\underline h(t)},
\end{align*}
and for $r\in [\underline h(t)-K_*/2,\underline h(t)]$, $t>0$ and $\theta\gg1$,
\begin{align*}
&-du^*\int_{\underline h(t)-r}^{K_*}  J_*(\rho) \lf(1-\frac{\rho+r}{\underline h(t)}\rr)\rd \rho+u^*\lf(1-\frac{r}{\underline h(t)}\rr)\\
\geq &-du^*\int_{\underline h(t)-r}^{K_*}  J_*(\rho) \lf(1-\frac{\rho+r}{\underline h(t)}\rr)\rd \rho\geq -du^*\int_{ K^*/2}^{K_*}  J_*(\rho) \lf(1-\frac{\rho+r}{\underline h(t)}\rr)\rd \rho\\
\geq& -du^*\int_{ 2K^*/3}^{K_*}  J_*(\rho) \lf(1-\frac{\rho+r}{\underline h(t)}\rr)\rd \rho\geq \frac{dK_*u^*}{6\underline h(t)} \int_{ 2K^*/3}^{K_*}  J_*(\rho) \rd \rho,
\end{align*}
so we always have, for $r\in [\underline h(t)-K_*,\underline h(t)]$, $t>0$, with $\underline h(t)\geq 2E_1\theta\gg1$,
\begin{align*}
&d \int_{0}^{\underline h(t)} \td J(r,\rho) \underline u(t,\rho)\rd \rho-d\underline u(t,r)+f(\underline u(t,r)\\
&\geq \frac{D_2}{\underline h(t)}-\frac{4dNK_*^2u^*}{\underline h^2(t)}\geq \frac{D_2}{2\underline h(t)},
\end{align*}
where
\begin{align*}
D_2:=\min\lf\{\frac{D_*u^*K_*}{2}, \frac{dK_*u^*}{6} \int_{ 2K^*/3}^{K_*}  J_*(\rho) \rd \rho\rr\}>0.
\end{align*}
Here we assume that $K_*>0$ is the minimal number such that the supporting set of $J_*$ is contained in $[-K_*, K_*]$.

For such $r, t$ and $\theta$, from the calculation in Case (iii) we have
\begin{align*}
\underline u_t(t,r)\leq \frac{4E_1u^*}{\underline h(t)},
\end{align*}
and so \eqref{3.35} holds if  
\begin{align*}
4E_1u^*\leq D_2/2.
\end{align*}
The proof is now complete.
\end{proof}

\begin{proof}[{\bf Proof of Lemma \ref{lem-lb}}]
  Let $(c_0,\phi_0)$ be the solution of the semi-wave problem in Proposition \ref{prop1.4}. 	Define
	\begin{align*}\begin{cases}
	 \underline h(t):=c_0 t+\delta(t),  & t\geq 0,\\
	\underline u(t,r):=(1-\epsilon(t)) \phi_0(r-\underline h(t)), & t\geq 0,\  0\leq r\leq \underline h(t),
	\end{cases}
	\end{align*}
	with 
	\begin{align*}
	\epsilon(t):=K_1(t+\theta)^{-1}, \ 	\ \  \delta(t):=c_0\theta -K_2[\ln (t+\theta)-\ln \theta]
	\end{align*}
	for some positive constants $\theta\geq \theta_1$, $K_1\geq E_2/u^*$ and $K_2>0$ to be  determined, where $\theta_1$ and $E_2$ are given by Lemma \ref{lemma3.12}.  We assume that $\theta\gg 1$.   Then it is clear that 
	\begin{align}\label{3.38}
c_0 (t+\theta)\geq 	\underline h(t)\geq c_0 (t+\theta)/2 \ \mbox{ for } \ t\geq 0.
	\end{align}

In the following we choose suitable $\theta$, $K_1,\ K_2$ and $t_0>0$ such that $(\underline u,\underline h)$ satisfies
	\begin{equation}\label{3.39}
	\begin{cases}
	\dd \underline u_t(t,r)\leq d \int_{0}^{\underline h(t)} \td J(r,\rho) \underline u(t,\rho)\rd \rho-d\underline u(t,r)+f(\underline u(t,r)), & 
	t>0,\; r\in  (\wtd E\underline h(t),\underline h(t)),\\[2mm]
	\dd \underline h'(t)\leq  \frac{\mu}{\underline h^{N-1}(t)}\dd\int_{0}^{\underline h(t)} r^{N-1}\underline u(t,r)\int_{\underline h(t)}^{+\yy} \td J(r,\rho)\rd \rho\rd r, & t>0,\\
	\underline u(t,r)\leq  u(t+t_0,r), \; \underline u(t,\underline h(t))=0, & t>0,\; r\in [0,\wtd E\underline h(t)],\\
	\underline u(0,r)\leq  u(t_0,r),\ \underline h(t)\leq  h(t_0), & r\in [0,\underline h(0)],
	\end{cases}
	\end{equation}
		where $\wtd E:=E_1/c_0$ with $E_1$ given by Lemma \ref{lemma3.12}. 

	If \eqref{3.39} is proved, then we can apply Lemma \ref{lemma3.4a} to conclude that
	\[
	h(t_0+t)\geq \underline h(t)=c_0t-K_2[\ln (t+\theta)-\ln \theta] \mbox{ for all } t>0,
	\]
	which implies \eqref{3.36a}. Therefore to complete the proof, it suffices to show \eqref{3.39}, which will be accomplished in three steps below.

	{\bf Step 1.} 	We prove  the second  inequality of \eqref{3.39}.  
	
	Applying \eqref{3.33}, we deduce for $N\geq 3$,
	\begin{align*}
	&\frac{\mu}{\underline h^{N-1}(t)}\dd\int_{0}^{\underline h(t)} r^{N-1}\underline u(t,r)\int_{\underline h(t)}^{+\yy} \td J(r,\rho)\rd \rho\rd r\\
	=&\ \frac{\mu}{\underline h^{N-1}(t)}\dd\int_{0}^{\underline h(t)} (1-\epsilon(t)) r^{N-1}\phi_0(r-\underline h)\int_{\underline h(t)}^{+\yy} \td J(r,\rho)\rd \rho\rd r\\
	\geq &\ \frac{\mu (1-\epsilon(t))}{\underline h^{N-1}(t)}\dd\int_{0}^{\underline h(t)}\!\!\!\int_{\underline h(t)}^{+\yy}  (r\rho)^{(N-1)/2}\lf[\frac{(\rho+r)^2-K_*^2}{4r\rho }\rr]^{(N-3)/2}
	\!\!\!\phi_0(r-\underline h(t))  J_*(r-\rho)\rd \rho\rd r\\
	=&\ \mu (1-\epsilon(t))\left[\dd\int_{-\underline h(t)}^{0} \int_{0}^{+\yy}  \phi_0(r) J_*(r-\rho)\rd \rho\rd r+ A(N)\right]
	\end{align*}
	with 
	\begin{align*}
	A(N)\!:=\!\!\int_{-\underline h(t)}^{0}\!\int_{0}^{+\yy}\!\! \lf(\left[\frac{(r\!+\!\underline h(t))(\rho\!+\!\underline h(t))}{\underline h^2(t)}\right]^{\frac{N-1}2}\!\lf[\frac{(\rho+r+2\underline h(t))^2-K_*^2}{4(r+\underline h(t))(\rho+\underline h(t)) }\rr]^{\frac{N-3}2}\!\!\!\!-\!1\!\rr)\phi_0(r) J_*(r\!-\!\rho)\rd \rho\rd r.
	\end{align*}
	
	For $N=2$, we similarly deduce
	\begin{align*}
	&\frac{\mu}{\underline h(t)}\dd\int_{0}^{\underline h(t)} r^{N-1}\underline u(t,r)\int_{\underline h(t)}^{+\yy} \td J(r,\rho)\rd \rho\rd r\\
	\geq&\ \mu(1-\epsilon(t))\left[\int_{-\underline h(t)}^{0} \int_{0}^{+\yy} \phi_0(r) J_*(r-\rho)\rd \rho\rd r+ A(2)\right]
	\end{align*}
	with
	\begin{align*}
	A(2):=\dd\int_{-\underline h(t)}^{0} \int_{0}^{+\yy}  \lf(\frac{(r+\underline h(t))^{1/2}(\rho+\underline h(t))^{1/2}}{\underline h(t)}-1\rr)\phi_0(r) J_*(r-\rho)\rd \rho\rd r.
	\end{align*}
	
	{\bf Claim 1}.  There are some constants $C_{i,j}\geq 0$ with $C_{1,0}= C_{0,1}>0$ such that   for $N\geq 2$ and  $x, y\in \R$ close to 0,
	\begin{align}\label{3.43}
	(1+x)^{(N-1)/2}(1+y)^{(N-1)/2}\geq 1+\sum_{1\leq  i+j\leq N}C_{i,j}x^{i}y^{j}.
	\end{align}
	Clearly, \eqref{3.43} holds for  odd integer $N\geq 2$. For even integer $N\geq 2$ and  $x, y\in \R$ close to 0,
	\begin{align*}
	&(1+x)^{(N-1)/2}(1+y)^{(N-1)/2}-1\\
	=&(1+x)^{(N-2)/2}(1+y)^{(N-2)/2}(1+x)^{1/2}(1+y)^{1/2}-1\\
	\geq &(1+x)^{(N-2)/2}(1+y)^{(N-2)/2}(1+\frac x 2-\frac{x^2}4)(1+\frac y 2-\frac{y^2}4)-1\\
	=&\sum_{1\leq  i+j\leq N}C_{i,j}x^{i}y^{j} \mbox{ with } C_{0,1}=C_{1,0}>0,
	\end{align*}
	since for $z\in \R$ close to 0,
	\begin{align*}
	(1+z)^{1/2}=1+\frac{z}{2}-\frac{(1+\xi)^{-3/2}z^2}{8}\geq 1+\frac{z}{2}-\frac{z^2}4
	\end{align*}
	with $\xi$ lying between 0 and $z$. 	The claim is proved.
	
	By the mean value theorem,  we obtain for $N\geq 3$ and small $x\geq 0$,
	\begin{align}\label{3.44}
	(1-x)^{(N-3)/2}\geq 1-\frac{N-3}{2}(1-\xi)^{(N-5)/2}x \geq 1-Nx
	\end{align}
	for some  $\xi\in [0,x]$. Therefore, for $r\in [-K_*, 0], \ t\geq 0,\ \rho\geq 0$ and $\theta\gg 1$, we have, when $N\geq 3$,
	\begin{align*}
	\lf[\frac{(\rho+r+2\underline h(t))^2-K_*^2}{4(r+\underline h(t))(\rho+\underline h(t)) }\rr]^{(N-3)/2}&\geq \lf[1-\frac{K_*^2}{4(r+\underline h(t))(\rho+\underline h(t)) }\rr]^{(N-3)/2}\\
	&\geq 1-\frac{NK_*^2}{4(r+\underline h(t))(\rho+\underline h(t)) },
	\end{align*}
	and 
	\begin{align*}
	\left[\frac{(r\!+\!\underline h(t))(\rho\!+\!\underline h(t))}{\underline h^2(t)}\right]^{\frac{N-1}2}&=\left[\left(1+\frac r{\underline h(t)}\right)\left(1+\frac\rho{\underline h(t)}\right)\right]^{\frac{N-1}2}\\
	&\geq 1+ \sum_{1\leq i+j\leq N}C_{i,j}\frac{r^i\rho^j}{\underline h^{i+j}(t)}.
	\end{align*}
	
	Since ${\rm spt}(J_*)\subset [-K_*,K_*]$,  we deduce for $N\geq 3$ and $\theta\gg 1$,
	\begin{align*}
		A(N)\geq&\dd\int_{-K_*}^{0}\!\int_{0}^{K_*} \!\!  \lf[\lf(1+\!\!\!\sum_{1\leq  i+j\leq N}\!\!C_{i,j}\frac{r^i\rho^j}{\underline h^{i+j}(t)}\rr)\!\!\lf(1-\frac{NK_*^2}{4(r+\underline h(t))(\rho\!+\!\underline h(t)) }\rr)-1\rr]\phi_0(r) J_*(r\!-\!\rho)\rd \rho\rd r\\
	=&\dd\sum_{1\leq  i+j\leq N}C_{i,j}\int_{-K_*}^{0}\int_{0}^{K_*} \frac{ r^{i}\rho^{j}}{\underline h^{i+j}(t)}\phi_0(r) J_*(r-\rho)\rd \rho\rd r\\
	&-\int_{-K_*}^{0}\int_{0}^{K_*} \lf(1+\sum_{1\leq  i+j\leq N}C_{i,j}\frac{r^i\rho^j}{\underline h^{i+j}(t)}\rr)\frac{NK_*^2}{4(r+\underline h(t))(\rho+\underline h(t)) }\phi_0(r) J_*(r-\rho)\rd \rho\rd r\\
	\geq &\ C_{0,1}\dd\int_{-K_*}^{0}\int_{0}^{K_*}  \frac{(r+\rho)  }{\underline h(t)}\phi_0(r) J_*(r-\rho)\rd \rho\rd r -\sum_{2\leq  i+j\leq N}|C_{i,j}|\frac{K_*^{i+j+2}\|J_*\|_{L^\yy}u^*}{\underline h^{i+j}(t)}-\frac{NK_*^4\|J_*\|_{L^\yy}u^*}{2\underline h^2(t) }\\
	\geq &\ C_{0,1}\dd\int_{-K_*}^{0}\int_{0}^{K_*}  \frac{(r+\rho)  }{\underline h(t)}\phi_0(r) J_*(r-\rho)\rd \rho\rd r -\frac{\big(\sum_{2\leq  i+j\leq N}C_{i,j}K_*^{i+j+2}+NK_*^4\big)\|J_*\|_{L^\yy}u^*}{\underline h^{2}(t)}\\
	=:&\ \frac{\td B}{\underline h(t)}-\frac{D(N)}{\underline h^{2}(t)};
	\end{align*}
	and when $N=2$, similarly
	\begin{align*}
	A(2)\geq  &\dd\sum_{1\leq  i+j\leq 2}C_{i,j}\int_{-K_*}^{0}\int_{0}^{K_*} \frac{ r^{i}\rho^{j}}{\underline h^{i+j}(t)}\phi_0(r) J_*(r-\rho)\rd \rho\rd r\\
	\geq &C_{0,1}\dd\int_{-K_*}^{0}\int_{0}^{K_*}  \frac{(r+\rho)  }{\underline h(t)}\phi_0(r) J_*(r-\rho)\rd \rho\rd r -\sum_{i+j=2}C_{i,j}\frac{K_*^{i+j+2}\|J_*\|_{L^\yy}u^*}{\underline h^{2}(t)}\\
	=:&\ \frac{\tilde B}{\underline h(t)} -\frac{D(2)}{\underline h^{2}(t)}.
	\end{align*}
	
	From the proof of Lemma \ref{lemma3.10} we see  $\tilde B<0$. Since $\underline h(t)\geq c_0\theta/2\gg 1$,
	\begin{align*}
	\mu\int_{-\underline h(t)}^0\int_0^\infty \phi_0(r)J_*(r-\rho)d\rho dr&=\mu\int_{-K_*}^0\int_0^\infty \phi_0(r)J_*(r-\rho)d\rho dr\\
	&=\mu\int_{-\infty}^0\int_0^\infty \phi_0(r)J_*(r-\rho)d\rho dr=c_0.
	\end{align*}
	
	Putting all these together, we obtain, for $N\geq 2$ and $\theta\gg1$,
	\begin{align*}
	&\frac{\mu}{\underline h^{N-1}(t)}\dd\int_{0}^{\underline h(t)} r^{N-1}\underline u(t,r)\int_{\underline h(t)}^{+\yy} \td J(r,\rho)\rd \rho\rd r\\
	\geq &\  (1-\epsilon(t))\big[c_0+\mu A(N)\big]\geq  (1-\epsilon(t))\Big[c_0+\frac{\mu\td B}{\underline h(t)}-\frac{\mu D(N)}{\underline h^{2}(t)}\Big]
	\\
		\geq &\ (1-2\epsilon(t))c_0=c_0-2c_0K_1(t+\theta)^{-1}\\
	\geq &\ c_0-K_2(t+\theta)^{-1}=\underline h'(t) \  \mbox{ provided that  $2c_0K_1\leq K_2$.}
	\end{align*}
This completes Step 1.
	
	{\bf Step 2.} We prove  the first  inequality of \eqref{3.39}, namely,  for $t>0$ and $r\in (\wtd E \underline h(t),\underline h(t))$, 
	\begin{align}\label{4.28}
	\underline u_t(t,r)\geq &d \int_{0}^{\underline h(t)} \td J(r,\rho) \underline u(t,\rho)\rd \rho - d\underline u(t,r)+f(\underline u(t,r)).
	\end{align}

	By the definition of $\underline u$, we have
	\begin{align*}
	\underline u_t(t,r)=&-(1-\epsilon(t))[c_0+\delta'(t)]\phi_0'(r-\underline h(t))-\epsilon'(t)\phi_0(r-\underline h(t))\\
	=& -(1-\epsilon(t))c_0\phi'_0(r-\underline h(t))-(1-\epsilon(t)) \delta'(t)\phi_0'(r-\underline h(t))-\epsilon'(t)\phi_0(r-\underline h(t)).
	\end{align*}
	For $t>0$, $r\in (\wtd E \underline h(t),\underline h(t))$ and $\theta\gg1$, using Lemma \ref{lemma3.11} we obtain
	\begin{align*}
	& -(1-\epsilon(t))c_0\phi'(r-\underline h(t))\\ 
	= & (1-\epsilon(t)) \lf[d \int_{-\yy}^{\underline h(t)}  J_*(r-\rho)  \phi_0(\rho-\underline h(t))\rd \rho-d\phi_0(r-\underline h(t))+f(\phi_0(r-\underline h(t)))\rr]\\ 
	\leq & (1-\epsilon(t)) \bigg[d \int_{0}^{\underline h(t)}  \td J(r,\rho)  \phi_0(\rho-\underline h(t))\rd \rho-d\phi_0(r-\underline h(t))+f(\phi_0(r-\underline h(t)))\bigg]\\
	&- (1-\epsilon(t))d\left[ \frac{N-1}r\int_{-K_*}^{\min\{h(t)-r, K_*\}} \rho
J_*(\rho) \phi_0(\rho+r-\underline h(t))\rd \rho -\frac{NK_*^2\|\phi_0\|_{L^\infty}}{r^2}\right]\\
	\leq &\  d \int_{0}^{\underline h(t)}  \td J(r,\rho)  \underline u(t,\rho)\rd \rho-d\underline u(t,r)+(1-\epsilon(t))f(\phi_0(r-\underline h(t)))+\frac{2dNK_*u^*}{\wtd E\underline h(t)} .
	\end{align*}
Hence, for such $r,\ t$ and $\theta$,
\begin{align*}
	\underline u_t(t,r)\leq d \int_{0}^{\underline h(t)}  \td J(r,\rho)  \underline u(t,\rho)\rd \rho-d \underline u(t,r)+f(\underline u(t,r))+B(r),
\end{align*}
	with
	\begin{align*}
	B(r)=B(t,r)=:&(1-\epsilon(t))f(\phi_0(r-\underline h(t)))-f(\underline u(t,r))\\
	&-(1-\epsilon(t))\delta'(t)\phi'_0(r-\underline h(t))-\epsilon'(t)\phi_0(r-\underline h(t))+\frac{2dNK_*u^*}{\wtd E\underline h(t)}.
	\end{align*}
	
	To complete the proof, it remains to verify $B(r)\leq 0$, and we do so for   $r\in [\underline h(t)-L,\underline h(t)]$ and $r\in [\td E\underline h(t), \underline h(t)-L]$,  separately, where $L$ is a positive constant to be determined. 
	
For	$r\in [\td E\underline h(t), \underline h(t)-L]$, by  similar arguments to those leading to \eqref{3.33a}, we know that  for large  $L>0$ there exists $\wtd C=\wtd C(L)>0$   such that 
\begin{align*}
(1-\epsilon(t))f(\phi_0(r-\underline h(t)))-f(\underline u(t,r))\leq -\wtd C \epsilon(t) \mbox{ for all } t>0. 
\end{align*}
 Hence, for such $L$, due to $\delta'(t)\phi'_0(r-\underline h(t))\geq 0$ and \eqref{3.38}, we obtain, for $r\in [\td E\underline h(t), \underline h(t)-L]$ and $t>0$,
\begin{align*}
B(r)\leq& -\wtd C\epsilon(t)-\epsilon'(t)\phi_0(r-\underline h(t))+\frac{2dNK_*u^*}{\wtd E\underline h(t)}\\
\leq&- \wtd C\frac{K_1 }{t+\theta}+\frac{K_1u^*}{(t+\theta)^2}+\frac{2dNK_*u^*}{\wtd Ec_0(t+\theta)} \\
=&\ \frac{1 }{t+\theta}\lf(-K_1\wtd C +\frac{K_1u^*}{t+\theta}+\frac{2dNK_*u^*}{\wtd Ec_0} \rr)\leq 0
\end{align*}
if $\theta$ is large and 
\begin{align*}
K_1\wtd C\geq \frac{3dNK_*u^*}{\wtd Ec_0}.
\end{align*}
	
	Denote $C_L=\inf_{\xi\in [-L,0]} [-\phi'_0(\xi)]>0$.	For $r\in [\underline h(t)-L,\underline h(t)]$, $t>0$ and $\theta\gg 1$,  using $1-\epsilon(t)\geq 1/2$, \eqref{3.38},
	and {\bf (f)},
	we deduce
	\begin{align*}
	B(r)\leq &-(1-\epsilon(t))\delta'(t)\phi'_0(r-\underline h(t))-\epsilon'(t)\phi_0(r-\underline h(t))+\frac{2dNK_*u^*}{\wtd E\underline h(t)}\\
	\leq& -\frac{1}{2}\frac{K_2C_L}{t+\theta}+\frac{K_1u^*}{(t+\theta)^2}+\frac{2dNK_*u^*}{\wtd Ec_0(t+\theta)} \\
	=&\ \frac{1}{t+\theta}\lf[-\frac{K_2C_L}{2}+\frac{K_1u^*}{t+\theta}+\frac{2dNK_*u^*}{\wtd Ec_0}\rr]\leq 0
	\end{align*}
for large $\theta$ and 
	\begin{align*}
	K_2C_L\geq \frac{5dNK_*u^*}{\wtd Ec_0}.
	\end{align*}
This concludes Step 2.
	
		{\bf Step 3}. We verify the last two inequalities of \eqref{3.39}.
	
	We now fixed $K_1$, $K_2$ and $\theta$ such that the conclusions in Step 1 and Step 2 hold.  Then from Lemma \ref{lemma3.12}, $K_1\geq E_2/u^*$ and 
	$\wtd E\underline h(t)\leq \wtd Ec_0(t+\theta)=E_1(t+\theta)$, we see that there is $t_0>0$ such that 
	\begin{align*}
		\underline u(t,r)\leq \lf(1-\frac{K_1}{t+\theta}\rr) u^*\leq u^*-\frac{E_2}{t+\theta}\leq  u(t+t_0,r)\ \mbox{ for } \  t>0,\; r\in [0,\wtd E\underline h(t)].
	\end{align*}
	Moreover, since spreading happens,  by enlarging  $t_0$ if necessary we may assume $\underline h(0)=c_0\theta \leq  h(t_0)$ and 
	\begin{align*}
		\underline u(0,r)\leq (1-\epsilon(0))u^*\leq  u(t_0,r)\ \mbox{ for } r\in [0,\underline h(0)].
	\end{align*}
	Evidently $\underline u(t,\underline h(t))\equiv 0$.
	Therefore
	the last two inequalities of \eqref{3.39} are satisfied for the above chosen $t_0$. 
	 The proof of the lemma is now complete. 
\end{proof}


\section{Rate of accelerated spreading}

In this section, we consider the case that {\bf (J1)} is not satisfied by the kernel function $J$ and hence accelerated spreading can happen. We will focus on the class of $J$ satisfying
$J(r)\approx r^{-\beta}$ near $\infty$, namely, 
\begin{align}\label{5.1}
C_1r^{-\beta}\leq J(r)\leq C_2r^{-\beta}  
\end{align}
for some positive constants $C_1$, $C_2$, and all $r\gg 1$. For such $J$, clearly {\bf (J1)} holds if and only if $\beta>N+1$, and {\bf (J)} holds if and only if $\beta>N$. So we will consider the case $\beta\in (N, N+1]$, which is the exact range that accelerated spreading can happen for \eqref{1.4} with such a kernel function $J$. We will determine the rate of $h(t)$ as $t\to\infty$ and prove Theorem \ref{th1.8}.

\subsection{Some further estimates on $J_*$ and $\tilde J$.}
\begin{lemma}\label{lemma5.1a} Assume \eqref{5.1} holds with $\beta\in (N,N+1]$. Then for  $h\gg1$,
	\begin{equation}\label{5.2}
	\int_{0}^{h}\int_{h}^{\yy}J_*(r-\rho) \rd \rho\rd r\approx
\begin{cases}
\ h^{N+1-\beta} & {\rm when}\ \beta\in (N,N+1),\\[2mm]
\ \ln h& {\rm when}\ \beta=N+1.
\end{cases}
\end{equation}
		and
	\begin{equation}\label{5.3b}
	\int_{0}^{h}\int_{h}^{\yy}\td J(r,\rho) \rd \rho\rd r\approx
\begin{cases}
\ h^{N+1-\beta}& {\rm when}\ \beta\in (N,N+1),\\[2mm]
\ \ln h& {\rm when}\ \beta=N+1.
\end{cases}
	\end{equation}
	\end{lemma}
\begin{proof}  To prove \eqref{5.2}, we first calculate
\begin{align*}
\int_{0}^{h}\int_{h}^{\yy}J_*(r-\rho) \rd \rho \rd r=\int_{-h}^{0}\int_{-r}^{\yy}J_*(\rho) \rd \rho \rd r=\int_{0}^{h} J_*(\rho) \rho\rd \rho+h\int_{h}^{\yy} J_*(\rho) \rd \rho.
\end{align*}
Moreover, by \eqref{1.10}, for  $\rho\geq 1$ and $\beta\in (N, N+1]$,
\begin{equation}\label{7.3}\begin{aligned}
J_*(\rho) & =\omega_{N-1}\int_{\rho}^{\yy}   (r^2-\rho^2)^{(N-3)/2} rJ(r)\rd r\approx \omega_{N-1}\int_{\rho}^{\yy}   (r^2-\rho^2)^{(N-3)/2} r^{1-\beta}\rd r\\
&=\omega_{N-1}\rho^{N-\beta-1}\int_{1}^{\yy} (\xi^2-1)^{(N-3)/2} \xi^{1-\beta}\rd \xi \approx \rho^{N-\beta-1}.
\end{aligned}
\end{equation}
Hence,
for $\beta\in (N,N+1)$,
\begin{align*}
\int_{0}^{h}\int_{h}^{\yy}J_*(r-\rho) \rd \rho \rd r = \int_{0}^{1} J_*(\rho) \rho\rd \rho+\int_{1}^{h} J_*(\rho) \rho\rd \rho+h\int_{h}^{\yy} J_*(\rho) \rd \rho\approx  h^{N-\beta+1},
\end{align*}
and for $\beta=N+1$,
\begin{align*}
\int_{0}^{h}\int_{h}^{\yy}J_*(r-\rho) \rd \rho \rd r= \int_{0}^{1} J_*(\rho) \rho\rd \rho+\int_{1}^{h} J_*(\rho) \rho\rd \rho+h\int_{h}^{\yy} J_*(\rho) \rd \rho\approx 
\ln h.
\end{align*}
This proves \eqref{5.2}.

We next prove \eqref{5.3b}, which is more involved.  We write
\begin{align*}
&\int_{0}^{h}\int_{h}^{\yy} \td J(r,\rho) \rd \rho \rd r\\
=&\int_{h/2}^{h}\int_{h}^{2h} \td J(r,\rho) \rd \rho \rd r+\int_{0}^{h/2}\int_{h}^{2h} \td J(r,\rho) \rd \rho \rd r+\int_{0}^{h}\int_{2h}^{\yy} \td J(r,\rho) \rd \rho \rd r\\
=:&\ I+II+III.
\end{align*}

{\bf Step 1}. Upper bound for $I$.

{\bf Case 1}. $N\geq 3$.

From \eqref{2.3a} and \eqref{5.1}, for  $|\rho-r|\geq 1$,
\begin{align*}
\td J(r,\rho)&= \omega_{N-1}\int_{|\rho-r|}^{\rho+r}   (\rho/r)^{(N-1)/2}\lf[\frac{(\rho+r)^2-\eta^2}{4r\rho }\rr]^{(N-3)/2}[\eta^2-(\rho-r)^2]^{(N-3)/2}\eta J(\eta) \rd \eta\\
&\leq \omega_{N-1}\int_{|\rho-r|}^{\rho+r}   (\rho/r)^{(N-1)/2}[\eta^2-(\rho-r)^2]^{(N-3)/2}\eta J(\eta) \rd \eta\\
&\leq(\rho/r)^{(N-1)/2}C_2\omega_{N-1}\int_{|\rho-r|}^{\rho+r}   [\eta^2-(\rho-r)^2]^{(N-3)/2}\eta^{1-\beta} \rd \eta\\
&\leq(\rho/r)^{(N-1)/2}C_2\omega_{N-1}\int_{|\rho-r|}^{\rho+r}   \eta^{N-\beta-2} \rd \eta\\
&\leq C\left(\frac{\rho}r\right)^{(N-1)/2}|\rho-r|^{N-\beta-1}
\end{align*}
for some $C>0$ independent of $r$ and $\rho$.
Therefore,
\begin{align*}
I=&\int_{h/2}^{h-1}\int_{h}^{2h} \td J(r,\rho) \rd \rho \rd r+\int_{h-1}^{h}\int_{h}^{2h} \td J(r,\rho) \rd \rho \rd r\\
\leq &\int_{h/2}^{h-1}\int_{h}^{2h} \td J(r,\rho) \rd \rho \rd r+1\leq C\int_{h/2}^{h-1}\int_{h}^{2h}(\rho/r)^{(N-1)/2}(\rho-r)^{N-\beta-1} \rd \rho \rd r+1\\
\leq&\ 4^{(N-1)/2}C \int_{h/2}^{h-1}\int_{h}^{2h}(\rho-r)^{N-\beta-1} \rd \rho \rd r+1\\
=&\ \frac{4^{(N-1)/2}C}{\beta-N} \int_{h/2}^{h-1}[(h-r)^{N-\beta}-(2h-r)^{N-\beta}]  \rd r+1\\
=&\ \frac{4^{(N-1)/2}C}{\beta-N} \int_{1}^{h/2}[\xi^{N-\beta}-(h+\xi)^{N-\beta}] \rd  \xi+1\\
\leq  &\ \td C  h^{N-\beta+1} \ \mbox{ for $h\gg 1$ and some $\tilde C>0$ independent of $h$.}
\end{align*}

{\bf Case 2}. $N= 2$.

By \eqref{2.3a} and \eqref{5.1}, for  $|\rho-r|\geq 1$, $\rho\geq h$ and $r\geq h/2$ with $h\gg1$, we have
\begin{align*}
\td J(r,\rho)=&\ \omega_{1}\int_{|\rho-r|}^{\rho+r}   (\rho/r)^{1/2}\lf[\frac{(\rho+r)^2-\eta^2}{4r\rho }\rr]^{-1/2}[\eta^2-(\rho-r)^2]^{-1/2}\eta J(\eta) \rd \eta\\
\leq&\  C_2\omega_{0}\int_{|\rho-r|}^{\rho+r}   (\rho/r)^{1/2}\lf[\frac{(\rho+r)^2-\eta^2}{4r\rho }\rr]^{-1/2}[\eta^2-(\rho-r)^2]^{-1/2}\eta^{1-\beta} \rd \eta\\
=&\ C_2\omega_{1}\lf(\int_{|\rho-r|}^{\theta_1|\rho-r|}+\int_{\theta_1|\rho-r|}^{\theta_2(\rho+r)} +\int_{\theta_2(\rho+r)}^{\rho+r}\rr ) (\rho/r)^{1/2}\lf[\frac{(\rho+r)^2-\eta^2}{4r\rho }\rr]^{-1/2}[\eta^2-(\rho-r)^2]^{-1/2}\eta^{1-\beta} \rd \eta
\end{align*}
where the constants $\theta_2<1<\theta_1$ are chosen to be close to 1.

Clearly, for such $\rho$ and $r$,
\begin{align*}
&\int_{|\rho-r|}^{\theta_1|\rho-r|} \lf[\frac{(\rho+r)^2-\eta^2}{4r\rho }\rr]^{-1/2}[\eta^2-(\rho-r)^2]^{-1/2}\eta^{1-\beta} \rd \eta\\
\leq&\lf[\frac{(\rho+r)^2-\theta_1^2(\rho-r)^2}{4r\rho }\rr]^{-1/2}|\rho-r|^{1-\beta}\int_{|\rho-r|}^{\theta_1|\rho-r|} [\eta^2-(\rho-r)^2]^{-1/2} \rd \eta\\
=&\lf[\frac{(\rho+r)^2-\theta_1^2(\rho-r)^2}{4r\rho }\rr]^{-1/2}|\rho-r|^{1-\beta}\int_{1}^{\theta_1} [\xi^2-1]^{-1/2} \rd \xi,
\end{align*}
and
\begin{align*}
&\int_{\theta_1|\rho-r|}^{\theta_2(\rho+r)} \lf[\frac{(\rho+r)^2-\eta^2}{4r\rho }\rr]^{-1/2}[\eta^2-(\rho-r)^2]^{-1/2}\eta^{1-\beta} \rd \eta\\
\leq&\ \lf[(1-\theta_2^2)\frac{(\rho+r)^2}{4r\rho }\rr]^{-1/2}\int_{\theta_1|\rho-r|}^{\theta_2(\rho+r)} [(1-1/\theta_1^2)\eta^2]^{-1/2}\eta^{1-\beta} \rd \eta
\\
=&\frac{(1-1/\theta_1^2)^{-1/2}}{\beta-1}\lf[(1-\theta_2^2)\frac{(\rho+r)^2}{4r\rho }\rr]^{-1/2}[\theta_1^{1-\beta}|\rho-r|^{1-\beta}-\theta_2^{1-\beta}(\rho+r)^{1-\beta}]\\
\leq&\ \tilde C_1 \lf[\frac{(\rho+r)^2}{4r\rho }\rr]^{-1/2}|\rho-r|^{1-\beta} \mbox{ for some $\td C_1>0$ independent of $r$ and $\rho$},
\end{align*}
and
\begin{align*}
&\int_{\theta_2(\rho+r)}^{\rho+r} \lf[\frac{(\rho+r)^2-\eta^2}{4r\rho }\rr]^{-1/2}[\eta^2-(\rho-r)^2]^{-1/2}\eta^{1-\beta} \rd \eta\\
\leq&\int_{\theta_2(\rho+r)}^{\rho+r} \lf[\frac{(\rho+r)^2-\eta^2}{4r\rho }\rr]^{-1/2}[(1-1/\theta_1^2)\eta^2]^{-1/2}\eta^{1-\beta} \rd \eta\\
\leq&\ (1-1/\theta_1^2)^{-1/2}\theta_2^{-\beta}(\rho+r)^{-\beta}\int_{\theta_2(\rho+r)}^{\rho+r} \lf[\frac{(\rho+r)^2-\eta^2}{4r\rho }\rr]^{-1/2} \rd \eta\\
=&\ (1-1/\theta_1^2)^{-1/2}\theta_2^{-\beta}(\rho+r)^{-\beta}(4r\rho)^{1/2}\int_{\theta_2}^{1}[1-\xi^2]^{-1/2} \rd \xi\\
=&\ \td C_2(\rho+r)^{-\beta}(4r\rho)^{1/2} \mbox{ for some $\td C_2>0$ independent of $r$ and $\rho$}.
\end{align*}

Moreover, for $(r,\rho)\in [h/2,h-1]\times [h,2h]$ with $h\gg 1$, we have $\rho/r\leq 4$ and
\begin{align*}
\frac{\rho-r}{\rho+r}=1-\frac{2}{\rho/r+1}\leq \frac{3}{5}.
\end{align*}
Therefore
\begin{align*}
&\lf[\frac{(\rho+r)^2-\theta_1^2(\rho-r)^2}{4l\rho }\rr]^{-1/2}\leq \lf[\frac{(1-9\theta_1^2/25)(\rho+r)^2}{4r\rho }\rr]^{-1/2}\\
\leq& \lf[\frac{(1-9\theta_1^2/25)\times 9h^2/4}{8h^2 }\rr]^{-1/2}=\lf[\frac{9(1-9\theta_1^2/25)}{32 }\rr]^{-1/2},
\end{align*}
and
\begin{align*}
\lf[\frac{(\rho+r)^2}{4r\rho }\rr]^{-1/2}\leq \lf[\frac{9}{32 }\rr]^{-1/2}.
\end{align*}
Thus our earlier inequalities yield,  for $(r,\rho)\in [h/2,h-1]\times [h,2h]$ with $h\gg 1$,
\begin{align*}
\td J(r,\rho)\leq A_1(\rho-r)^{1-\beta}+A_2h^{1-\beta}, 
\end{align*}
where $A_1$ and $A_2$ are positive constants independent of $r$ and $\rho$.

We may now use similar calculation as in Case 1 to obtain
\begin{align*}
I=&\int_{h/2}^{h-1}\int_{h}^{2h} \td J(r,\rho) \rd \rho \rd r+\int_{h-1}^{h}\int_{h}^{2h} \td J(r,\rho) \rd \rho \rd r\\
\leq &\int_{h/2}^{h-1}\int_{h}^{2h} \td J(r,\rho) \rd \rho \rd r+1\leq \int_{h/2}^{h-1}\int_{h}^{2h}[A_1(\rho-r)^{1-\beta}+A_2h^{1-\beta}] \rd \rho \rd r+1\\
\leq&\ \td C_3 h^{3-\beta} =\td C_3 h^{N-\beta+1}  \mbox{ for  $h\gg 1$ and some $\td C_3>0$ independent of $h$}.
\end{align*}

{\bf Step 2}. Upper bound for $II$.
	
	Using of the definition of $ \td J$, we have
\begin{align*}
II=&\int_0^{h/2}  \lf[\int_{B_{2h}\backslash B_h} J(|x_r^1-y|) \rd y\rr] \rd r, 
\end{align*}
where 
\begin{align*}
&x_r^1:=(r,0,\cdots,0)\in\R^N.
\end{align*}
 Set
\begin{align*}
\Omega:=\lf\{z=(z_1,z_2,\cdots,z_N): z_1\in [0,h/2],  |z_i|\leq \Lambda h,\ 2\leq i\leq N\rr\}
\end{align*}
with $\Lambda:=\frac{1}{2}\sqrt{\frac 3{N-1}}$.  Then $\Omega\subset B_h$
 because $(h/2)^2+(N-1)\Lambda ^2 h^2< h^2$. Define 
\begin{align*}
&\Omega^{(1)}:=\lf\{z=(z_1,z_2,\cdots,z_N): z_1\in (-\yy,0)\cup (h/2,\yy)\ {\rm and}\ z_i\in \R\ {\rm for}\ 2\leq i\leq N\rr\},\\
& \Omega^{(j)}:=\lf\{z=(z_1,z_2,\cdots,z_N): |z_j|>\Lambda h\ {\rm and}\ z_i\in \R\ {\rm for}\ i\neq j\rr\},\ \ 2\leq j\leq N.
\end{align*}
 Then
 \begin{align*}
B_{2h}\backslash B_h\subset \R^N\backslash B_h\subset \R^N\backslash \Omega\subset \cup_{j=1}^{N} \Omega_2^{(j)}.
 \end{align*}
and so
\begin{align*}
II
\leq \sum_{j=1}^{N} \int_0^{h/2}  \lf[\int_{\Omega^{(j)}} J(|x_r^1-y|) \rd y\rr] \rd r.
\end{align*}
By the definition of $J_*$, we deduce
\begin{align*}
&\int_0^{h/2}  \int_{\Omega^{(1)}} J(|x_r^1-y|) \rd y \rd r\\
=&\int_{0}^{h/2} \int_{h/2}^{\yy}  J_*(r-\rho) \rd \rho \rd r+\int_{0}^{h/2} \int_{-\yy}^{0} J_*(r-\rho) \rd \rho \rd r\\
=&\ 2\int_0^{h/2} \int_{h/2}^{\yy}  J_*(r-\rho) \rd \rho \rd r,
\end{align*}
and for $2\leq j\leq N$,
\begin{align*}
\int_0^{h/2}  \int_{\Omega^{(j)}} J(|x_r^1-y|) \rd y \rd r=h \int_{\Lambda h}^{\yy} J_*(r) \rd r.
\end{align*}
Hence,
\begin{align*}
II\leq 2\int^{h/2}_{0} \int_{h/2}^{\yy}  J_*(r-\rho) \rd \rho \rd r+(N-1)h \int_{\Lambda h}^{\yy} J_*(r) \rd r.
\end{align*}
Making use of  \eqref{5.2} and $J_*(r)\approx r^{N-\beta-1}$, we deduce for $\beta\in (N,N+1)$,
\begin{align*}
 2\int_0^{h/2} \int_{h/2}^{\yy}  J_*(r-\rho) \rd \rho \rd r+(N-1)h \int_{\Lambda h}^{\yy} J_*(r) \rd r
\approx   h^{N+1-\beta},
\end{align*} 
and for $\beta=N+1$,
\begin{align*}
 2\int_0^{h/2} \int_{h/2}^{\yy}  J_*(r-\rho) \rd \rho \rd r+(N-1)h \int_{\Lambda h}^{\yy} J_*(r) \rd r
\approx 
 \ln h.
 \end{align*}
 Therefore there exists $\td C_4>0$ such that for all large $h>0$,
 \[
 II\leq \begin{cases} \td C_4 h^{N+1-\beta} & \mbox{ when } \beta\in (N, N+1),\\
 \td C_4 \ln h & \mbox{ when } \beta =N+1.
 \end{cases}
 \]

{\bf Step 3}. Upper bound of $III$. 

By  similar analysis as in Step 2, we have, for $\beta\in (N,N+1]$ and $h\gg 1$,
\begin{align*}
III\leq &\ 2\int^{h}_{0} \int_{h}^{\yy}  J_*(r-\rho) \rd \rho \rd r+2(N-1)h \int_{2 \Lambda h}^{\yy} J_*(r) \rd r\\
\leq& \begin{cases} \td C_5 h^{N+1-\beta} & \mbox{ when } \beta\in (N, N+1),\\
  \td C_5 \ln h & \mbox{ when } \beta=N+1,
  \end{cases}
  \end{align*}
 for some $\td C_5>0$ independent of $h$.
 
 {\bf Step 4}. Completion of the proof of \eqref{5.3b}. 

Combining the above estimates for $I, II$ and $III$, we obtain 
\[
\int_{0}^{h}\int_{h}^{\yy}\td J(r,\rho) \rd \rho\rd r\leq \begin{cases} \hat C h^{N+1-\beta} & \mbox{ when } \beta\in (N, N+1),\\
  \hat C \ln h & \mbox{ when } \beta=N+1,
  \end{cases}
\]
 for $h\gg 1$ and some $\hat C>0$ independent of $h$.

	To complete the proof of \eqref{5.3b}, it remains to obtain a similar  lower bound for $\dd\int_{0}^{h}\int_{h}^{\yy}\td J(r,\rho) \rd \rho\rd r$. In view  of the definition of $ \td J$ and $J_*$, we have
\begin{equation}\label{5.6a}
\begin{aligned}
\int_{0}^{h}\int_{h}^{\yy}\td J(r,\rho) \rd \rho\rd r =&\int_0^h   \lf[\int_{\R^N\backslash B_h} J(|x_r^1-y|) \rd y\rr] \rd r \\
\geq &\int_0^h  \lf[\int_{\{y_1\geq h\}\cap\R^N} J(|x_r^1-y|) \rd y\rr] \rd r=\int_{0}^{h} \int_{h}^{\yy}  J_*(r-\rho) \rd \rho \rd r,
\end{aligned}
\end{equation}

and the desired lower bound follows from \eqref{5.2}. The proof is complete.
\end{proof}

\subsection{Proof of Theorem \ref{th1.8}}

\begin{lemma}\label{lemma5.2b}
	If spreading happens and \eqref{5.1} holds with $\beta\in (N,N+1]$, then there exits $C=C(\beta,N)>0$ such that for $t\gg 1$,
	\begin{equation}\label{5.3a}
	h(t)\leq \begin{cases}
	 Ct^{1/(\beta-N)}&{\rm if}\ \beta\in (N,N+1),\\
	 Ct\ln t&{\rm if}\ \beta=N+1.
	\end{cases}
	\end{equation}
\end{lemma}
\begin{proof}
Clearly, $\td u(t,r)\leq 2u^*$ for large $t>0$. By  \eqref{1.4},
\begin{align*}
h'(t)=\frac{\mu}{h^{N-1}(t)} \dd\int_{0}^{h(t)} r^{N-1}\td u(t,r)\int_{h(t)}^{+\yy} \td J(r,\rho)\rd \rho\rd r
\leq 2u^* \mu\dd\int_{0}^{h(t)}\int_{h(t)}^{+\yy} \td J(r,\rho)\rd \rho\rd r,
\end{align*}
and so \eqref{5.3a} follows easily from Lemma \ref{lemma5.1a}. 
\end{proof}

\begin{lemma}{\rm \cite[Lemma 5.3]{dn-speed}}\label{lemma5.3a}
	Let $L_1$ and $L_2$ with $0<L_1<L_2$ be two constants, and define
	\begin{align*}
	\psi(x)=\psi(x;L_1,L_2):=\min\lf\{1,\frac{L_2-|x|}{L_1}\rr\},\ \ \ \ x\in \R.
	\end{align*}
	If $J_1\in C(\R)\cap L^\infty(\R)$  satisfies
	\begin{align}\label{5.4a}
	J_1(x)=J_1(-x)\geq 0 ,\ J_1(0)>0,\  \int_{\R}J_1(x)\rd x=1,
	\end{align}
then for any given small $\epsilon>0$,  there exists $L_\epsilon\gg 1$  depending on $J_1$ and $\epsilon$ such that 
	\begin{align*}
	\int_{0}^{L_2} J_1(x-y)\psi(y)\rd y\geq (1-\epsilon) \psi(x) \mbox{ for } x\in  [L_\epsilon,L_2],
	\end{align*}
	provided that  $L_1\geq L_\epsilon$ and $L_2-L_1\geq L_\epsilon$.

\end{lemma}

\begin{lemma}\label{lemma5.4}
	Let $L_1$, $L_2$ and $\psi$ be defined as in Lemma \ref{lemma5.3a}.
	If\, {\bf (J1)} holds, 	then for any small $\epsilon>0$, there are $L_\epsilon>0$, $D_1=D_1(\epsilon)>0$ and $D_2=D_2(\epsilon)>0$ such that for  $L_1>L_\epsilon$ and $L_2-L_1>2L_\epsilon$, we have
	\begin{equation}\label{5.8b}
\int_{0}^{L_2} \td J(r,\rho)\psi(\rho)\rd \rho\geq
\begin{cases}
	 (1-\epsilon) \psi(r),& r\in [0,(L_2-L_1)/2],\\[3mm]
 (1-\epsilon) \psi(r)-\frac{D_1}{L_1r}-\frac{D_2}{r^2},& r\in [(L_2-L_1)/2,L_2].
\end{cases}
	\end{equation}
\end{lemma}

\begin{proof}
For fixed  $K_*>1$, 	define 
\begin{equation*}
	\xi_{K_*}(x):=\begin{cases}
	1,& |x|\leq K_*-1,\\
	K_*-|x|, & K_*-1\leq |x|\leq  K_*,\\
	0,& |x|\geq   K_*,
	\end{cases}
	\end{equation*}
	and
	\begin{align*}
	P(x):=\xi_{K_*}(|x|)J(|x|),\ \  P_1(x):=\frac{P(x)}{\|P\|_{L^1}}\  \mbox{ for } \ \ x\in \R^N.
	\end{align*}
Since $\|J\|_{L^1}=1$, for given small $\epsilon>0$, we can fix $K_*$ large so that
	\[
	\|P\|_{L^1}=\int_{\R^N}P(|x|) \rd x\geq \frac{1-\epsilon}{1-\epsilon/2}.
	\]
		Clearly $P$ and $P_1$ are compactly supported. 	Define
	\begin{align*}
	P_*(\rho):=\omega_{N-1}\int_{|\rho|}^{\yy}   (\eta^2-\rho^2)^{(N-3)/2} \eta P_1(\eta)\rd \eta, \ \  \rho\in \R.
	\end{align*}
	Then by  \eqref{1.10}, we know that \eqref{5.4a} holds with $J_1=P_*$. It then follows from  Lemma \ref{lemma5.3a}  that there exists  $L_\epsilon\gg 1$ such that for all $L_1\geq L_\epsilon$ and $L_2-L_1\geq   2 L_\epsilon$,
\begin{align}\label{5.9}
\int_{0}^{L_2} P_*(r-\rho)\psi(\rho)\rd \rho\geq (1-\epsilon/2) \psi(r),\ \ \ r\in [L_\epsilon,L_2].
\end{align}
 
Denote $A:=(L_2-L_1)/2$.   We now prove \eqref{5.8b} for $r\in [0,A]$. Using the definition of $\td J$, we see that for $r\in [0,A]$,
\begin{align*}
&\int_{0}^{L_2} \td J(r,\rho)\psi(\rho)\rd \rho\geq \int_{0}^{L_2-L_1} \td J(r,\rho)\psi(\rho)\rd \rho=\int_{0}^{L_2-L_1} \td J(r,\rho)\rd \rho\\
=&\int_{B_{L_2-L_1}(x_r^1)}  J(|x_r^1-y|)\rd y\geq  \int_{B_{L_2-L_1-A}(0)}  J(|y|)\rd y\geq 1-\epsilon=(1-\epsilon)\psi(r),
\end{align*}
where $x_r^1:=(r,0\cdots,0)$, $B_{L_2-L_1}(x_r^1)=\{y\in \R^N: |x_r^1-y|\leq L_2-L_1\}$, and we have used $L_2-L_1-A=( L_2-L_1)/2\gg 1$.

We next show \eqref{5.8b} for $r\in [A,L_2]$. Define
\begin{align*}
\wtd P(r,\rho):=\omega_{N-1}2^{3-N}\frac\rho{r^{N-2}}\int_{|\rho-r|}^{\rho+r}   \lf(\big[(\rho+r)^2-\eta^2\big]\big[\eta^2-(\rho-r)^2\big]\right)^{(N-3)/2}\eta P_1(\eta) \rd \eta.
\end{align*}
 For $r\in [A,L_2]$, since $ A\gg 1$, by Lemma \ref{lemma3.11},   there exist $B_1>0$ and $B_2>0$ such that
\begin{align*}
\int_{0}^{L_2}  \wtd P(r,\rho) \psi(\rho)\rd \rho\geq& \int_{0}^{ L_2}  P_*(r-\rho) \psi(\rho)\rd \rho+\frac{B_1} r\int_{-K_*}^{\min\{L_2-r, K_*\}} 
\rho P_*(\rho) \psi(\rho+l)\rd \rho\nonumber -\frac{B_2}{r^2}\\
\geq& \int_{0}^{ L_2}  P_*(r-\rho) \psi(\rho)\rd \rho+\frac{B_1} r \int_{-K_*}^{ K_*} 
\rho P_*(\rho) \psi(\rho+r)\rd \rho\nonumber -\frac{B_2}{r^2},
\end{align*}
because $\psi(\rho+r)\leq 0$ for $\rho+r\geq L_2$. This combined with \eqref{5.9} yields, for $r\in [A,L_2]$,
\begin{align*}
\int_{0}^{L_2}  \wtd P(r,\rho) \psi(\rho)\rd \rho\geq (1-\epsilon/2) \psi(r)+\frac{B_1}r \int_{-K_*}^{ K_*} 
\rho P_*(\rho) \psi(\rho+r)\rd \rho -\frac{B_2}{r^2}.
\end{align*}
Recalling that $\wtd P(r,\rho)\leq \frac{1}{\|P\|_{L^1}}\td J(r,\rho)$ and $\|P\|_{L^1}\geq (1-\epsilon)/(1-\epsilon/2)$, we deduce,  for $ r\in [A, L_2]$,
\begin{align*}
\int_{0}^{L_2}  \wtd J(r,\rho) \psi(\rho)\rd \rho\geq (1-\epsilon) \psi(r)+\frac{\|P\|_{L^1}B_1} r \int_{-K_*}^{ K_*} 
\rho P_*(\rho) \psi(\rho+r)\rd \rho -\frac{\|P\|_{L^1}B_2}{r^2}.
\end{align*}

To completes the proof of \eqref{5.8b}, it remains to estimate $\int_{-K_*}^{ K_*} 
\rho P_*(\rho) \psi(\rho+r)\rd \rho$.   Obviously, for $r\in [0,(L_2-L_1)-K_*]$, 
\begin{align*}
\int_{-K_*}^{ K_*} 
\rho
P_*(\rho) \psi(\rho+r)\rd \rho=\int_{-K_*}^{ K_*} 
\rho
P_*(\rho) \rd \rho=0,
\end{align*}
and  for $r\in [(L_2-L_1)-K_*,(L_2-L_1)+K_*]$,
\begin{align*}
\int_{-K_*}^{ K_*} 
\rho
P_*(\rho) \psi(\rho+r)\rd \rho=& \int_{-K_*}^{ (L_2-L_1)-r} 
\rho
P_*(\rho) \psi(\rho+r)\rd \rho+\int_{ (L_2-L_1)-r}^{K_*}
\rho
P_*(\rho) \psi(\rho+r)\rd \rho\\
=&\int_{-K_*}^{ (L_2-L_1)-r} 
\rho
P_*(\rho)\rd \rho+\int_{ (L_2-L_1)-r}^{K_*}
\rho
P_*(\rho) \frac{L_2-(\rho+r)}{L_1}\rd \rho\\
=&\int_{ (L_2-L_1)-r}^{K_*}
\rho
P_*(\rho) \lf[\frac{L_2-(\rho+r)}{L_1}-1\rr]\rd \rho\\
=&\int_{ (L_2-L_1)-r}^{K_*}
\rho
P_*(\rho) \frac{L_2-L_1-r}{L_1}\rd \rho+\int_{ (L_2-L_1)-r}^{K_*}
\rho
P_*(\rho) \frac{-\rho}{L_1}\rd \rho\\
\geq&\frac{-K_*|L_2-L_1-r|}{L_1}\int_{ (L_2-L_1)-r}^{K_*}  P_*(\rho)
 \rd \rho-\frac{K_*^2}{L_1}\int_{ (L_2-L_1)-r}^{K_*}
 P_*(\rho) \rd \rho\\
 \geq&\frac{-2K_*^2}{L_1},
\end{align*}
and for $r\in [(L_2-L_1)+K_*, L_2]$
\begin{align*}
\int_{-K_*}^{ K_*} 
\rho
P_*(\rho) \psi(\rho+r)\rd \rho=&\int_{-K_*}^{ K_*} 
\rho
P_*(\rho) \frac{L_2-(\rho+r)}{L_1}\rd \rho\\
=&\int_{-K_*}^{ K_*} 
\rho
P_*(\rho) \frac{-\rho}{L_1}\rd \rho\geq \frac{-K_*^2}{L_1}.
\end{align*}
Thus \eqref{5.8b} holds with $D_1=2K_*^2\|P_1\|_{L^1}B_1$ and $D_2=\|P_1\|_{L^1}B_2$. 
\end{proof}

\begin{lemma}\label{lemma5.5}
	If spreading happens and \eqref{5.1} holds with $\beta\in (N,N+1)$, then there exists $C>0$ such that 
	\begin{align}\label{5.12a}
	h(t)\geq Ct^{1/(\beta-N)}\ \mbox{ for } t\gg 1.
	\end{align}
\end{lemma}
\begin{proof}
	Define 
	\begin{align*}
	&\underline h(t):=(K_1t+\theta)^{1/(\beta-N)},\ \ t\geq 0,\\
	&	\underline u(t,r):=K_2\min\lf\{1,2\frac{\underline h(t)-r}{\underline h(t)}\rr\}, \ \ \ t\geq 0,\ r\in [0, \underline h(t)]
	\end{align*}
	with  $\theta\gg 1$ and $K_1,\ K_2>0$ constants to be determined.

	{\bf Step 1.}
We  prove that by choosing $\theta\gg1$ and $K_1>0$ suitably small,
	\begin{align}\label{5.8a}
	&\underline h'(t)\leq  \frac{\mu}{\underline h^{N-1}(t)}\dd\int_{0}^{\underline h(t)} r^{N-1}\underline u(t,r)\int_{\underline h(t)}^{+\yy} \td J(r,\rho)\rd \rho\rd r \ \mbox{ for }\  t> 0.
	\end{align}
	
By similar argument leading to \eqref{5.6a}, and by \eqref{7.3}, we obtain
	\begin{align*}
	 &\frac{\mu}{\underline h^{N-1}(t)}\dd\int_{0}^{\underline h(t)} r^{N-1}\underline u(t,r)\int_{\underline h(t)}^{+\yy} \td J(r,\rho)\rd \rho\rd r
	 \geq \frac{\mu}{ \underline h^{N-1}(t)}\dd\int_{\underline h(t)/2}^{\underline h(t)} r^{N-1}\underline u(t,r)\int_{\underline h(t)}^{+\yy} J_*(r-\rho)\rd \rho\rd r\\
	 \geq&\ 2^{-(N-1)}\mu\dd\int_{\underline h(t)/2}^{\underline h(t)} \frac{\underline h(t)-r}{\underline h(t)}\int_{\underline h(t)}^{+\yy} J_*(r-\rho)\rd \rho\rd r
	 =\frac{2^{-(N-1)}}{\underline h(t)}\mu\dd\int_{-\underline h(t)/2}^{0} (-r)\int_{0}^{+\yy} J_*(r-\rho)\rd \rho\rd r\\
	 =&\ \frac{2^{-(N-1)}}{\underline h(t)}\mu\dd\int_{-\underline h(t)/2}^{0} (-r)\int_{-r}^{+\yy} J_*(\rho)\rd \rho\rd r
	 =\frac{2^{-(N-1)}}{\underline h(t)}\mu\dd\lf(\int_0^{\underline h(t)/2} \int_{0}^{\rho}+\int_{\underline h(t)/2}^{\yy} \int_{0}^{\underline h(t)/2}\rr) J_*(\rho)r \rd r\rd \rho\\
	 \geq&\ \frac{2^{-(N-1)}}{\underline h(t)}\mu\dd\int_0^{\underline h(t)/2} \int_{0}^{\rho} J_*(\rho)r \rd r\rd \rho\geq  \frac{2^{-(N-1)}}{\underline h(t)}\mu\dd\int_1^{\underline h(t)/2} J_*(\rho)\rho^2\rd \rho\geq \mu A_1 \underline h^{N+1-\beta}(t)
	\end{align*}
	for some $A_1=A_1(N,\beta)>0$ since $\underline h(t)\geq \theta^{1/(\beta-N)}\gg 1$. Hence
	\begin{align*}
	\underline h'(t)&= \frac{K_1}{\beta-N}(K_1t+\theta)^{1/(\beta-N)-1}\leq \mu A_1  (K_1t+\theta)^{(1/(\beta-N)-1}= \mu A_1 \underline h^{N+1-\beta}(t)\\
	&\leq \frac{\mu}{\underline h^{N-1}(t)}\dd\int_{0}^{\underline h(t)} r^{N-1}\underline u(t,r)\int_{\underline h(t)}^{+\yy} \td J(r,\rho)\rd \rho\rd r\ \ \mbox{ provided }
	\frac{K_1}{\beta-N}\leq \mu A_1.
	\end{align*}

	{\bf Step 2.} 
We show  that	by choosing $\theta\gg1$ and $K_1, K_2>0$ suitably small,  for $t>0$, $r\in [0,\underline h(t))\backslash \{\underline h(t)/2\}$, 
	\begin{align}\label{5.14a}
	\underline u_t(t,r)\leq d \int_{0}^{\underline h(t)}  \td{J}(r, \rho) \underline u(t,\rho)\rd \rho -d\,\underline u(t,r)+f(\underline u(t,r)).
	\end{align}

	{\bf Claim}. For $r\in [\underline h(t)/4,\underline h(t)]$, there is $\xi=\xi(N,\beta)>0$ such that
	\begin{align}\label{5.15a}
	\int_{0}^{\underline h(t)}  \td J(r,\rho) \underline u(t,\rho)\rd \rho\geq  K_2\xi
	\underline h^{N-\beta}(t).
	\end{align}
	
	We first estimate $\td J(r,\rho)$. For $r\in [\underline h(t)/4,\underline h(t)]$ and $\rho\in [\underline h(t)/8,\underline h(t)]$, a simple calculation gives
	\begin{align*}
	|\rho-r|\leq \frac{7}{9}(\rho +r).
	\end{align*}
Fix $1<\xi_1<\xi_2< 9/7$. Then from Lemma \ref{lemma2.2} and \eqref{5.1}, we deduce for $r\in [\underline h(t)/4,\underline h(t)]$, $\rho\in [\underline h(t)/8,\underline h(t)]$ and $|\rho-r|\geq 1$,
		\begin{align*}
	\td J(r,\rho)&= \omega_{N-1}\int_{|\rho-r|}^{\rho+r}   (\rho/l)^{(N-1)/2}\lf[\frac{(\rho+r)^2-\eta^2}{4r\rho }\rr]^{(N-3)/2}[\eta^2-(\rho-r)^2]^{(N-3)/2}\eta J(\eta) \rd \eta \\
	&\geq  \xi_3\int_{\xi_1|\rho-r|}^{\xi_2|\rho-r|}   \lf[\frac{(\rho+r)^2-\eta^2}{4r\rho }\rr]^{(N-3)/2}[\eta^2-(\rho-r)^2]^{(N-3)/2}\eta^{1-\beta}\rd \eta \ \ \mbox{for some $\xi_3>0$.}
	\end{align*}

	A simple calculation gives, for $r\in [\underline h(t)/4,\underline h(t)]$, $\rho\in [\underline h(t)/8,\underline h(t)]$, $\eta\in [\xi_1|\rho-r|, \xi_2|\rho-r|]$ and $N\geq 3$,
	\begin{align*}
	&\lf[\frac{(\rho+r)^2-\eta ^2}{4r\rho }\rr]^{(N-3)/2}\geq \lf[\frac{(\rho+r)^2-\xi_2^2(\rho-r)^2}{4r\rho }\rr]^{(N-3)/2}\\
	\geq& \lf[\frac{(\rho+r)^2(1-49\xi_2^2/81)}{4r\rho }\rr]^{(N-3)/2}\geq (1-49\xi_2^2/81)^{(N-3)/2},
	\end{align*}
	and for $N=2$ with such $\rho, r, \eta$,
	\begin{align*}
	\lf[\frac{(\rho+r)^2-\eta^2}{4r\rho }\rr]^{-1/2}\geq \lf[\frac{(\rho+r)^2}{4r\rho }\rr]^{-1/2}\geq  32^{-1/2},
	\end{align*} 
	while for such $\rho, r, \eta$ and $N\geq 2$,
	\begin{align*}
	[\eta^2-(\rho-r)^2]^{(N-3)/2}\geq \min\lf\{(\xi_1^2-1)^{(N-3)/2},(\xi_2^2-1)^{(N-3)/2}\rr\} |\rho-r|^{N-3}.
	\end{align*}
	Hence,
	\begin{equation}\begin{aligned}\label{5.16a}
	\td J(r,\rho)
	\geq &\  \xi_3\int_{\xi_1|\rho-r|}^{\xi_2|\rho-r|}   \lf[\frac{(\rho+r)^2-\eta^2}{4r\rho }\rr]^{(N-3)/2}[\eta^2-(\rho-r)^2]^{(N-3)/2}\eta^{1-\beta}\rd \eta\\
	\geq&\ \xi_4|\rho-r|^{N-1-\beta}\ \mbox{ for some $\xi_4>0$.}
	\end{aligned}
	\end{equation}

	Now using \eqref{5.16a} and
\begin{align*}
\underline u(t,\rho)\geq K_2\frac{\underline h(t)-\rho}{\underline h(t)},
\end{align*}
 we obtain, for $r\in [h(t)/4,h(t)]$,
	\begin{align*}
	&\int_{0}^{\underline h(t)}  \td {J}(r,\rho) \underline u(t,\rho)\rd \rho
	\geq \int_{\underline h(t)/8}^{r-1}  \td {J}(r,\rho) \underline u(t,\rho)\rd \rho
	\geq \int_{\underline h(t)/8}^{r-1}  \xi_4|\rho-r|^{N-1-\beta} \underline u(t,\rho)\rd \rho\\
	=&\ \int_{\underline h(t)/8-r}^{-1}  \xi_4|\rho|^{N-1-\beta}  \underline u(t,\rho+r)\rd \rho
	\geq K_2\xi_4\int_{-\underline h(t)/8}^{-1}  |\rho|^{N-1-\beta} \frac{\underline h(t)-(\rho+r)}{\underline h(t)}\rd \rho\\
	\geq&\  \frac{K_2\xi_4}{\underline h(t)}\int_{-\underline h(t)/8}^{-1}  |\rho|^{N-\beta}\rd \rho= \frac{K_2\xi_4}{\underline h(t)}\int_1^{\underline h(t)/8} \rho^{N-\beta}\rd \rho\\
	=&\ \frac{K_2\xi_4}{(N+1-\beta)\underline h(t)}
	[(\underline h(t)/8)^{N+1-\beta}-1] \geq  \frac{K_2\xi_4}{2(N+1-\beta)\underline h(t)}
	(\underline h/8)^{N+1-\beta},
	\end{align*}
	which gives \eqref{5.15a}.

With the above estimates at hand, we are ready to prove \eqref{5.14a}. 	Due to $\underline u\leq K_2$ and $f'(0)>0$, we see that for small $K_2>0$, 
	\begin{align*}
	f(\underline u(t,r))\geq \min_{v\in [0,K_2]}f'(v)\underline  u(t,r)\geq \frac{3}4 f'(0) \underline  u(t,r).
	\end{align*}
Let $\epsilon\in (0, \frac 3{4d}f'(0))$ be a small constant. From Lemma \ref{lemma5.4}, for $L_1=\underline h(t)/2$ and $L_2=\underline h(t)$ with $\theta\gg 1$, 
	\begin{equation*}
	\int_{0}^{L_2} \td J(r,\rho)\underline u(t,\rho)\rd \rho\geq
\begin{cases}
\dd	 (1-\epsilon) \underline u(t,r),& r\in [0,\underline h(t)/4],\\[3mm]
\dd		 (1-\epsilon)\underline u(t,r)-\frac{2D_1}{\underline h(t)r}-\frac{D_2}{r^2},& r\in [\underline h(t)/4,\underline h(t)].
\end{cases}
\end{equation*}
Hence for $r\in [0,\underline h(t)/4]$,
\begin{align*}
d \int_{0}^{\underline h(t)} \td {J}(r,\rho) \underline u(t,\rho)\rd \rho -d\underline u(t,r)+f(\underline u(t,r))\geq -d\epsilon \underline u(t,r)+\frac{3}4f'(0) \underline  u(t,r)\geq 0,
\end{align*}
and for $r\in [\underline h(t)/4,\underline h(t)]$, from \eqref{5.15a},
\begin{align*}
&d \int_{0}^{\underline h(t)} \td {J}(r,\rho) \underline u(t,\rho)\rd \rho -d\underline u(t,r)+f(\underline u(t,r))\\
\geq&\  d \int_{0}^{\underline h(t)} \td {J}(r,\rho) \underline u(t,\rho)\rd \rho-[d-{3f'(0)}/{4}]\underline u(t,r)\\
=&\ \min\{{f'(0)}/{2},d \}\int_{0}^{\underline h(t)} \td {J}(r,\rho) \underline u(t,\rho)\rd \rho+\max\{d-{f'(0)}/{2},0\} \int_{0}^{\underline h(t)} \td {J}(r,\rho) \underline u(t,\rho)\rd \rho\\
&-[d-{3f'(0)}/{4}]\underline u(t,r)\\
\geq &\ \min\{{f'(0)}/{2},d \}K_2\xi \underline h^{N-\beta}(t) +\max\{d-{f'(0)}/{2},0\} \lf[(1-\epsilon)\underline u(t,r)-\frac{2D_1}{\underline h(t) r}-\frac{D_2}{r^2}\rr]\\
&-[d-{3f'(0)}/{4}]\underline u(t,r)\\
\geq &\ \min\{{f'(0)}/{2},d \}K_2\xi \underline h^{N-\beta}(t)
-\max\{d-{f'(0)}/{2},0\} \lf(\frac{2D_1}{\underline h(t) r}+\frac{D_2}{r^2}\rr)\\
\geq&\ \min\{{f'(0)}/{2},d \}K_2\xi\underline h^{N-\beta}(t)-d\frac{8D_1+16D_2}{\underline h^2(t)}\geq \frac 12\min\{{f'(0)}/{2},d \}K_2\xi\underline h^{N-\beta}(t)
\end{align*}
since  $\underline h(t)\geq \theta^{1/(\beta-N)}\gg 1$ (due to $\theta\gg 1$).  	

In view of  the definition of $\underline u$, we have for $t>0$, $r\in (0,\underline h(t))$, 
\begin{align*}
&\underline u_t(t,r)=0, \ \ \ t>0,\  r\in (0,\underline h(t)/2),\\
&\underline u_t(t,r)=2K_2\frac{r\underline h'(t)}{\underline h^2(t)}\leq 2K_2 \frac{\underline h'(t)}{\underline h(t)}=   \frac{2K_1K_2}{\beta-N}\underline h^{N-\beta}(t),\ \ t>0,\  r\in (\underline h(t)/2,\underline h(t)).
\end{align*}
Therefore, \eqref{5.14a} holds if $K_1$ is small such that 
	\begin{align*}
\frac{2K_1}{\beta-N}\leq 	\frac{\min\{{f'(0)}/{2},d \}\xi}{2}.
	\end{align*}

{\bf Step 3.} Completion of the proof.
	
	Let $\theta$, $K_1$ and $K_2$ meet the above requirements.  	It is clear that 
	\begin{align*}
	\underline u(t,\pm \underline h(t))=0\ \mbox{ for } \ t\geq 0. 
	\end{align*}
	Since spreading happens,  there exists a large  $t_0>0$ such that 
	\begin{align*}
	&[0,\underline h(0)]\subset [0,h(t_0)/2]\ \ {\rm and}\ \ u(0,r)\leq K_2\leq  u(t_0,r)\ \mbox{ for } \  r\in [0,\underline h(0)]. 
	\end{align*}
	It then follows from Lemma \ref{lemma3.4a} (and Remark \ref{rmk3.6}) that
	\begin{align*}
	&\underline h(t)\leq h(t+t_0), \ \ \ \ \ \ \ \ \ t\geq 0,\\
	&\underline u(t,r)\leq  u(t+t_0,r), \ \ \ \  t\geq 0,\ r\in [0,\underline h(t)],
	\end{align*}
	which implies \eqref{5.12a}. The proof is complete. 
\end{proof}

\begin{lemma}\label{lemma4.5}
	If spreading happens and \eqref{5.1} holds with $\beta=N+1$, then there exists $C>0$ such that 
	\begin{align}\label{5.17a}
	h(t)\geq Ct\ln t \ \ \mbox{\rm for } t\gg 1.
	\end{align}
\end{lemma}
\begin{proof}
Fix $\alpha\in (0,1)$ and define 
	\begin{align*}
	&\underline h(t):=K_1(t+\theta)\ln (t+\theta),\ \ t\geq 0,\\
	&	\underline u(t,r):=K_2\min\lf\{1, \frac{\underline h(t)-r}{(t+\theta)^{\alpha}}\rr\}, \ \ \ t\geq 0,\ r\in[0, \underline h(t)],
	\end{align*}
	where $\theta\gg 1$ and $K_1, K_2>0$ are constants  to be determined. Note that $\theta\gg 1$ implies  
	\[
	(t+\theta)^\alpha\leq \frac{\underline h(t)}{K_1\theta^{1-\alpha}\ln \theta}=o(1)\underline h(t)\ \mbox{ uniformly  for } t\geq 0.
	\]

We first show  that by choosing $\theta\gg 1$ and $K_1, K_2>0$ suitably,
\begin{align}\label{5.18}
&\underline h'(t)\leq  \frac{\mu}{h^{N-1}(t)}\dd\int_{0}^{\underline  h(t)} r^{N-1}\underline u(t,r)\int_{\underline h(t)}^{+\yy} \td J(r,\rho)\rd \rho\rd r\ \mbox{ for } \   t> 0.
\end{align}
	
	By  \eqref{5.6a},
	\begin{align*}
	 &\frac{\mu}{\underline h^{N-1}(t)}\dd\int_{0}^{\underline h(t)} r^{N-1}\underline u(t,r)\int_{\underline h(t)}^{+\yy} \td J(r,\rho)\rd \rho\rd r\\
	 &\geq \frac{\mu}{\underline h^{N-1}(t)}\dd\int_{\underline h(t)/2}^{\underline h(t)-(t+\theta)^\alpha} r^{N-1}\underline u(t,r)\int_{\underline h(t)}^{+\yy} \td J(r,\rho)\rd \rho\rd r
	 \\
	 &\geq 2^{-(N-1)}K_2\mu\dd\int_{\underline h(t)/2}^{\underline h(t)-(t+\theta)^\alpha} \int_{\underline h(t)}^{+\yy} \td J(r,\rho)\rd \rho\rd r\\
	& \geq 2^{-(N-1)}K_2\mu\dd\int_{\underline h(t)/2}^{\underline h(t)-(t+\theta)^\alpha} \int_{\underline h(t)}^{+\yy}  J_*(r-\rho)\rd \rho\rd r
	 =2^{-(N-1)}K_2\mu\dd\int_{-\underline h(t)/2}^{-(t+\theta)^\alpha} \int_{-r}^{+\yy}  J_*(\rho)\rd \rho\rd r\\
	 &=2^{-(N-1)}K_2\mu\dd\lf(\int_{(t+\theta)^\alpha}^{\underline h(t)/2}\int_{(t+\theta)^\alpha}^\rho+\int_{\underline h(t)}^{\yy}\int_{(t+\theta)^\alpha}^{\underline h(t)}\rr)  J_*(\rho)\rd r \rd \rho\\
	 &\geq 2^{-(N-1)}K_2\mu\dd\int_{(t+\theta)^\alpha}^{\underline h(t)/2}\int_{(t+\theta)^\alpha}^\rho J_*(\rho)\rd r \rd \rho=2^{-(N-1)}K_2\mu\dd\int_{(t+\theta)^\alpha}^{\underline h(t)/2}[\rho-(t+\theta)^\alpha] J_*(\rho) \rd \rho\\
	 &\geq2^{-N}K_2\mu\dd\int_{2(t+\theta)^\alpha}^{\underline h(t)/2}\rho J_*(\rho) \rd \rho.
	\end{align*}
	By \eqref{7.3}  there is $A_1>0$ such that 
	\begin{align*}
	&\dd\int_{2(t+\theta)^\alpha}^{\underline h(t)/2}\rho J_*(\rho) \rd \rho\geq 	A_1\dd\int_{2(t+\theta)^\alpha}^{\underline h(t)/2} \rho^{-1} \rd \rho
	=A_1[\ln \underline h(t)-2\ln 2-\alpha\ln (t+\theta)]\\
	&=A_1[\ln (t+\theta)+\ln \ln (t+\theta)+\ln K_1-2\ln 2-\alpha\ln (t+\theta)]\\
	&\geq A_1(1-\alpha) \ln (t+\theta)\ \ \  \mbox{due to $\theta\gg 1$.}
	\end{align*}
	 Hence
	\begin{align*}
	&\frac{\mu}{\underline h^{N-1}(t)}\dd\int_{0}^{\underline h(t)} r^{N-1}\underline u(t,r)\int_{\underline h(t)}^{+\yy} \td J(r,\rho)\rd \rho\rd r\geq \ 2^{-N}K_2\mu A_1(1-\alpha) \ln (t+\theta)\\
	&\geq\  K_1\ln (t+\theta)+K_1=\underline h'(t)
	\end{align*}
	provided 
	\begin{align*}
	K_1\leq 2^{-N-1}K_2\mu A_1(1-\alpha).\ \ 
	\end{align*}
	This proves \eqref{5.18}.

Nest we show that by choosing $\theta\gg1,\; K_1, K_2>0$ suitably, 	 for $t>0$ and $r\in (-\underline h(t),\underline h(t))\backslash \{\underline h(t)-(t+\theta)^\alpha\}$, 
	\begin{align}\label{5.19a}
	\underline u_t(t,r)\leq d \int_{-\underline h(t)}^{\underline h(t)}  \td {J}(r,\rho) \underline u(t,\rho)\rd \rho -d\, \underline u(t,r)+f(\underline u(t,r)).
	\end{align}
	From the definition of $\underline u$, for $t>0$,
	\begin{equation}\label{7.18}
	\underline u_t(t,r)=\begin{cases}\frac{K_1K_2(1-\alpha)\ln (t+\theta)+K_1K_2}{(t+\theta)^{\alpha}} +\frac{K_2\alpha r}{(t+\theta)^{1+\alpha}}& \mbox{ when } r\in (\underline h(t)-(t+\theta)^\alpha, \underline h(t)],\\
	0\ &\mbox{ when } \ r\in [0, \underline h(t)-(t+\theta)^\alpha).
	\end{cases}
	\end{equation}

	 Using \eqref{5.16a} and 
\begin{align*}
\underline u(t,\rho)\geq K_2\frac{\underline h(t)-\rho}{2(t+\theta)^\alpha} \mbox{ for } \rho\in [\underline h(t)-2(t+\theta)^\alpha, \underline h(t)],
\end{align*}
we deduce, for $r\in [\underline h(t)-(t+\theta)^\alpha,h(t)]$,
\begin{align*}
&\int_{0}^{\underline h(t)}  \td {J}(r,\rho) \underline u(t,\rho)\rd \rho
\geq \int_{\underline h(t)-2(t+\theta)^\alpha}^{r-1}  \td {J}(r,\rho) \underline u(t,\rho)\rd \rho\\
\geq &\int_{\underline h(t)-2(t+\theta)^\alpha}^{r-1}  \xi|\rho-r|^{N-1-\beta} \underline u(t,\rho)\rd \rho
=\int_{\underline h-2(t+\theta)^\alpha-r}^{-1}  \xi|\rho|^{N-1-\beta}  \underline u(t,\rho+r)\rd \rho\\
\geq&\ K_2\xi/2\int_{-(t+\theta)^\alpha}^{-1}  |\rho|^{N-1-\beta}  \frac{\underline h(t)-(\rho+r)}{(t+\theta)^\alpha}\rd \rho
\geq K_2\xi/2\int_{-(t+\theta)^\alpha}^{-1}  |\rho|^{N-1-\beta}  \frac{-\rho}{(t+\theta)^\alpha}\rd \rho\\
=&\ \frac{K_2\xi}{2(t+\theta)^\alpha}\int_{1}^{(t+\theta)^\alpha} \rho^{N-\beta}  \rd \rho
= \frac{K_2\xi}{2(t+\theta)^\alpha}\int_1^{(t+\theta)^\alpha} \rho^{-1}  \rd \rho
=\frac{K_2\xi\alpha \ln (t+\theta)}{2(t+\theta)^\alpha},
\end{align*}
and for $r\in \lf[\frac{\underline h(t)-(t+\theta)^\alpha}{2},\underline h(t)-(t+\theta)^\alpha\rr]$,
\begin{align*}
&\int_{0}^{\underline h(t)}  \td {J}(r,\rho) \underline u(t,\rho)\rd \rho
\geq \int_{[\underline h(t)-(t+\theta)^\alpha]/4}^{r-1}  \td {J}(r,\rho) \underline u(t,\rho)\rd \rho\\
\geq&\ K_2\int_{[\underline h(t)-(t+\theta)^\alpha]/4}^{r-1}  \td {J}(r,\rho) \rd \rho
\geq K_2 \xi\int_{[\underline h(t)-(t+\theta)^\alpha]/4}^{r-1}  |\rho-r|^{N-1-\beta} \rd \rho\\
\geq&\ K_2 \xi\int_1^{[\underline h(t)-(t+\theta)^\alpha]/4}  \rho^{N-1-\beta}\rd \rho= K_2 \xi\int_1^{[\underline h(t)-(t+\theta)^\alpha]/4}  \rho^{-2} \rd \rho\\
=&\ K_2 \xi\bigg(1-4[\underline h(t)-(t+\theta)^\alpha]^{-1}\bigg)\geq K_2 \xi/2\geq \frac{K_2\xi\alpha \ln (t+\theta)}{2(t+\theta)^\alpha}\ \ \mbox{ since $\theta\gg 1$. }
\end{align*}

For small $\epsilon>0$, using Lemma \ref{lemma5.4} with $L_1=(t+\theta)^\alpha$, $L_2=\underline h(t)$ and $\theta\gg 1$, we have
	\begin{equation*}
	\int_{0}^{L_2} \td J(r,\rho)\underline u(t,\rho)\rd \rho\geq
\begin{cases}
\dd	 (1-\epsilon) \underline u(t,r),& r\in [0, \frac{\underline h(t)-(t+\theta)^\alpha}{2}],\\[3mm]
\dd		 (1-\epsilon)\underline u(t,r)-\frac{2D_1}{\underline h(t)r}-\frac{D_2}{r^2},& r\in [\frac{\underline h(t)-(t+\theta)^\alpha}{2},\underline h(t)].
\end{cases}
\end{equation*}

Hence, similar to the argument in the proof of Lemma \ref{lemma5.5},
 for $r\in [0,\frac{\underline h(t)-(t+\theta)^\alpha}{2}]$,
\begin{align*}
d \int_{0}^{\underline h(t)} \td {J}(r,\rho) \underline u(t,\rho)\rd \rho -d\underline u(t,r)+f(\underline u(t,r))\geq -d\epsilon \underline u(t,r)+\frac{3}4f'(0) \underline  u(t,r)\geq 0,
\end{align*}
and for $r\in [\frac{\underline h(t)-(t+\theta)^\alpha}{2} ,\underline h(t)]$, 
\begin{align*}
&d \int_{0}^{\underline h(t)} \td {J}(r,\rho) \underline u(t,\rho)\rd \rho -d\underline u(t,r)+f(\underline u(t,r))\\
\geq&\  d \int_{0}^{\underline h(t)} \td {J}(r,\rho) \underline u(t,\rho)\rd \rho-[d-{3f'(0)}/{4}]\underline u(t,r)\\
=&\ \min\{{f'(0)}/{2},d \}\int_{0}^{\underline h(t)} \td {J}(r,\rho) \underline u(t,\rho)\rd \rho+\max\{d-{f'(0)}/{2},0\} \int_{0}^{\underline h(t)} \td {J}(r,\rho) \underline u(t,\rho)\rd \rho\\
&-[d-{3f'(0)}/{4}]\underline u(t,r)\\
\geq &\ \min\{{f'(0)}/{2},d \}\frac{K_2\xi\alpha \ln (t+\theta)}{2(t+\theta)^\alpha} +\max\{d-{f'(0)}/{2},0\} \lf[(1-\epsilon)\underline u(t,r)-\frac{2D_1}{\underline h(t) r}-\frac{D_2}{r^2}\rr]\\
&-[d-{3f'(0)}/{4}]\underline u(t,r)\\
\geq &\ \min\{{f'(0)}/{2},d \}\frac{K_2\xi\alpha \ln (t+\theta)}{2(t+\theta)^\alpha}
-\max\{d-{f'(0)}/{2},0\} \lf(\frac{2D_1}{\underline h(t) r}+\frac{D_2}{r^2}\rr)\\
\geq&\ \min\{{f'(0)}/{2},d \}\frac{K_2\xi\alpha \ln (t+\theta)}{2(t+\theta)^\alpha}-d\frac{8D_1+16D_2}{\underline h^2(t)}\geq \frac 12 \min\{{f'(0)}/{2},d \}\frac{K_2\xi\alpha \ln (t+\theta)}{2(t+\theta)^\alpha}
\end{align*}
since	 $\theta\gg 1$. From this and \eqref{7.18}, we see that \eqref{5.19a} holds.

With the above inequalities \eqref{5.18} and \eqref{5.19a}, as in the proof of Lemma \ref{lemma5.5}, we can apply the comparison principle Lemma \ref{lemma3.4a} (and Remark \ref{rmk3.6}) to obtain \eqref{5.17a}. 
\end{proof}

Theorem \ref{th1.8} clearly follows directly from Lemmas \ref{lemma5.2b}, \ref{lemma5.5} and \ref{lemma4.5}.




\end{document}